\documentclass[11pt]{article}
\usepackage{macros}

\title{Noncommutative $L^p$-integrators and \\
stochastic differential equations in $L^p$}
\author[1]{David A.\ Jekel\thanks{Supported by NSF grant DMS-2002826 and the EU Horizon Marie Sk{\l}odowska Curie Action, FREEINFOGEOM, grant 101209517}}
\author[2]{Todd A.\ Kemp\thanks{Supported by NSF grants DMS-2400246, DMS-2055340, and DMS-1800733}}
\author[3]{Evangelos A.\ Nikitopoulos\thanks{Supported by NSF grant DGE-2038238 and partially supported by NSF grant DMS-2055340}}
\affil[1]{Department of Mathematics and Statistics, University of Ottawa\protect\\
150 Louis-Pasteur Private, Ottawa, ON K1N 6N5 (Canada)\protect\\
{Email: \tt \href{mailto:djekel@uottawa.ca}{djekel@uottawa.ca}}\vspace{2mm}}
\affil[2]{Department of Mathematics, University of California San Diego\protect\\
\noindent 9500 Gilman Drive, La Jolla, CA 92093-0112 (USA)\protect\\
{Email: \tt \href{mailto:tkemp@ucsd.edu}{tkemp@ucsd.edu}}\vspace{2mm}}
\affil[3]{Department of Mathematics, University of Michigan\protect\\
\noindent 530 Church Street, Ann Arbor, MI 48109-1043 (USA)\protect\\
{Email: {\tt \href{mailto:enikitop@umich.edu}{enikitop@umich.edu}}}}
\date{\vspace{-7ex}}

\begin{document}

\maketitle

\begin{abstract}
We propose a new framework for noncommutative stochastic integration in $L^p$, inspired by Bichteler's notion of $L^p$-integrators, that massively generalizes our previous theory of $L^2$-valued stochastic integration against $L^2$-decomposable processes.
Central to our framework is a noncommutative analog of a predictable integrand we call a predictable linear process.
We develop basic properties of stochastic integrals against ``noncommutative $L^p$-integrators'' and establish a useful criterion for a noncommutative stochastic process to be a noncommutative $L^p$-integrator.
This criterion applies, in particular, to a rich class of processes we call measured decomposable processes, which includes free Brownian motion and, more generally, the $q$-Brownian motions.
We also show that, when specialized appropriately, our framework recovers the classical notion of an $L^p$-integrator up to the fact that the noncommutative theory sees only the modification class of a classical stochastic process.

As an application, we use our theory to study noncommutative stochastic differential equations (SDEs) in $L^p$.
Under natural local-Lipschitz and boundedness assumptions, we establish the existence and uniqueness of maximal solutions and show that a maximal solution with finite lifetime must blow up in $L^p$ norm.
As corollaries, we obtain new existence and uniqueness results for free SDEs and SDEs driven by $q$-Brownian motion.
Notably, in the free case, our results apply to coefficients arising through the continuous functional calculus from merely locally Lipschitz scalar functions, as opposed to locally operator-Lipschitz functions.

\medskip

\noindent \textbf{Keyphrases:} free probability, noncommutative probability, noncommutative stochastic analysis, integrators, stochastic integrals, stochastic differential equations

\medskip

\end{abstract}

\tableofcontents

\section{Introduction}

Noncommutative probability theory abstracts the constructs of probability and measure, elevating random variables themselves to the fundamental objects so that the arena may be widened to include spaces of random variables that do not commute when multiplied.
Essentially born out of the mathematical physics community's efforts to make quantum field theory rigorous, this framework has been enormously productive to the theory of operator algebras.
In the last few decades, the rise in popularity of random matrix theory has provided a more accessible avenue to explore and understand the value of noncommutative probability.
A random $N\times N$ matrix can, of course, be treated as a random vector, but this ignores the algebraic structure of the state space of $N \times N$ matrices.
In particular, the most interesting questions about random matrices focus on their eigenvalues and eigenvectors, whose study demands viewing the state space as an algebra of linear operators.

Recently, there has been a good deal of progress on understanding the large-dimension asymptotics of the eigenvalues and other statistics of matrix-valued \emph{stochastic processes};
please see, e.g., \cite{Kemp2016,CDK2018,DGS2021,CGP2022,DHK2022,MP2022,Parraud2022,BCC2025,DHHKNNP2025,BCKP2025,JNP2026}.
A common feature of all this work is a combination of tools from stochastic calculus (applied to finite-rank matrices) and the \emph{free} stochastic calculus (applied in operator algebras that host the limits in noncommutative distribution of the matrix-valued processes) developed by Biane and Speicher in \cite{Biane1997,BS1998,BS2001}.
Their framework is built on a foundation developed initially in the 1980s in the quantum field theory world, notably starting with Hudson--Parthasarathy \cite{HP1984} and Applebaum--Hudson \cite{AH1984}, with progress by K\"ummerer and Speicher \cite{KS1992} leading to the more general treatment by Biane and Speicher a few years later.
Their development was the first to produce a theory of noncommutative stochastic integration built intrinsically for a specific integrator, rather than built from a concrete Fock-space model.
This integrator was \emph{free Brownian motion} $S=(S(t))_{t\ge 0}$, a special operator-valued process with freely independent increments.
Biane--Speicher defined stochastic integrals of the form $\into A(t)\,\d S(t)\,B(t)$ for operator-valued processes $A$ and $B$, and more generally for ``biprocesses'' taking values in a tensor product---the above integral is the special case $\into (A(t)\otimes B(t))\sh \d S(t)$.
Their development relied heavily on the free independence of the increments of $S$.
Some of their constructions were later generalized to other integrators with free increments--type properties (such as $q$-L\'evy processes);
please see, e.g., \cite{Anshelevich2000,Anshelevich2002,Anshelevich2004,DS2013,DS2018}.
Along a parallel track, philosophical descendants of the original framers of noncommutative stochastic calculus for constructive quantum field theory have continued development of those theories in hairier (non-tracial) contexts, even quite recently;
please see \cite{ABDVG2022,DVFGG2025,DVFG2025,CHP2025}.

The authors of the present work are undertaking the challenge of developing a \emph{general} theory of noncommutative stochastic calculus (in the tracial context), following as closely as possible the rich development in the classical world.
Our first paper in this direction \cite{JKN2026} developed a theory of stochastic integration and quadratic (co)variation for a class of integrators, inspired by semimartingales, we called \emph{$L^2$-decomposable processes}.
Key to our stochastic integration framework was to realize integrands as processes with values in a space of linear maps---\emph{linear processes}.
The usual biprocess $A \otimes B$, acting by $(A\otimes B)\# S = ASB$, becomes a special case of a linear process.
In the present paper, which is, at least spiritually, a sequel to \cite{JKN2026}, we push this valuable perspective much further by formulating a notion of a \emph{predictable linear process}.
With predictable linear processes as our integrands, we expand our stochastic integration theory to a much wider class of integrators: \emph{noncommutative $L^p$-integrators}.
We then apply our theory to develop an existence-and-uniqueness theory for noncommutative stochastic differential equations (SDEs).
As we explain in subsection \ref{subsec.NCSDEs}, our development vastly generalizes all known results about free SDEs, comports with classical It\^o SDE theory, and offers the first theory of It\^o-type SDEs driven by $q$-Brownian motion with $-1\le q<0$.

\bigskip

\noindent\textbf{Organization of paper.} Subsection \ref{subsec.clintegrators} discusses the notion of a classical $L^p$-integrator, which motivates our framework for noncommutative stochastic integration, as well as some important notation and terminology relating to bounded convergence.
Subsection \ref{subsec.NCLpintegrators} lays out our new framework for noncommutative stochastic integration, including our definition of a noncommutative $L^p$-integrator, and summarizes our results on examples of noncommutative $L^p$-integrators.
Subsection \ref{subsec.NCSDEs} discusses important special cases of our results on noncommutative SDEs.

Subsection \ref{subsec.generalities} develops general properties of predictable linear processes and the noncommutative stochastic integral, discusses the most common examples of predictable linear processes, and establishes a useful criterion for a noncommutative stochastic process to be a noncommutative $L^p$-integrator.
The remainder of section \ref{sec.NCstochint} is devoted to studying examples of noncommutative integrators---measured decomposable processes (subsection \ref{subsec.MD}), multidimensional free Brownian motion (subsection \ref{subsec.fBMintegrator}), and classical integrators (subsection \ref{subsec.classical})---and their associated integrals.

Section \ref{sec.SIEp} is devoted to formulating and proving our major result, Theorem \ref{thm.SIE}, on the existence and uniqueness of maximal solutions to a large, general class of noncommutative stochastic integral equations.
The result includes a finite-time blowup phenomenon new to the world of noncommutative probability.

Finally, section \ref{sec.SIEpexamples} uses results from subsections \ref{subsec.MD}--\ref{subsec.classical} and $L^p$-Lipschitz properties of the continuous functional calculus to study important examples of noncommutative stochastic integral equations covered by Theorem \ref{thm.SIE}.

\subsection{Classical \texorpdfstring{$L^p$}{}-integrators}\label{subsec.clintegrators}

Let us begin by explaining the inspiration for our proposed new framework for noncommutative stochastic integration, the fruitful notion of an $L^p$-integrator introduced by K.\ Bichteler in the late 1970s and early 1980s (vid.\ \cite{Bichteler1979,Bichteler1981}) and expanded on later in Bichteler's book \cite{Bichteler2002}.
Although we shall review the most important terminology and notation as we go, some familiarity with stochastic analysis will be important at various points;
please consult \cite{CW1990,RY1999,Protter2005} for the requisite background.

Let $(\Om,\sF,(\sF_t)_{t \geq 0},P)$ be a classical filtered probability space satisfying the \textbf{usual conditions}, i.e., $(\sF_t)_{t \geq 0}$ is complete ($\{A \subseteq \Om : \exists B \in \sF$ with $A \subseteq B$ and $P(B) = 0\} \subseteq \sF_0$) and right-continuous ($\sF_{t+} \coloneqq \bigcap_{u > t} \sF_u = \sF_t$ for all $t \geq 0$).\footnote{Much of the theory of stochastic integration, stochastic calculus, and stochastic differential equations survives unscathed if the filtration is only assumed to be complete.
We demand the usual conditions for convenience and to avoid technical headaches, as do many others.}
If $(S,\sS)$ is a measurable space, an (\textbf{adapted}) \textbf{$\boldsymbol{S}$-valued stochastic process} is a map $X \colon \R_+ \times \Om \to S$ such that $X_t \coloneqq X(t,\cdot) \colon \Om \to S$ is $\sF$--$\sS$ measurable ($\sF_t$--$\sS$ measurable) for all $t \geq 0$.
Two $S$-valued stochastic processes $X$ and $Y$ are \textbf{modifications} of each other if for each $t \geq 0$, $X_t = Y_t$ almost surely.
They are \textbf{indistinguishable} from each other if the statement, ``$X_t = Y_t$ for all $t \geq 0$,'' is true almost surely.
If $S$ is a topological space, $X$ is (\textbf{left-}/\textbf{right-})\textbf{continuous} if the statement, ``$t \mapsto X_t$ is (left-/right-)continuous,'' is true almost surely.
Also, $X$ is \textbf{RCLL} (\textbf{r}ight-\textbf{c}ontinuous with \textbf{l}eft \textbf{l}imits) if $X$ is right-continuous and the statement, ``$\lim_{s \nearrow t}X_s$ exists for every $t > 0$,'' is true almost surely.
Finally, let $\sP \subseteq 2^{\R_+ \times \Om}$ denote the \textbf{predictable $\boldsymbol{\sigma}$-algebra}, i.e., the $\sigma$-algebra on $\R_+ \times \Om$ generated by the family
\begin{equation}
    \sP_0 \coloneqq \{\{0\} \times F : F \in \sF_0\} \cup \{(s,t] \times F : 0 \leq s < t, \, F \in \sF_s\}.\label{eq.P0}
\end{equation}
An $S$-valued stochastic process $X \colon \R_+ \times \Om \to S$ is \textbf{predictable} if it is $\sP$--$\sS$ measurable.
For $(S,\sS) = (\C^n,\cB_{\C^n})$, the most important examples of predictable processes are adapted processes with left-continuous paths (i.e., $t \mapsto X_t$ is \emph{surely} left-continuous).

Loosely speaking, a stochastic process $X$ is an $L^p$-integrator if it is RCLL and there exists an $L^p$-valued stochastic integral against $X$ defined on all bounded predictable processes that satisfies a bounded convergence theorem--type continuity property.
We formulate the precise definition differently from Bichteler.
Up to modifications, our definition is equivalent to his;
the reader may consult \cite[Chs.\ 2--4]{Bichteler2002} for the details.

An \textbf{elementary predictable process} is a process of the form
\begin{equation}
    Y = 1_{\{0\} \times \Om}\,f_0 + \sum_{i=1}^k 1_{(s_i,t_i] \times \Om} \, f_i,\label{eq.EP}
\end{equation}
where $0 \leq s_i < t_i$ for all $i=1,\ldots,k$ and $f_i \colon \Om \to \C$ is a bounded, $\sF_{s_i}$-measurable random variable for all $i=0,\ldots,k$ (with $s_0 \coloneqq 0$).
Write $\sEP$ for the set of all elementary predictable processes.
For any stochastic process $X \colon \R_+ \times \Om \to \C$ and any $Y \in \sEP$ as in \eqref{eq.EP}, one can always define
\[
\int_0^t Y\,\d X = \int_0^t Y_s\,\d X_s \coloneqq \sum_{i=1}^k f_i(X_{t_i \wedge t} - X_{s_i \wedge t}) \qquad (t \geq 0),
\]
where $a \wedge b \coloneqq \min\left\{a,b\right\}$.
If $X$ is adapted, then so is
\[
\into Y\,\d X \coloneqq \left(\int_0^t Y\,\d X\right)_{t \geq 0}.
\]
If $X$ is RCLL (respectively, continuous), then so is $\into Y\,\d X$.

As hinted above, the definition of an $L^p$-integrator involves bounded convergence.
In fact, bounded convergence and related notions are essential to much of the paper.

\begin{definition}\label{def.bddconv}
Let $\cV$ be a normed vector space and $\Xi$ be a set or topological space, as needed.
Write $\ell^{\infty}(\Xi;\cV)$ for the space of bounded functions from $\Xi$ to $\cV$ with norm
\[
\norm{F}_{\ell^{\infty}(\Xi;\cV)} \coloneqq \sup_{\xi \in \Xi}\norm{F(\xi)}_{\cV}.
\]
If $\Xi$ is a topological space, we write $\ell_{\loc}^{\infty}(\Xi;\cV)$ for the set of locally bounded functions from $\Xi$ to $\cV$, i.e., the set of functions $F \colon \Xi \to \cV$ such that for all $\xi \in \Xi$, there exists an open neighborhood $U \subseteq \Xi$ of $\xi$ on which $F$ is bounded.
If $\cV=\C$, we suppress $\cV$ in the notation.
\begin{enumerate}[font=\normalfont, label=(\roman*)]
    \item A sequence of maps $(F_n \colon \Xi \to \cV)_{n \in \N}$ \textbf{converges boundedly} to $F \colon \Xi \to \cV$ if $(F_n)_{n \in \N}$ converges pointwise to $F$ and $\sup \big\{ \norm{F_n}_{\ell^{\infty}(\Xi;\cV)} : n \in \N\big\} < \infty$.
    A sequence of maps $(F_n \colon \Xi \to \cV)_{n \in \N}$ \textbf{converges locally boundedly} to $F \colon \Xi \to \cV$ if for every $\xi \in \Xi$, there exists an open neighborhood $U \subseteq \Xi$ of $\xi$ such that $(F_n|_U)_{n \in \N}$ converges boundedly to~$F|_U$.\label{item.(loc)bddconv}
    \item A subset $\cS$ of $\ell_{\text{(}\loc\text{)}}^{\infty}(\Xi;\cV)$ is \textbf{closed under} (\textbf{locally}) \textbf{bounded convergence} if whenever $(F_n)_{n \in \N}$ is a sequence in $\cS$ converging (locally) boundedly to a function $F \colon \Xi \to \cV$, we have that $F \in \cS$.\label{item.cl(loc)bddconv}
    \item Let $\cX$ be a topological space and $\cS \subseteq \ell_{\text{(}\loc\text{)}}^{\infty}(\Xi;\cV)$.
    A function $\Phi \colon \cS \to \cX$ is \textbf{continuous with respect to} (\textbf{locally}) \textbf{bounded convergence} if whenever $(F_n)_{n \in \N}$ is a sequence in $\cS$ converging (locally) boundedly to $F \in \cS$, we have that $\Phi(F_n) \to \Phi(F)$ in $\cX$ as $n \to \infty$.\label{item.cont(loc)bddconv}
\end{enumerate}
If $Y \colon \R_+ \times \Om \to \cV$ is a stochastic process that is bounded on compact time intervals, i.e., $\norm{Y}_{\ell^{\infty}([0,t] \times \Om;\cV)} < \infty$ for all $t \geq 0$, then we may view $Y$ as the locally bounded function $\R_+ \ni t \mapsto \tilde{Y}(t) \coloneqq Y_t \in \ell^{\infty}(\Om;\cV)$.
Taking this view, the above concepts involving locally bounded convergence apply;
when applying those concepts in this setting, we shall use the term \textbf{bounded convergence on compact time intervals} instead of locally bounded convergence.
For example, if $\cX$ is a topological space and $\sC$ is some collection of stochastic processes that are bounded on compact time intervals, a function $\Phi \colon \sC \to \cX$ is \textbf{continuous with respect to bounded convergence on compact time intervals} if $\Phi(Y_n) \to \Phi(Y)$ in $\cX$ as $n \to \infty$ whenever $Y \in \sC$ and $(Y_n)_{n \in \N}$ is a sequence in $\sC$ such that for all $t \geq 0$, $(1_{[0,t] \times \Om}Y_n)_{n \in \N}$ converges boundedly to $1_{[0,t] \times \Om}Y$.
\end{definition}

\begin{remark}\label{rem.notop}
Bounded convergence, even of sequences, does not come from any topology.
This is why (pseudo-)topological terminology must be carefully formulated when bounded convergence is the desired mode of convergence.
\end{remark}

Now, if $Y \colon \R_+ \times \Om \to \C$ is a stochastic process---or, more precisely, an indistinguishability class of stochastic processes---define
\[
|Y|_t^* \coloneqq \sup_{0 \leq s \leq t} |Y_s| \qquad ( t \geq 0).
\]
If $1 \leq p \leq \infty$, write $\cR_{a,p}$ (respectively, $\cC_{a,p}$) for the space of indistinguishability classes of adapted RCLL (respectively, adapted continuous) stochastic processes $Y \colon \R_+ \times \Om \to \C$ such that $\norm{|Y|_t^*}_p < \infty$ for all $t \geq 0$, where $\norm{\cdot}_p = \norm{\cdot}_{L^p(P)} = \norm{\cdot}_{L^p(\Om,\sF,P)}$.
We endow $\cR_{a,p}$ with the topology induced by the collection
\[
\{Y \mapsto \norm{|Y|_t^*}_p : t \geq 0\}
\]
of seminorms, which makes it into a complex Fr\'echet space with $\cC_{a,p}$ as a closed subspace.
Finally, write $\sP_{\mathrm{bct}}$ for the set of predictable processes $Y \colon \R_+ \times \Om \to \C$ that are bounded on compact time intervals.

\begin{definition}\label{def.Lpintegrator}
Suppose $1 \leq p < \infty$.
A stochastic process $X\colon \R_+ \times \Om \to \C$ is an \textbf{$\boldsymbol{L^p}$-integrator} if there exists a map $I_X^c \colon \sP_{\mathrm{bct}} \to \cR_{a,p}$ that is continuous with respect to~bounded convergence on compact time intervals and satisfies $I_X^c[Y] = \into Y\,\d X$ for all $Y \in \sEP$.
In this case, $I_X^c$ (is unique and) is called the (\textbf{classical}) \textbf{stochastic integral against $\boldsymbol{X}$}.
We write $\into Y\,\d X = \into Y_t\,\d X_t \coloneqq I_X^c[Y]$ for all $Y \in \sP_{\mathrm{bct}}$.
\end{definition}

Though the filtration $(\sF_t)_{t \geq 0}$ is suppressed in the definition of an $L^p$-integrator, we draw attention to the fact that it is featured in both the domain and codomain of the stochastic integral map.
It matters less in the codomain:
The notion of an $L^p$-integrator would remain the same if we replaced $\cR_{a,p}$ with the set of indistinguishability classes of RCLL stochastic processes $Y \colon \R_+ \times \Om \to \C$ such that $\norm{|Y|_t^*}_p < \infty$ for all $t \geq 0$.
The filtration's role is absolutely crucial, however, in the domain in that it determines the collection of processes eligible to be integrated.
More specifically, the filtration determines the collection $\sEP$ of building-block integrands, and then, by a standard application of the multiplicative system theorem (vid.\ \cite[Thm.\ I.21]{DM1975}), the full collection $\sP_{\mathrm{bct}}$ of integrands is the smallest set of processes that contains $\sEP$ and is closed under bounded convergence on compact time~intervals.

\subsection{Noncommutative \texorpdfstring{$L^p$}{}-integrators}\label{subsec.NCLpintegrators}

We now turn to the noncommutative setting.
We appeal frequently to standard objects and constructions in noncommutative probability and the theory of noncommutative $L^p$ spaces;
please see \cite[\S2.1 \& app.\ A]{JKN2026} and the references therein, as well as \cite{DdPS2023}, for the requisite background.
We shall review some of the important terminology and notation as we go.

A $\boldsymbol{\mathrm{C}^*}$\textbf{-probability space} is a pair $(\cA,\E)$, where $\cA$ is a unital $\mathrm{C}^*$-algebra and $\E \colon \cA \to \C$ is a faithful, tracial state.
Such a pair is a $\boldsymbol{\mathrm{W}^*}$\textbf{-probability space} if $\cA$ is a von Neumann algebra and $\E$ is also normal (i.e., continuous in the $\sigma$-weak operator topology).
A \textbf{filtration} of a $\mathrm{C}^*$-probability space $(\cA,\E)$ is a family $(\cA_t)_{t \geq 0}$ of unital $\mathrm{C}^*$-subalgebras of $\cA$ such that $\cA_s \subseteq \cA_t$ whenever $0 \leq s \leq t$.
In this case, $(\cA,(\cA_t)_{t \geq 0},\E)$ is a \textbf{filtered $\boldsymbol{\mathrm{C}^*}$-probability space}.
If $(\cA,\E)$ is a $\mathrm{W}^*$-probability space and $\cA_t$ is a $\mathrm{W}^*$-subalgebra of $\cA$ for all $t \geq 0$, then $(\cA,(\cA_t)_{t \geq 0},\E)$ is called a \textbf{filtered $\boldsymbol{\mathrm{W}^*}$-probability space}.
For example, if $(\Om,\sF,(\sF_t)_{t \geq 0},P)$ is a classical filtered probability space, then
\begin{equation}
    (\cA,(\cA_t)_{t \geq 0},\E) \coloneqq (L^{\infty}(\Om,\sF,P), (L^{\infty}(\Om,\sF_t,P))_{t \geq 0}, \E_P)\label{eq.classicalfWps}
\end{equation}
is a filtered $\mathrm{W}^*$-probability space.
For the rest of this subsection, let $(\cA,(\cA_t)_{t \geq 0},\E = \E_{\scA})$ and $(\cB,(\cB_t)_{t \geq 0},\E_{\scB})$ be filtered $\mathrm{W}^*$-probability spaces.

If $1 \leq p \leq \infty$, a \textbf{noncommutative $\boldsymbol{L^p}$ stochastic process}---an \textbf{$\boldsymbol{L^p}$ process} for short---is a function from $\R_+ = [0,\infty)$ to the noncommutative $L^p$ space $L^p(\E) = L^p(\cA,\E)$.
If $(\cA,(\cA_t)_{t \geq 0},\E)$ is a ``classical $\mathrm{W}^*$-probability space'' as in \eqref{eq.classicalfWps}, this object is the same thing as the modification class of a classical $L^p$ stochastic process.
An $L^p$ process $X \colon \R_+ \to L^p(\E)$ is \textbf{adapted} (\textbf{to $\boldsymbol{(\cA_t)_{t \geq 0}}$}) if $X(t) \in L^p(\cA_t,\E)$ for all $t \geq 0$.

So, in this setting, what is a reasonable analog of the definition of an $L^p$-integrator?
To prepare for our answer, let us reframe classical elementary predictable processes as linear map--valued processes in the style of our previous work \cite{JKN2026}.
Let $Y \colon \R_+ \times \Om \to \C$ be a classical $L^r$ stochastic process for some $r \in [1,\infty]$.
We saw in the previous paragraph that the modification class of $Y$ may be viewed as the noncommutative $L^r$ stochastic process $\R_+ \ni t \mapsto Y_t \in L^r(\E_P) = L^r(\Om,\sF,P)$.
There is another way to ``noncommutativize'' the modification class of $Y$ that is useful when thinking of $Y$ as a stochastic integrand.
Indeed, suppose $1 \leq q \leq p \leq \infty$ and $1/r+1/p=1/q$.
If $y \in L^r(\E_P)$, write $M_y \colon L^p(\E_P) \to L^q(\E_P)$ for the multiplication-by-$y$ operator:
\begin{equation}
    M_yx \coloneqq yx \in L^q(\E_P) \qquad (x \in L^p(\E_P)).\label{eq.Moperator}
\end{equation}
If $M \colon L^r(\E_P) \to B(L^p(\E_P);L^q(\E_P))$ is the map $y \mapsto M_y$, then $M$ is a linear isometry.
Thus, the modification class of $Y$ may be identified with the linear map--valued process
\begin{equation}
    \R_+ \ni t \mapsto H_Y(t) \coloneqq M_{Y_t}  \in B(L^p(\E_P);L^q(\E_P)).\label{eq.HY}
\end{equation}
Consequently, a linear map--valued process $H \colon \R_+ \to B(L^p(\E_P)) = B(L^p(\E_P);L^p(\E_P))$ ``is'' the modification class of a $Y \in \sEP$ if and only if there exist times $s_i < t_i$ ($i=1,\ldots,k$) and elements $y_i \in L^{\infty}(\Om,\sF_{s_i},P)$ ($i=0,\ldots,k$, with $s_0 \coloneqq 0$) such that
\[
H = 1_{\{0\}} M_{y_0} + \sum_{i=1}^k 1_{(s_i,t_i]} M_{y_i}.
\]
Note that if $i =0,\ldots,k$, $u \geq s_i$, and $x \in L^p(\E_P)$, then
\[
\E_P[M_{y_i}x \mid \sF_u] = y_i\E_P[x \mid \sF_u] = M_{y_i}\E_P[x \mid \sF_u].
\]
Thus, such an $H$ is an elementary predictable $(L^p;L^p)$-process in the sense of \cite[Def.\ 4.9]{JKN2026}.
However, it is a \emph{special kind} of elementary predictable $(L^p;L^p)$-process;
it is adapted to a smaller ``filtration'' of the space of bounded linear maps on $L^p$ than the one introduced in \cite{JKN2026}.
This insight leads us to our proposed definitions of predictability and, soon after, integrators in the noncommutative case.

\begin{definition}\label{def.predlin}
Suppose $1 \leq p,q \leq \infty$.
Write $B^{p;q} = B^{p;q}(\E_{\scA};\E_{\scB})$ for the space of bounded real-linear maps from $L^p(\E_{\scA})$ to $L^q(\E_{\scB})$ and $\norm{\cdot}_{p;q} = \norm{\cdot}_{L^p \to L^q}$ for the operator norm on~$B^{p;q}$.
\begin{enumerate}[label=(\roman*),font=\normalfont]
    \item A \textbf{normed filtration} of $B^{p;q}$ is a pair $(\G,\norm{\cdot})$, where $\G = (\cG_t)_{t \geq 0}$ is an increasing family of complex-linear subspaces of $B^{p;q}$ and $\norm{\cdot}$ is a norm on $\cG_{\infty} \coloneqq \bigcup_{t \geq 0} \cG_t$ such that $(\cG_t,\norm{\cdot})$ is a Banach space for each $t \geq 0$.
    Such a pair is \textbf{compatible} (with $(\cA_t)_{t \geq 0}$ and $(\cB_t)_{t \geq 0}$) if for all $t \geq 0$,
    \[
    \E_{\scB}[ Tx \mid \cB_u] = T\E_{\scA}[x \mid \cA_u] \qquad \big(u \geq t, \; T \in \cG_t, \; x \in L^p(\E_{\scA})\big).
    \]
    For the remainder of this definition, fix a normed filtration $(\G,\norm{\cdot})$ of $B^{p;q}$.\label{item.normedfiltration}
    \item $H \colon \R_+ \to \cG_{\infty}$ is a \textbf{$\boldsymbol{\G}$-adapted linear process} if $H(t) \in \cG_t$ for all $t \geq 0$.\label{item.Gadapt}
    \item $H \colon \R_+ \to \cG_{\infty}$ is called an \textbf{elementary $\boldsymbol{\G}$-predictable linear process} if there exist times  $t_i > s_i$ ($i=1,\ldots,k$) and elements $T_i \in \cG_{s_i}$ ($i=0,\ldots,k$, with $s_0 \coloneqq 0$) such~that
    \begin{equation}
        H = 1_{\{0\}} T_0 + \sum_{i=1}^k 1_{(s_i,t_i]} T_i.\label{eq.EPG}
    \end{equation}
    Write $\EP_{\G} \subseteq \ell_{\loc}^{\infty}(\R_+;\cG_{\infty})$ for the set of all elementary $\G$-predictable linear processes.\label{item.EPG}
    \item Write $\cI_{\G}$ for the smallest subset of $\ell_{\loc}^{\infty}(\R_+;\cG_{\infty})$ that contains $\EP_{\G}$ and is closed under locally bounded convergence with respect to $\norm{\cdot}$ (Definition \ref{def.bddconv}).
    The members of $\cI_{\G}$ are called (\textbf{locally bounded}) \textbf{$\boldsymbol{\G}$-predictable linear processes}.\label{item.IG}
\end{enumerate}
\end{definition}

\begin{remark}\label{rem.differentnorm}
The level of generality in Definition \ref{def.predlin}\ref{item.normedfiltration} may seem strange to the reader.
We shall see, e.g., in subsection \ref{subsec.fBMintegrator} and at the end of subsection \ref{subsec.classical}, that it is useful to have a great deal of flexibility in the choice of the norm on our filtrations;
please see also Notation \ref{nota.trpolyfiltration} and what follows it.
\end{remark}

\begin{example}[Classical case]\label{ex.clpred}
Let $(\Om,\sF,(\sF_t)_{t \geq 0},P)$ be a classical filtered probability space satisfying the usual conditions, take $(\cA,(\cA_t)_{t \geq 0},\E_{\scA})$ and $(\cB,(\cB_t)_{t \geq 0},\E_{\scB})$ both to be the classical filtered $\mathrm{W}^*$-probability space from \eqref{eq.classicalfWps}, and suppose $1 \leq r,p,q \leq \infty$ satisfy $1/r+1/p = 1/q$.
If
\[
\cM_t^{p;q} \coloneqq \{M_y : y \in L^r(\Om,\sF_t,P)\} \subseteq B^{p;q} \qquad (t \geq 0)
\]
and $\M^{p;q} \coloneqq (\cM_t^{p;q})_{t \geq 0}$, then $(\M^{p;q},\norm{\cdot}_{p;q})$ is a compatible normed filtration of $B^{p;q}$.
The $\M^{p;q}$-adapted linear processes are precisely the modification classes of classical adapted $L^r$ stochastic processes viewed as time-dependent multiplication operators.
One of the main results of subsection \ref{subsec.classical}, Theorem \ref{thm.clpred}, says:
(1) Every $H \in \cI_{\M^{p;q}}$ is of form $H=H_Y$ (vid.\ \eqref{eq.HY}) for some classical predictable $L^r$ stochastic process $Y \colon \R_+ \times \Om \to \C$, and (2) if $p > q$ (i.e., $r < \infty$) and $Y \colon \R_+ \times \Om \to \C$ is a classical predictable $L^r$ stochastic process, then $H_Y \in \cI_{\M^{p;q}}$.
Thus, Definition \ref{def.predlin}\ref{item.IG} is a reasonable ``noncommutativization'' of the notion of predictability from classical stochastic analysis.
\end{example}

\begin{example}[$(L^p;L^q)$-processes]\label{ex.LpLq}
Suppose $1 \leq p,q \leq \infty$.
If
\[
\cF_t^{p;q} = \cF_t^{p;q}(\E_{\scA};\E_{\scB})  \coloneqq \bigcap_{u \geq t} \{T \in B^{p;q} : \E_{\scB}[Tx \mid \cB_u] = T\E_{\scA}[x \mid \cA_u] \; \forall x \in L^p(\E_{\scA})\} \quad (t \geq 0)\pagebreak
\]
and $\F^{p;q} = \F^{p;q}(\E_{\scA};\E_{\scB}) \coloneqq (\cF_t^{p;q})_{t \geq 0}$, then $(\F^{p;q},\norm{\cdot}_{p;q})$ is a compatible normed filtration of $B^{p;q}$.
We write
\[
\EP^{p;q} = \EP^{p;q}(\E_{\scA};\E_{\scB}) \coloneqq \EP_{\F^{p;q}} \; \text{ and } \; \cI^{p;q} = \cI^{p;q}(\E_{\scA};\E_{\scB}) \coloneqq \cI_{\F^{p;q}}.
\]
An $\F^{p;q}$-adapted linear process is called an \textbf{adapted $\boldsymbol{(L^p;L^q)}$-process}.
An (elementary) $\F^{p;q}$-predictable linear process is called a(n \textbf{elementary}) \textbf{predictable $\boldsymbol{(L^p;L^q)}$-process}.
Adapted $(L^p;L^q)$-processes were introduced in \cite[Def.\ 3.3]{JKN2026}, and elementary predictable $(L^p;L^q)$-processes were introduced in \cite[Def.\ 4.9]{JKN2026}.
The notion of a predictable  $(L^p;L^q)$-process is new.
\end{example}

We study more examples and general properties of $\G$-predictable linear processes in subsection \ref{subsec.generalities}.

Next come the integrators.
If $X \colon \R_+ \to L^p(\E_{\scA})$ is any $L^p$ process, $(\G,\norm{\cdot})$ is a normed filtration of $B^{p;q}$, and $H \in \EP_{\G}$ is as in \eqref{eq.EPG}, then we define
\[
\int_0^t H[\d X] = \int_0^t H(s)[\d X(s)] \coloneqq \sum_{i=1}^k T_i[X(t_i \wedge t) - X(s_i \wedge t)] \qquad (t \geq 0).
\]
Also, write $R_{a,p} = R_{a,p}(\E_{\scA})$ (respectively, $C_{a,p} = C_{a,p}(\E_{\scA})$) for the space of adapted $L^p$ processes that are RCLL (respectively, continuous) with respect to $\norm{\cdot}_p$, the noncommutative $L^p$ norm.
We endow $R_{a,p}$ with the topology of uniform convergence on compact sets, which makes it into a complex Fr\'echet space with $C_{a,p}$ as a closed subspace.

\begin{definition}\label{def.NCintegrator}
Suppose $1 \leq p_0,p \leq \infty$ and $(\G,\norm{\cdot})$ is a compatible normed filtration of $B^{p_0;1}$.
An $L^{p_0}$ process $X \colon \R_+ \to L^{p_0}(\E_{\scA})$ is a (\textbf{noncommutative}) \textbf{$\boldsymbol{L^p}$-integrator with respect to $\boldsymbol{\G}$} if there exists a map $I_X = I_X^{\G} \colon \cI_{\G} \to R_{a,p}$ that is continuous with respect to locally bounded convergence (Definition \ref{def.bddconv}) and satisfies $I_X[H] = \into H[\d X]$ for all $H \in \EP_{\G}$.
In this case, the map $I_X$ is called the (\textbf{noncommutative}) \textbf{stochastic integral against $\boldsymbol{X}$}.
We write $\into H[\d X] = \into H(t)[\d X(t)] \coloneqq I_X[H]$ for all $H \in \cI_{\G}$.
\end{definition}

Suppose $X \colon \R_+ \to L^{p_0}(\E_{\scA})$ is an $L^p$-integrator with respect to $\G$.
We shall see in subsection \ref{subsec.generalities} (specifically, Theorem \ref{thm.abstractstochint}) that the map $I_X$ is unique, so calling it \emph{the} stochastic integral against $X$ is reasonable.
We shall also see that if $\into H[\d X] \in C_{a,p}$ for all $H \in \EP_{\G}$, then $\into H[\d X] \in C_{a,p}$ for all $H \in \cI_{\G}$.
In this case, we say $X$ is a \textbf{continuous $\boldsymbol{L^p}$-integrator with respect to $\boldsymbol{\G}$}.

\begin{example}[Classical case]\label{ex.clint}
Return to the context of Example \ref{ex.clpred}, and suppose $p < \infty$.
Let $X \colon \R_+ \times \Om \to \C$ be a classical $L^p$ stochastic process, and write $x \colon \R_+ \to L^p(\E)$ for its modification class, i.e., $x(t) \coloneqq X_t$ for all $t \geq 0$.
The other main result of subsection \ref{subsec.classical}, Theorem \ref{thm.clvsncstochint}, says that (i) if $x$ is a noncommutative $L^p$-integrator with respect to $\M^{p;p}$, then $X$ has a modification that is a classical $L^p$-integrator; and
(ii) if $X$ has a modification that is a classical $L^p$-integrator with the property that stochastic integrals against it depend only on the modification class of the integrand,\footnote{In general, if $Z$ is a classical $L^p$-integrator and $U,V \in \sP_{\mathrm{bct}}$ are indistinguishable, then $I_Z^c[U] = I_Z^c[V]$ (vid.\ \cite[Cor.\ 3.7.13]{Bichteler2002}).
This is not always true when $U$ and $V$ are only modifications of each other.}
then $x$ is a noncommutative $L^p$-integrator with respect to $\M^{p;p}$.
Consequently, when applied appropriately to the classical case, Definition \ref{def.NCintegrator} recovers the classical notion of an $L^p$-integrator modulo the fact that the noncommutative framework only sees the modification class of a stochastic process.
\end{example}

\begin{example}[Measured decomposable processes]\label{ex.MD}
In \cite{JKN2026}, we introduced the notion of an $L^p$-decomposable process and constructed $L^2$-valued stochastic integrals against $L^2$-decomposable processes.
In subsection \ref{subsec.MD}, we introduce the more restrictive notion of an $L^p$-measured decomposable process (Definition \ref{def.measured}) and prove that if $1 \leq p_0 \leq \infty$, $X \colon \R_+ \to L^{p_0}(\E_{\scA})$ is an $L^{p_0}$-measured decomposable process, and $2 \leq p < \infty$, then $X$ is a continuous $L^p$-integrator with respect to $\F^{p_0;p}$ (Theorem \ref{thm.MDass}).
In Example \ref{ex.2measured}, we observe that all $L^2$-decomposable processes are automatically $L^2$-measured decomposable processes, so the result in the previous sentence generalizes our stochastic integral construction from \cite{JKN2026}.
We also show that several familiar noncommutative stochastic processes, e.g., the $q$-Brownian motions (vid.\ \cite{BKS1997,DonatiMartin2003,DS2013,DS2018}), are measured decomposable processes.
Finally, as an application of the theory of stochastic integration against measured decomposable processes, we generalize our continuous-time noncommutative Burkholder--Davis--Gundy inequalities from \cite{JKN2026};
please see Theorem \ref{thm.contimeNCDG}.
\end{example}

\begin{example}[Free Brownian motion]\label{ex.scBM}
One of the most celebrated noncommutative stochastic processes is the free Brownian motion, the large-$N$ limit of (appropriately scaled) Brownian motion on the space of $N \times N$ Hermitian matrices;
please see \cite{Biane1997,BS1998} or Definition \ref{def.ndimfBM} for a proper definition.
In \cite{BS1998}, P.\ Biane and R.\ Speicher developed a free stochastic calculus, i.e., a stochastic calculus for ``It\^o processes'' driven by free Brownian motion.
One useful feature of this theory is the operator-norm boundedness of the free stochastic integral (vid.\ \cite[Thm.~3.2.1]{BS1998}).
We use this result to prove that a free Brownian motion $S \colon \R_+ \to L^{\infty}(\E_{\scA}) = \cA$ is a continuous $L^{\infty}$-integrator with respect to the ``sub-filtration''
\[
(\G,\norm{\cdot}) \coloneqq \left( \left(\overline{\spn\left\{ x \mapsto axb + \E[cx]\,d : a,b,c,d \in \cA_t\right\}} \right)_{t \geq 0},\norm{\cdot}_{2;2}\right)
\]
of $(\F^{2;2},\norm{\cdot}_{2;2}) = (\F^{2;2}(\E_{\scA};\E_{\scA}),\norm{\cdot}_{2;2})$, where the closure takes place in $B^{2;2}$.
A standard example of a $\G$-predictable linear process is a linear process $H \colon \R_+ \to B^{2;2}$ of the form
\[
H(t)[x] = \sum_{i=1}^k \left(a_i(t)\,x\,b_i(t) + \E[c_i(t)\,x]d_i(t)\right) \qquad \big(t \geq 0, \, x \in L^2(\E_{\scA})\big),
\]
where $a_i,b_i,c_i,d_i \colon \R_+ \to \cA$ are continuous and adapted ($i=1,\ldots,k$).
Please see subsection \ref{subsec.fBMintegrator}, especially Theorem \ref{thm.Linfbound}, for details.

\end{example}

We go through additional general properties of the noncommutative stochastic integral as well as a condition under which it can be constructed in subsection \ref{subsec.generalities}.

\subsection{Noncommutative stochastic differential equations}\label{subsec.NCSDEs}

Retain the filtered $\mathrm{W}^*$-probability spaces $(\cA,(\cA_t)_{t \geq 0},\E = \E_{\scA})$ and $(\cB,(\cB_t)_{t \geq 0},\E_{\scB})$ from the previous subsection.
Suppose $1 \leq p_0,p \leq \infty$, $(\G,\norm{\cdot})$ is a compatible normed filtration of $B^{p_0;1}$, and $X \colon \R_+ \to L^{p_0}(\E_{\scA})$ is a continuous $L^p$-integrator with respect to $\G$.
Given an appropriate linear map--valued coefficient function $F(t,y)$, one can consider the \textbf{noncommutative stochastic differential equation} (NCSDE)
\[
\begin{cases}
    \d Y(t) = F(t,Y(t))[\d X(t)] & \\
    \;\, Y(0) = y_0.
\end{cases}
\]
We shall say that the above NCSDE has a \textbf{unique maximal solution} if there exists a time $T_{\ell} \in (0,\infty]$ and a continuous adapted process $Y$ defined on $[0,T_{\ell})$ such that:
\begin{enumerate}[label=(\alph*)]
    \item $\displaystyle Y = y_0 + \into F(t,Y(t))[\d X(t)]$ on $[0,T_{\ell})$;
    \item if $I \subseteq \R_+$ is an interval containing $0$ and $Y_0$ is a continuous adapted process defined on $I$ such that $\displaystyle Y_0 = y_0 + \into F(t,Y_0(t))[\d X(t)]$ on $I$, then $I \subseteq [0,T_{\ell})$, and $Y|_I = Y_0$.
\end{enumerate}
In this case, $Y$ is called the NCSDE's \textbf{maximal solution}, and $T_{\ell}$ is called its \textbf{lifetime}.
(Please see subsection \ref{subsec.SIEpsetup} for more details and precise definitions.)

In section \ref{sec.SIEp}, we formulate and prove a general result, Theorem \ref{thm.SIE}, on the existence and uniqueness of maximal solutions to a large class of NCSDEs.
Then, in section \ref{sec.SIEpexamples}, we examine a number of corollaries of Theorem \ref{thm.SIE}.
The present subsection highlights those that most clearly illustrate the novelty of Theorem \ref{thm.SIE}.

We begin by considering the case of free stochastic differential equations (FSDEs), i.e., NCSDEs driven by free Brownian motion.
Even in this relatively well-studied case, our results are new.
For the duration of this subsection, fix a choice of $\beta \in \{\emptyset,\sa\}$ that selects between, e.g., $\cA_{\emptyset} = \cA$ and $\cA_{\sa} = \{a \in \cA : a^*=a\}$.

\begin{theorem}[Existence/uniqueness for FSDEs]\label{thm.SDESCBMdriver}
Let $S \colon \R_+ \to \cA_{\sa}$ be a free Brownian motion.
Suppose that $F \colon \R_+ \times \cA_{\beta} \to B^{2;2}$ and $G \colon \R_+ \times \cA_{\beta} \to \cA_{\beta}$ are continuous and~satisfy:
\begin{enumerate}[font=\normalfont,label=(\roman*)]
    \item for all $t \geq 0$ and $y \in (\cA_t)_{\beta}$, $F(t,y)$ belongs to the closure in $B^{2;2}$ of the span of $\left\{ x \mapsto axb + \E[c\,x]d : a,b,c,d \in \cA_t\right\} \subseteq B^{2;2}$, and $G(t,y) \in (\cA_t)_{\beta}$;\label{item.FGadapted}
    \item $\displaystyle \int_r^t F(s,Y(s))[\d S(s)] \in (\cA_t)_{\beta}$ for all $t \geq r \geq 0$ and continuous adapted $Y \colon [r,t] \to \cA_{\beta}$;\label{item.intFbetaintro}
    \item for all $R \geq 0$, there exists an increasing family $(c_{T,R})_{T \geq 0}$ of positive real numbers such that for all $y,z \in \cA_{\beta}$ with $\norm{y}_{\infty} \leq R$ and $\norm{z}_{\infty} \leq R$,\label{item.FGLip}
    \[
    \norm{F(t,y) - F(t,z)}_{\infty;2}+\norm{G(t,y)-G(t,z)}_2 \leq c_{T,R}\norm{y-z}_2 \qquad (0 \leq t \leq T);
    \]
    \item for all $T,R \geq 0$,\label{item.FGbdd}
    \[
    \sup\left\{\norm{F(t,y)}_{2;2} + \norm{G(t,y)}_{\infty} : 0 \leq t \leq T, \, y \in \cA_{\beta}, \, \norm{y}_{\infty} \leq R  \right\} < \infty.
    \]
\end{enumerate}
For each $y_0 \in (\cA_0)_{\beta}$, there exists a unique maximal solution $Y \colon [0,T_{\ell}) \to \cA_{\beta}$ to the FSDE
\[
\begin{cases}
    \d Y(t) = F(t,Y(t))[\d S(t)]  + G(t,Y(t))\,\d t & \\
    \;\, Y(0) = y_0.
\end{cases}
\]
In addition, if $T_{\ell} < \infty$, then $\norm{Y(t)}_{\infty} \to \infty$ as $t \nearrow T_{\ell}$.
Finally, if for each $T \geq 0$, there exists an $M_T <\infty$ such that 
\[
\norm{F(t,y)}_{2;2}+\norm{G(t,y)}_{\infty} \leq M_T\left(\norm{y}_{\infty}+1\right) \qquad (0 \leq t \leq T, \; y \in \cA_{\beta}),
\]
then $T_{\ell}=\infty$.
\end{theorem}

This result is a special case of its ``multidimensional'' generalization, Theorem \ref{thm.SIESCBMdriver}.
Assumption \ref{item.FGadapted} enables us to make use of our operator-norm estimate on the stochastic integral against $S$ (Theorem \ref{thm.Linfbound}), which takes the form
\begin{equation}
    \norm{\int_0^t H[\d S]}_{\infty} \lesssim \left( \int_0^t \norm{H(s)}_{2;2}^2\,\d s \right)^\frac12 \qquad (t \geq 0).\label{eq.LinftybdFSI}
\end{equation}
With this in mind, it might seem more reasonable to replace $\norm{\cdot}_{\infty;2}$ with $\norm{\cdot}_{2;2}$ and $\norm{\cdot}_2$ with $\norm{\cdot}_{\infty}$ in assumption \ref{item.FGLip}, in which case assumption \ref{item.FGbdd} would be redundant.
In fact, if one makes these replacements, the resultant theorem is true and is still a corollary of Theorem \ref{thm.SIE}.
In practice, however, it is sometimes easier for a coefficient to satisfy the Lipschitz assumption in \ref{item.FGLip} than the one with the norms $\norm{\cdot}_{2;2}$ and $\norm{\cdot}_{\infty}$.
This is the case, for example, when the coefficients come from scalar functions via the continuous functional calculus (vid.\ \cite[Ch.\ VIII]{Conway1990}).
Indeed, using a celebrated result of D.\ Potapov and F.\ Sukochev from \cite{PS2011}, we are able to deduce the following from Theorem \ref{thm.SDESCBMdriver}.

\begin{corollary}\label{cor.funkycalcFSDE}
Suppose $f_i,g_i \colon \R_+ \times \R \to \C$ ($i=1,\ldots,k$) and $h \colon \R_+ \times \R \to \R$ are continuous functions that are locally Lipschitz in space, locally uniformly in time.\footnote{A function $f \colon \R_+ \times \R \to \C$ is \textbf{locally Lipschitz in space, locally uniformly in time} if for all $R \geq 0$ and $T \geq 0$, there exists a constant $c_{T,R} < \infty$ such that $|f(t,\lambda)-f(t,\mu)| \leq c_{T,R}|\lambda-\mu|$ for all $\lambda,\mu \in [-R,R]$ and all $t \in [0,T]$.}
If $S \colon \R_+ \to \cA_{\sa}$ is a free Brownian motion and $y_0 \in (\cA_0)_{\sa}$, then there exists a unique maximal solution $Y \colon [0,T_{\ell}) \to \cA_{\sa}$ to the FSDE
\[
\begin{cases}
    \displaystyle\d Y(t) =  \sum_{i=1}^k \left( f_i(t,Y(t))\,\d S(t)\, g_i(t,Y(t)) + \overline{g_i}(t,Y(t))\,\d S(t)\, \overline{f_i}(t,Y(t))\right) + h(t,Y(t))\,\d t & \\
    \;\, Y(0) = y_0,
\end{cases}
\]
where $h(t,Y(t))$, $f_i(t,Y(t))$, $g_i(t,Y(t))$, etc., are defined via the continuous functional calculus.
In addition, if $T_{\ell} < \infty$, then $\norm{Y(t)}_{\infty} \to \infty$ as $t \nearrow T_{\ell}$.
Finally, if for each $T \geq 0$, there exists an $M_T < \infty$ such that
\[
\sum_{i=1}^k\norm{f_i}_{\ell^{\infty}([0,T] \times [-R,R])}\norm{g_i}_{\ell^{\infty}([0,T] \times [-R,R])} + \norm{h}_{\ell^{\infty}([0,T] \times [-R,R])} \leq M_T (R+1) \quad (R \geq 0),
\]
then $T_{\ell} = \infty$.
Here and throughout, $\norm{\varphi}_{\ell^{\infty}(\Xi)} \coloneqq \sup\left\{|\varphi(\xi)| : \xi \in \Xi\right\}$ for a set $\Xi$ and a function $\varphi \colon \Xi \to \C$.
\end{corollary}

We explain in Example \ref{ex.funkycalcSCBMSIE} how to obtain Corollary \ref{cor.funkycalcFSDE} from Theorem \ref{thm.SDESCBMdriver}.
Note well that the Lipschitz assumptions in Corollary \ref{cor.funkycalcFSDE} are of the plain-old, scalar-function variety.
In particular, we emphasize the lack of operator Lipschitz--type assumptions, which were required in almost all previous work on similar FSDEs, e.g., \cite{BS2001,Kargin2011,DS2013,DS2018}.
Dropping these more stringent operator Lipschitz--type assumptions is possible precisely because we ``mix norms'' in the hypotheses of Theorem \ref{thm.SDESCBMdriver}.
To our knowledge, the only other existence and uniqueness result for FSDEs that avoids operator-Lipschitz assumptions comes from the preprint \cite{WY2026} of J.\ Wei and Z.\ Yin that was released during the final stages of the drafting of the present manuscript.
Specifically, their result implies the global existence and uniqueness \emph{of $L^2$ solutions} to FSDEs of the form $\d Y(t) = \d S(t) + f(Y(t))\,\d t$, where $f \colon \R \to \R$ is a globally Lipschitz function.
Corollary \ref{cor.funkycalcFSDE} says that the solutions to such equations, and much more general ones, are actually $L^{\infty}$ solutions.

Finally, we call special attention to the finite-time blow-up phenomenon in Theorem \ref{thm.SDESCBMdriver} (and Corollary \ref{cor.funkycalcFSDE}).
Though there are some local existence and uniqueness results for FSDEs in literature prior to the present work, e.g., \cite[Thm.\ 3.1]{Kargin2011}, none establish that a maximal solution must blow up if its lifetime is finite.

The second special class of NCSDEs we highlight consists of those driven by $q$-Brownian motion with $-1 \leq q < 1$ (vid.\ \cite{BKS1997,DonatiMartin2003,DS2013,DS2018}).
The family of $q$-Brownian motions interpolates between the fermionic ($q=-1$) and classical ($q=1$) Brownian motion, with the free Brownian motion in between ($q=0$).
If $-1 \leq q < 1$, then $q$-Brownian motion $X \colon \R_+ \to \cA_{\sa}$ is $L^{\infty}$-continuous and, as we show in subsection \ref{subsec.MD}, is a noncommutative $L^p$-integrator with respect to $\F^{\infty;p}$ for all finite $p$;
please see Example \ref{ex.LpLq} for the definition of $\F^{p_0;p} = (\cF_t^{p_0;p})_{t \geq 0}$.

\begin{theorem}[Existence/uniqueness for NCSDEs driven by $q$-BM]\label{thm.SDEqBM}
Suppose $-1 \leq q < 1$, $X \colon \R_+ \to \cA_{\sa}$ is a $q$-Brownian motion, and $2 \leq p < \infty$.
Assume $F \colon \R_+ \times L^p(\E_{\scB})_{\beta} \to B^{\infty;p}$ and $G \colon \R_+ \times L^p(\E_{\scB})_{\beta} \to L^p(\E_{\scB})_{\beta}$ are continuous and satisfy:
\begin{enumerate}[font=\normalfont,label=(\roman*)]
    \item $F(t,y) \in \cF_t^{\infty;p}$ and $G(t,y) \in L^p(\cB_t,\E_{\scB})_{\beta}$ whenever $t \geq 0$ and $y \in L^p(\cB_t,\E_{\scB})_{\beta}$;
    \item $\displaystyle \int_r^t F(s,Y(s))[\d X(s)] \in L^p(\cB_t,\E_{\scB})_{\beta}$ for all $t \geq r \geq 0$ and all continuous adapted processes $Y \colon [r,t] \to L^p(\E_{\scB})_{\beta}$;
    \item for all $R \geq 0$, there exists an increasing family $(c_{T,R})_{T \geq 0}$ of positive real numbers such that for all $y,z \in L^p(\E_{\scB})_{\beta}$ with $\norm{y}_p \leq R$ and $\norm{z}_p \leq R$,\label{item.qlocLip}
    \[
    \norm{F(t,y) - F(t,z)}_{\infty;p} + \norm{G(t,y) - G(t,z)}_p \leq c_{T,R}\norm{y-z}_p \qquad (0 \leq t \leq T).
    \]
\end{enumerate}
For each $y_0 \in L^p(\cB_0,\E_{\scB})_{\beta}$, there exists a unique maximal solution $Y \colon [0,T_{\ell}) \to L^p(\E_{\scB})_{\beta}$ to the NCSDE
\[
\begin{cases}
    \d Y(t) = F(t,Y(t))[\d X(t)] + G(t,Y(t))\,\d t & \\
    \;\, Y(0) = y_0.
\end{cases}
\]
In addition, if $T_{\ell} < \infty$, then $\norm{Y(t)}_p \to \infty$ as $t \nearrow T_{\ell}$.
Finally, if for each $T \geq 0$, there exists an $M_T <\infty$ such that 
\[
\norm{F(t,y)}_{\infty;p} +\norm{G(t,y)}_p \leq M_T\big(\norm{y}_p+1\big) \qquad \big(0 \leq t \leq T, \; y \in L^p(\E_{\scB})_{\beta}\big),
\]
e.g., if $c_{T,R}$ may be taken to be independent of $R$, then $T_{\ell}=\infty$.
\end{theorem}

This result is a consequence of Theorem \ref{thm.SIEMDdriver}, a similar-looking result about noncommutative stochastic integral equations driven by general measured decomposable processes.
Using once again results of Potapov and Sukochev, we obtain from Theorem \ref{thm.SDEqBM} a nice result about NCSDEs with functional-calculus coefficients.

\begin{corollary}\label{cor.funkycalcqSDE}
Suppose $f_i,g_i \colon \R_+ \times \R \to \C$ ($i=1,\ldots,k$) and $h \colon \R_+ \times \R \to \R$ are continuous functions that are bounded and Lipschitz in space, locally uniformly in time.\footnote{A function $f \colon \R_+ \times \R \to \C$ is \textbf{bounded in space, locally uniformly in time} if $\norm{f}_{\ell^{\infty}([0,T] \times \R)} < \infty$ for all $T \geq 0$.
A function $f \colon \R_+ \times \R \to \C$ is \textbf{Lipschitz in space, locally uniformly in time} if for all $T \geq 0$, there exists a $c_T < \infty$ such that $|f(t,\lambda)-f(t,\mu)| \leq c_T|\lambda-\mu|$ for all $\lambda,\mu \in \R$ and all $t \in [0,T]$.}
If $-1 \leq q < 1$, $X \colon \R_+ \to \cA_{\sa}$ is a $q$-Brownian motion, $2 \leq p < \infty$, and $y_0 \in L^p(\cA_0,\E_{\scA})_{\sa}$, then there exists a unique global solution $Y \colon \R_+ \to L^p(\E_{\scA})_{\sa}$ to the NCSDE
\[
\begin{cases}
    \displaystyle\d Y(t) = \sum_{i=1}^k \left( f_i(t,Y(t))\,\d X(t)\, g_i(t,Y(t)) + \overline{g_i}(t,Y(t))\,\d X(t)\, \overline{f_i}(t,Y(t))\right) +  h(t,Y(t))\,\d t & \\
    \;\, Y(0) = y_0.
\end{cases}
\]
\end{corollary}

We explain in Example \ref{ex.funkycalc} how to obtain Corollary \ref{cor.funkycalcqSDE} from Theorem \ref{thm.SDEqBM}.

A few remarks are in order about Theorem \ref{thm.SDEqBM} and Corollary \ref{cor.funkycalcqSDE}.
Throughout, we assume $-1 \leq q < 1$.
First, the reason for the finiteness of $p$ in these two results is that it is presently not known how to control the $L^{\infty}$ norm of the stochastic integral against $q$-Brownian motion in the style of, e.g., inequality \eqref{eq.LinftybdFSI} when $q \neq 0$;
we use the noncommutative Burkholder--Davis--Gundy inequalities (Theorem \ref{thm.NCBDG} below) to control its $L^p$ norm when $p$ is finite.
If such $L^{\infty}$-norm control were possible, we would also be able to prove a result like Theorem \ref{thm.SDESCBMdriver} for $q$-Brownian motion.
Second, observe that Corollary \ref{cor.funkycalcqSDE} is only a global result while Corollary \ref{cor.funkycalcFSDE} contains both a local \emph{and} a global result.
The reason is that the Lipschitz condition in Theorem \ref{thm.SDEqBM} is enforced on balls in $L^p$ while that in Theorem \ref{thm.SDESCBMdriver} is enforced on balls in $L^{\infty}$, i.e., operator-norm balls.
Local assumptions on the scalar coefficient functions do not transfer conveniently to balls in $L^p$ the way they transfer to operator-norm balls.
(Ultimately, this too becomes an issue only because of the present lack of $L^{\infty}$-norm control of the stochastic integral against $q$-Brownian motion.)
Third, Theorem \ref{thm.SDEqBM} is the first general result about the existence and uniqueness of solutions to noncommutative It\^o-type NCSDEs driven by $q$-Brownian motion.
Previously, in \cite{DS2013,DS2018}, A.\ Deya and R.\ Schott studied rough differential equations (RDEs) driven by $q$-Brownian motion with $0 \leq q < 1$.
Deya and Schott's results were in $L^{\infty}$ but demanded the aforementioned operator Lipschitz--type assumptions on coefficients arising from the functional calculus.
As of quite recently, there is also the work \cite{CHP2025} of A.\ Chandra, M.\ Hairer, and M.\ Peev, which establishes, among other things, a local existence and uniqueness result for noncommutative NCSDEs---formulated in terms of regularity structures---driven by $q$-Brownian motion.
Chandra, Hairer, and Peev's result is also in $L^{\infty}$ but enforces a real-analyticity assumption on the coefficients.

\section{Noncommutative stochastic integration}\label{sec.NCstochint}

Let $(\cA,\E)$ be a $\mathrm{C}^*$-probability space.
A unital $\mathrm{C}^*$-subalgebra $\cB \subseteq \cA$ is \textbf{conditionable} if $\E[a \mid \cB] \in \cB$ for all $a \in \cA$;
please see \cite[Prop.\ 2.5]{JKN2026} and the surrounding discussion for information about the conditional expectation $\E[\cdot \mid \cB]$.
A filtered $\mathrm{C}^*$-probability space $(\cA,(\cA_t)_{t \geq 0},\E)$ is \textbf{conditionable} if $\cA_t$ is conditionable for each $t \geq 0$.
All $\mathrm{W}^*$-subalgebras of $\mathrm{W}^*$-probability spaces are conditionable, so filtered $\mathrm{W}^*$-probability spaces are conditionable.

For the duration of this section, let $(\cA,(\cA_t)_{t \geq 0},\E = \E_{\scA})$ and $(\cB,(\cB_t)_{t \geq 0},\E_{\scB})$ be conditionable filtered $\mathrm{C}^*$-probability spaces.
Though many of our results can be adjusted slightly to hold in the non-conditionable case, we avoid these adjustments to ease the exposition.

\subsection{Predictable linear processes and noncommutative stochastic integrals}\label{subsec.generalities}

For the duration of this subsection, suppose $1 \leq p_0,p,q \leq \infty$, and let $(\G,\norm{\cdot})$ be a normed filtration of $B^{p_0;1}$.
In Definitions \ref{def.predlin} and \ref{def.NCintegrator}, we introduced the notion of $\G$-adapted linear processes, the space $\cI_{\G}$ of $\G$-predictable linear processes, and the notion of a noncommutative $L^p$-integrator.
Here, we explore some general properties and examples of $\G$-predictable linear processes and noncommutative stochastic integrals.
We also formulate a criterion (Assumption \ref{ass.integ}) for an $L^{p_0}$ process to be an $L^p$-integrator with respect to a given normed filtration of $B^{p_0;1}$.
This sufficient condition is both common, as we shall see in subsections \ref{subsec.MD} and \ref{subsec.fBMintegrator}, and important for our results on noncommutative stochastic integral equations.

We begin by showing that (1) left-continuous, locally bounded $\G$-adapted linear processes are $\G$-predictable, and (2) while not always left-continuous, $\G$-predictable linear processes are always $\G$-adapted and strongly measurable.\footnote{Let $(\Om,\sF)$ be a measurable space and $\cV$ be a (real) normed vector space.
A function $f \colon \Om \to \cV$ is \textbf{strongly} (or \textbf{Bochner}) \textbf{measurable} if it satisfies any of the following equivalent conditions:
(i) There exists a sequence $(s_n)_{n \in \N}$ of $\sF$-simple functions from $\Om$ to $\cV$ converging pointwise to $f$;
(ii) $f$ is $\sF$/$\cB_{\cV}$-measurable, and $f(\Om) \subseteq \cV$ is separable;
and (iii) $\ell \circ f \colon \Om \to \R$ is $\sF$/$\cB_{\R}$ for all $\ell \in \cV^*$, and $f(\Om) \subseteq \cV$ is separable.
Please see \cite[\S{A.1}]{NikitopoulosDissertation} for a proof of this equivalence (or \cite[app.\ E]{Cohn2013} when $\cV$ is a Banach space).}

\begin{proposition}[Adapted \& LCLB $\Rightarrow$ predictable]\label{prop.LCLB}
If $H \colon \R_+ \to \cG_{\infty}$ is $\G$-adapted and ($\norm{\cdot}$-)\textbf{LCLB}, i.e., \textbf{l}eft-\textbf{c}ontinuous and \textbf{l}ocally \textbf{b}ounded (with respect to $\norm{\cdot}$), then $H \in \cI_{\G}$.
In particular, $\cI_{\G}$ is the smallest subset of $\ell_{\loc}^{\infty}(\R_+;\cG_{\infty})$ that is closed under locally bounded convergence and contains the LCLB $\G$-adapted linear processes.
\end{proposition}

\begin{proof}
Suppose $H$ is $\G$-adapted and LCLB.
For each $n \in \N$, define
\[
H_n \coloneqq 1_{\{0\}}H(0) + \sum_{k = 0}^{\infty} 1_{(d_k^n,d_{k+1}^n]} H(d_k^n),
\]
where $d_k^n \coloneqq k/2^n$ for all $k \in \N_0$.
Since $H$ is $\G$-adapted, $1_{[0,t]}H_n \in \EP_{\G}$ for all $t \geq 0$ and $n \in \N$.
Also, $(1_{[0,N]}H_n)_{N \in \N}$ converges locally boundedly to $H_n$ because $H_n$ is locally bounded.
Thus, $H_n \in \cI_{\G}$.
Finally, since $H$ is LCLB, $(H_n)_{n \in \N}$ converges locally boundedly to $H$ by \cite[Lem.\ 4.4(i)]{JKN2026}.
Thus, $H \in \cI_{\G}$.
\end{proof}

\begin{notation}\label{nota.Gminus}
For each $t > 0$, write $\cG_{t-}$ to the closure (with respect to $\norm{\cdot}$) of $\bigcup_{0 \leq s < t} \cG_s$ in $\cG_t$ for $t > 0$.
Also, write $\cG_{0-} \coloneqq \cG_0$.
Finally, define $\G_- \coloneqq (\cG_{t-})_{t \geq 0}$.
\end{notation}

\begin{proposition}[Adaptedness \& strong measurability of predictable linear processes]\label{prop.GpredisGadaptandmeas}
If $H \colon \R_+ \to \cG_{\infty}$ is a $\G$-predictable linear process, then $H$ is $\G_-$-adapted and strongly measurable (with respect to $\norm{\cdot}$).
Moreover, $\cI_{\G} \subseteq \ell_{\loc}^{\infty}(\R_+;\cG_{\infty})$ is a complex-linear subspace.
\end{proposition}

\begin{proof}
Let $\cS \coloneqq \{H \in \cI_{\G} : H$ is $\G_-$-adapted and strongly measurable$\}$.
Since $\cG_{t-} \subseteq \cG_{\infty}$ is closed for all $t \geq 0$, the property of $\G_-$-adaptedness is stable under pointwise limits.
Also, the property of strong measurability is stable under pointwise limits of sequences.
Consequently, $\cS$ is closed under locally bounded convergence.
Since $\EP_{\G} \subseteq \cS$, we obtain the containment $\cI_{\G} \subseteq \cS$ from the definition of $\cI_{\G}$.
\pagebreak

To prove that $\cI_{\G}$ is a complex-linear subspace of $\cS$, let $H \in \cI_{\G}$, and define $\cI^H$ to be the set of $K \in \cI_{\G}$ such that $aH+bK \in \cI_{\G}$ for all $a,b \in \C$.
Suppose $(K_n)_{n \in \N}$ is a sequence in $\cI^H$ converging locally boundedly to $K$.
If $a,b \in \C$ and $n \in \N$, then $aH+bK_n \in \cI_{\G}$ by definition of $\cI^H$.
Since $(aH+bK_n)_{n \in \N}$ converges locally boundedly to $aH+bK$ and $\cI_{\G}$ is closed under locally bounded convergence, $aH+bK \in \cI_{\G}$, i.e., $K \in \cI^H$.
Thus, $\cI^H$ is closed under locally bounded convergence.
Next, note that if $H \in \cI_{\G}$ is such that $\EP_{\G} \subseteq \cI^H$, then $\cI^H = \cI_{\G}$ by the previous sentence and the definition of $\cI_{\G}$.
In particular, if $H \in \EP_{\G}$, then $\cI^H = \cI_{\G}$ because the containment $\EP_{\G} \subseteq \cI^H$ is obvious in this case.
For general $H \in \cI_{\G}$, due to the symmetry in the definition of $\cI^H$, we conclude from the previous sentence that $\EP_{\G} \subseteq \cI^H$.
Thus, $\cI^H = \cI_{\G}$ for all $H \in \cI_{\G}$, which completes the proof.
\end{proof}

Recall that $(L^p;L^q)$-processes were defined in Example \ref{ex.LpLq}.
Adapted or predictable $(L^p;L^q)$-processes frequently arise ``in the wild'' from the application of trace polynomials to adapted processes.
Such processes are adapted to a normed filtration smaller than $\F^{p;q}$, which we introduce now.

\begin{notation}\label{nota.trpolyfiltration}
Write $p'$ for the H\"older conjugate of $p$, i.e., $1/p+1/p'=1$.
\begin{enumerate}[font=\normalfont,label=(\roman*)]
    \item For a real-linear map $T \colon \cA \to \cB$, let
    \[
    \vertiii{T} \coloneqq \sup_{1 \leq p \leq \infty} \norm{T}_{p;p} = \sup_{1 \leq p \leq \infty} \sup_{\underset{\norm{x}_p \leq 1}{x \in \cA,}} \norm{Tx}_p \in [0,\infty],
    \]
    and write $\B(\cA;\cB)$ for the space of real-linear maps $T \colon \cA \to \cB$ such that $\vertiii{T} < \infty$.
    We shall occasionally abuse notation and view elements of $\B(\cA;\cB)$ as bounded real-linear maps from $L^p(\E_{\scA})$ to $L^q(\E_{\scB})$ whenever $p \geq q$.
    Also, $\B(\cA) \coloneqq \B(\cA;\cA)$.
    \item If $t \geq 0$, then $\cF_t = \cF_t(\E_{\scA};\E_{\scB}) \coloneqq \B(\cA;\cB) \cap \cF_t^{\infty;\infty}$, where $\cF_t^{p;q} = \cF_t^{p;q}(\E_{\scA};\E_{\scB})$ is as in Example \ref{ex.LpLq}.
    Also, $\F = \F(\E_{\scA};\E_{\scB}) \coloneqq (\cF_t)_{t \geq 0}$.
    \item Assume that $(\cA,(\cA_t)_{t \geq 0},\E) = (\cB,(\cB_t)_{t \geq 0},\E_{\scB})$ and $p \geq q$.
    For $r_1,r_2 \in [1,\infty]$ satisfying $1/r_1+1/p+1/r_2 = 1/q$, $(a,b,c,d) \in L^{r_1}(\E) \times L^{r_2}(\E) \times  L^{p'}(\E) \times L^q(\E)$, and $\e,\delta \in \{1,\ast\}$, define $T_{\e,\delta}(a,b,c,d) \in B^{p;q}$ by
    \[
    T_{\e,\delta}(a,b,c,d)[x] \coloneqq ax^{\e}b + \E\big[cx^{\delta}\big]\,d \in L^q(\E) \qquad (x \in L^p(\E)).
    \]
    Now, for $t \geq 0$, define $\cT_{0,t}^{p;q} \subseteq B^{p;q}$ to be the complex span of
    \begin{align*}
        \bigg\{T_{\e,\delta}(a&,b,c,d) : \e,\delta \in \{1,\ast\}, \; r_1,r_2 \in [1,\infty] \text{ with } \frac1r_1+\frac1p+\frac1r_2 = \frac1q,\\
        & (a,b,c,d) \in L^{r_1}(\cA_t,\E) \times L^{r_2}(\cA_t,\E) \times  L^{p'}(\cA_t,\E) \times L^q(\cA_t,\E)\bigg\}
    \end{align*}
    and $\cT_t^{p;q}$ to be the closure of $\cT_{0,t}^{p;q}$ in $B^{p;q}$.
    Upon observing $\cT_t^{\infty;\infty} \subseteq \B(\cA)$, we also write $\cT_t$ for the closure of $\cT_t^{\infty;\infty}$ in $\B(\cA)$.
    Finally, $\T^{p;q} \coloneqq (\cT_t^{p;q})_{t \geq 0}$, and $\T \coloneqq (\cT_t)_{t \geq 0}$.
\end{enumerate}
\end{notation}

Suppose $p \geq q$.
Note that $(\F,\vertiii{\cdot})$ is a compatible normed filtration of $B^{p;q}$.
Also, since the conditional expectation $\E[\cdot \mid \cA_t]$ is a trace-preserving, $\ast$-respecting bimodule map, $\cT_t^{p;q} \subseteq \cF_t^{p;q}$ for all $t \geq 0$, i.e., $(\T^{p;q},\norm{\cdot}_{p;q})$ is a compatible normed filtration of $B^{p;q}$.
Finally, $\cT_t \subseteq \cF_t$ for all $t \geq 0$, so $(\T,\norm{\cdot})$ is a compatible normed filtration of $B^{p;q}$.

\begin{definition}\label{def.trbip}
Suppose $(\cA,(\cA_t)_{t \geq 0},\E) = (\cB,(\cB_t)_{t \geq 0},\E_{\scB})$ and $p \geq q$.
Write 
\[
\EP_{\cT}^{p;q} \coloneqq \EP_{\T^{p;q}} \; \text{ and } \; \cI_{\cT}^{p;q} \coloneqq \cI_{\T^{p;q}}.
\]
A $\T^{p;q}$-adapted linear process is called an \textbf{$\boldsymbol{(L^p;L^q)}$-trace biprocess},\footnote{The term ``trace biprocess'' is inspired by the term ``biprocess'' used by Biane--Speicher \cite{BS1998} for their stochastic integrands.} and a(n elementary) $\T^{p;q}$-predictable linear process is called a(n \textbf{elementary}) \textbf{predictable $\boldsymbol{(L^p;L^q)}$-trace biprocess}.
We shall shorten $(L^p;L^p)$ to $L^p$ in these terms.
\end{definition}

\begin{example}\label{ex.trbip}
Assume $(\cA,(\cA_t)_{t \geq 0},\E) = (\cB,(\cB_t)_{t \geq 0},\E_{\scB})$ and $p \geq q$.
Here is a motivating example of a trace biprocess.
Let
\[
L^{\infty-}(\E) \coloneqq \bigcap_{1 \leq r < \infty} L^r(\E),
\]
endowed with the Fr\'echet-space topology induced by the family $\{\norm{\cdot}_r : 1 \leq r < \infty\}$ of norms.
If $p > q$ and $X_1,X_2,X_3 \colon \R_+ \to L^{\infty-}(\E)$ are adapted processes, then
\[
H(t)[x] \coloneqq \E\big[X_1(t)^3\big] X_3(t)^*X_1(t)^2x^*X_2(t)^4X_1(t)^3 + \E[X_1(t)X_3(t)x]X_2(t)^*
\]
defines an $(L^p;L^q)$-trace biprocess.
If $X_i(t) \in L^{\infty}(\cA_t,\E) = \cA_t$ for all $t \geq  0$ and $i=1,2,3$, then $H$ is an $L^p$-trace biprocess;
actually, $H$ is $\T$-adapted in this case.

The example in the previous paragraph generalizes significantly.
To state the generalization properly requires the language of trace $\ast$-polynomials;
please see \cite[\S2.3]{JKN2026} for the requisite terminology.
Suppose $P(x_1,\ldots,x_n,y)$ is a trace $\ast$-polynomial that is real-linear in $y$.
If $X_1,\ldots,X_n \colon \R_+ \to L^{\infty-}(\E)$ are adapted and $p > q$, then
\begin{equation}
    H(t) \coloneqq P(X_1(t),\ldots,X_n(t),\cdot) \in B^{p;q} \qquad (t \geq 0)\label{eq.trbip}
\end{equation}
defines an $(L^p;L^q)$-trace biprocess.
If $X_i(t) \in \cA_t$ for all $t \geq 0$ and $i=1,\ldots,n$, then $H$ is an $L^p$-trace biprocess;
actually, $H$ is $\T$-adapted in this case.
Such trace biprocesses were introduced in \cite[Def.\ 3.7]{JKN2026}.

By Proposition \ref{prop.LCLB}, if $p > q$, $X_1,\ldots,X_n \colon \R_+ \to L^{\infty-}(\E)$ are adapted and continuous, and $H$ is as in \eqref{eq.trbip}, then $H \in \cI_{\cT}^{p;q}$.
If $p=q$, $X_1,\ldots,X_n \colon \R_+ \to L^{\infty}(\E) = \cA$ are adapted and continuous, and $H$ is as in \eqref{eq.trbip}, then $H \in \cI_{\T}$.
\end{example}

\begin{remark}[Classical processes as trace biprocesses]
Many classical stochastic processes may be viewed as special kinds of trace biprocesses.
Let $(\Om,\sF,(\sF_t)_{t \geq 0},P)$ be a classical filtered probability space, take
\[
(\cA,(\cA_t)_{t \geq 0},\E) = (\cB,(\cB_t)_{t \geq 0},\E_{\scB}) = (L^{\infty}(\Om,\sF,P),(L^{\infty}(\Om,\sF_t,P))_{t \geq 0},\E_P),
\]
and suppose $p \geq q$.
Since $\cA$ is commutative, the definition of $\cT_t^{p;q}$ simplifies slightly.
Let $J_p \colon L^p(\E) \to L^p(\E)$ be the complex-conjugation operator, $r = r(p,q) \in [1,\infty]$ be such that $1/r+1/p = 1/q$, and $M_y \colon L^p(\E) \to L^q(\E)$ be as in \eqref{eq.Moperator}.
Then $\cT_t^{p;q}$ is the closure in $B^{p;q}$ of
\[
\left\{M_a + M_bJ_p + \E[c(\boldsymbol{\cdot})]\,u + \E[dJ_p(\boldsymbol{\cdot})]\,v : a,b \in L^r(\cA_t,\E), c,d \in L^{p'}(\cA_t,\E), u,v \in L^q(\cA_t,\E)\right\}.
\]
Thus, $\cM_t^{p;q} \subseteq \cT_t^{p;q}$ for all $t \geq 0$, where $\cM_t^{p;q}$ is as in Example \ref{ex.clpred}.
In particular, if $Y \colon \R_+ \times \Om \to \C$ is an adapted $L^r$ stochastic process, then $H_Y$ is an $(L^p;L^q)$-trace biprocess, where $H_Y$ is (the modification class of $Y$) as in \eqref{eq.HY}.
In this way, an adapted $L^r$ stochastic process ``is'' an $(L^p;L^q)$-trace biprocess.
\end{remark}
\pagebreak

We now turn to the noncommutative stochastic integral.

\begin{theorem}[Properties of NC stochastic integral]\label{thm.abstractstochint}
Suppose $(\G,\norm{\cdot})$ is compatible and $X \colon \R_+ \to L^{p_0}(\E_{\scA})$ is an $L^p$-integrator with respect to $\G$.
\begin{enumerate}[font=\normalfont,label=(\roman*)]
    \item The stochastic integral map $I_X = I_X^{\G} \colon \cI_{\G} \to R_{a,p} = R_{a,p}(\E_{\scB})$ from {\rm Definition \ref{def.NCintegrator}} is unique.
    Furthermore, $I_X$ is complex linear, and for all $H \in \cI_{\G}$,\label{item.uniquenessofint}
    \begin{equation}
    I_X[1_{(s,t]}H](T) = I_X[H](t) - I_X[H](s) \qquad (0 \leq s \leq t \leq T).\label{eq.stottoT}
    \end{equation}
    \item If $X$ is an $L^{p_0}$-martingale, i.e., $\E_{\scA}[X(t) \mid \cA_s] = X(s)$ whenever $0 \leq s \leq t$, then $I_X[H]$ is an $L^p$-martingale for all $H \in \cI_{\G}$.\label{item.martingale}
    \item If $\into H[\d X] \in C_{a,p}$ for all $H \in \EP_{\G}$, then $\into H[\d X] \in C_{a,p}$ for all $H \in \cI_{\G}$.\label{item.contint}
\end{enumerate}
\end{theorem}

\begin{proof}
We take each item in turn.

\ref{item.uniquenessofint} If $I_X,J_X \colon \cI_{\G} \to R_{a,p}$ are continuous with respect to locally bounded convergence and satisfy $I_X[H] = \into H[\d X] = J_X[H]$ for all $H \in \EP_{\G}$, then $\cU \coloneqq \{H \in \cI_{\G} : I_X[H] = J_X[H]\}$ contains $\EP_{\G}$ and is closed under locally bounded convergence.
Thus, $\cU = \cI_{\G}$, i.e., $I_X = J_X$.

Next, let $H \in \cI_{\G}$, and define
\[
\cS_H \coloneqq \{K \in \cI_{\G} : \forall c,d \in \C, \, I_X[cH+dK] = cI_X[H]+dI_X[K]\}.
\]
Since $I_X$ is continuous with respect to locally bounded convergence, $\cS_H \subseteq \cI_{\G}$ is closed under locally bounded convergence.
If $H \in \EP_{\G}$, then it is obvious from the complex linearity of $\EP_{\G} \ni H \mapsto \into H[\d X] \in R_{a,p}$ that $\EP_{\G} \subseteq \cS_H$;
consequently, $\cS_H = \cI_{\G}$ in this case.
By squinting at the definition of $\cS_H$, it follows that $\EP_{\G} \subseteq \cS_H$ for all $H \in \cI_{\G}$.
Thus, $\cS_H = \cI_{\G}$ for all $H \in \cI_{\G}$, i.e., $I_X$ is complex linear.

Finally, if $\cS \coloneqq \{H \in \cI_{\G} : \text{\eqref{eq.stottoT} holds}\}$, then $\EP_{\G} \subseteq \cS$ by basic properties of the elementary integral, and $\cS$ is closed under locally bounded convergence because $I_X$ is continuous with respect to locally bounded convergence.
Thus, $\cS = \cI_{\G}$, i.e., \eqref{eq.stottoT} holds for all $H \in \cI_{\G}$.

\ref{item.martingale} Suppose $X$ is also an $L^{p_0}$-martingale, and define
\[
\cM \coloneqq \left\{H \in \cI_{\G} : I_X[H] \text{ is an } L^p\text{-martingale}\right\}.
\]
By \cite[Lem.\ 4.12(ii)]{JKN2026} and the compatibility of $(\G,\norm{\cdot})$, $\EP_{\G} \subseteq \cM$.
Since the pointwise $L^p$ limit of a sequence of $L^p$-martingales is an $L^p$-martingale and $I_X$ is continuous with respect to locally bounded convergence, $\cM$ is closed under locally bounded convergence.
Thus, $\cM = \cI_{\G}$, as desired.

\ref{item.contint} If $\cC \coloneqq \{H \in \cI_{\G} : I_X[H] \in C_{a,p}\}$, then $\cC$ is closed under locally bounded convergence because $I_X$ is continuous with respect to locally bounded convergence and $C_{a,p} \subseteq R_{a,p}$ is a closed subspace.
Consequently, if $\EP_{\G} \subseteq \cC$, then we obtain $\cC = \cI_{\G}$, as desired.
\end{proof}

Owing to \eqref{eq.stottoT}, we write
\[
\int_s^t H[\d X] = \int_s^t H(r)[\d X(r)] \coloneqq \int_0^t H[\d X] - \int_0^s H[\d X] = \int_0^t (1_{(s,t]}H)[\d X]
\]
whenever $H \in \cI_{\G}$ and $0 \leq s \leq t$.

We now prove that if $X$ is an $L^p$-integrator with respect to $\G$ and $H \colon \R_+ \to \cG_{\infty}$ is an LCLB $\G$-adapted linear process, then $I_X[H]$ may be computed as a limit of left-endpoint Riemann--Stieltjes sums.
Owing to Example \ref{ex.MD} (really, subsection \ref{subsec.MD}), this represents a massive generalization of \cite[Prop.\ 4.26]{JKN2026}.
\pagebreak

\begin{notation}[Partitions, etc.]\label{nota.part}
Suppose $-\infty < a  < b \leq \infty$, and write $I \coloneqq [a,b] \cap \R$.
Also, let $\cV$ be a real vector space and $F \colon I \to \cV$ be a function.
\begin{enumerate}[label=(\roman*),font=\normalfont]
    \item If $t \in I$, then $F^t \colon [a,\infty) \to \cV$ is the function $F^t(s) \coloneqq F(s \wedge t)$ ($s \geq a$).\label{item.stoppedfunc}
    \item If $b < \infty$, then a \textbf{partition} of $I$ is a finite subset $\Pi = \{a = t_0 < \cdots < t_n = b\} \subseteq I$.
    A \textbf{partition} of $[a,\infty)$ is a collection $\Pi = \{t_n : n \in \N_0\}$ such that $t_0 = a$, $t_n < t_{n+1}$ for all $n \in \N_0$, and $t_n \to \infty$ as $n \to \infty$.
    Write $\cP_I$ for the set of partitions of $I$.\label{item.part}
    \item Let $\Pi \in \cP_I$.
    If $t \in \Pi$, then $t_- \in \Pi$ is the member of $\Pi$ to the left of $t$;
    precisely, $a_- \coloneqq a$, and $t_- \coloneqq \max\left\{s \in \Pi : s < t \right\}$ for all $t \in \Pi \setminus \{a\}$.
    Also, $\Delta t \coloneqq t-t_-$, $\Delta_tF \coloneqq F(t) - F(t_-)$, and $|\Pi| \coloneqq \sup\left\{ \Delta s : s \in \Pi\right\}$ is the mesh of $\Pi$.\label{item.Delta}
    \item Suppose $\cV$ has a Hausdorff topology.
    If $t \in I$, then $F(t-) \coloneqq \lim_{s \nearrow t} F(s)$, with the convention $F(a-) \coloneqq F(a)$.
    Also, $F_- \colon I \to \cV$ is the function $t \mapsto F(t-)$.\label{item.ll}
\end{enumerate}
\end{notation}

\begin{proposition}\label{prop.IXRS}
Suppose $\G$ is compatible and $X \colon \R_+ \to L^{p_0}(\E_{\scA})$ is an $L^p$-integrator with respect to $\G$.
If $H \colon \R_+ \to \cG_{\infty}$ is $\G$-adapted and ($\norm{\cdot}$-)\textbf{LLLB}, i.e., \textbf{l}eft \textbf{l}imited and \textbf{l}ocally \textbf{b}ounded (with respect to $\norm{\cdot}$), then
\[
\sum_{t \in \Pi} H(t_-)[X(t \wedge \cdot) - X(t_- \wedge \cdot)] \xrightarrow[\Pi \in \cP_{\R_+}]{|\Pi| \to 0}\into H(t-)[\d X(t)]
\]
in the space $R_{a,p} = R_{a,p}(\E_{\scB})$.
In particular, if $t \geq 0$, then
\[
\int_0^t H(s-)[\d X(s)] = L^p\text{-}\lim_{\Pi \in \cP_{\R_+}}\sum_{s \in \Pi} H(s_-)\big[\Delta_s X^t\big] = L^p\text{-}\lim_{\Pi \in \cP_{[0,t]}}\sum_{s \in \Pi} H(s_-)[\Delta_s X].
\]
Note that $H(t-) = H(t)$ for all $t \in I$ if $H$ is left-continuous.
\end{proposition}

\begin{proof}
Let $(\Pi_n)_{n \in \N}$ be a sequence in $\cP_{\R_+}$ such that $|\Pi_n| \to 0$ as $n \to \infty$, and define
\[
H_n \coloneqq 1_{\{0\}} H(0) + \sum_{t \in \Pi_n} 1_{(t_-,t]} H(t_-) \qquad (n \in \N).
\]
If $n,N \in \N$, then $1_{[0,N]}H_n \in \EP_{\G}$ because $H$ is $\G$-adapted.
Also, if $n \in \N$, then $(1_{[0,N]}H_n)_{N \in \N}$ converges locally boundedly to $H_n$ because $H_n$ is locally bounded.
Thus, $H_n \in \cI_{\G}$.
Now, by \cite[Lem.\ 4.4(i)]{JKN2026}, $H_n \to H_-$ locally boundedly as $n \to \infty$.
Consequently, $H_- \in \cI_{\G}$, and $I_X[H_n] \to I_X[H_-]$ in $R_{a,p}$ as $n \to \infty$ by the continuity of $I_X$ with respect to locally bounded convergence.
Since it is easy to see that
\[
I_X[H_n] = \sum_{t \in \Pi_n} H(t_-)[X(t \wedge \cdot) - X(t_- \wedge \cdot)] \qquad (n \in \N),
\]
the result follows from the standard sequential characterization of convergence as $|\Pi| \to 0$ (vid.\ \cite[Fact 4.2]{JKN2026}).
\end{proof}

We end this subsection with a sufficient condition for a noncommutative stochastic process to be an $L^p$-integrator.
Suppose $X \colon \R_+ \to L^{p_0}(\E_{\scA})$ is an $L^{p_0}$ process and $(\G,\norm{\cdot})$ is compatible.
If it happens to be that $\into H[\d X] \in R_{a,p} = R_{a,p}(\E_{\scB})$ for all $H \in \EP_{\G}$, then we shall write
\[
\EP_{\G} \ni H \mapsto I_X^0[H] \coloneqq \into H[\d X] \in R_{a,p}
\]
for the \textbf{elementary stochastic integral against $\boldsymbol{X}$}.
\pagebreak

Note that ``$\into H[\d X] \in R_{a,p}$ for all $H \in \EP_{\G}$,'' is not such a far-fetched condition.
Indeed, if $(\G,\norm{\cdot})$ is a compatible normed filtration of $B^{p_0;p}$---not just of $B^{p_0;1}$---and if $X \in R_{a,p_0}(\E_{\scA})$, then $\into H[\d X] \in R_{a,p}(\E_{\scB})$ for all $H \in \EP_{\G}$, as the reader should verify.
If, in addition, $X \in C_{a,p_0}(\E_{\scA})$, then $I_X[H] \in C_{a,p}(\E_{\scB})$ for all $H \in \EP_{\G}$.

\begin{assumption}\label{ass.integ}
Assume $\G$ is compatible and $X \colon \R_+ \to L^{p_0}(\E_{\scA})$ satisfies:
\begin{enumerate}[font=\normalfont,label=(\roman*)]
    \item $X \colon \R_+ \to L^{p_0}(\E_{\scA})$ is adapted;
    \item $\displaystyle \int_0^t H(s)[\d X(s)] \in L^p(\E_{\scB})$ for all $H \in \EP_{\G}$ and $t \geq 0$;
    \item there exists a locally finite Borel measure $\rho$ on $[0,\infty)$ and a family $(a_T)_{T \geq 0}$ of finite, positive constants such that for all $H \in \EP_{\G}$,\label{item.rhoass}
    \[
    \sup_{0 \leq s \leq t} \norm{\int_0^s H(r)[\d X(r)]}_p \leq a_T\left(\int_{(0,t]} \norm{H(s)}^2 \,\rho(\d s) \right)^\frac12 \qquad (0 \leq t \leq T).
    \]
\end{enumerate}
We shall occasionally write $\int_s^t \boldsymbol{\cdot} \,\d\rho \coloneqq \int_{(s,t]} \boldsymbol{\cdot} \,\d\rho$ ($0 \leq s \leq t$).
\end{assumption}

\begin{remark}\label{rem.Assintegstuff}
Let $H \in \EP_{\G}$.
Since $\G$ is compatible, the first two items in Assumption \ref{ass.integ} guarantee, via \cite[Lem.\ 4.12(i)]{JKN2026}, that $\into H[\d X] \colon \R_+ \to L^p(\E_{\scB})$ is adapted.
Also, by replacing $H$ with $1_{(r,\infty)}H$, the third item in Assumption \ref{ass.integ} may be upgraded to 
\[
\sup_{r \leq s \leq t} \norm{\int_r^s H[\d X]}_p \leq a_T\left(\int_{(r,t]} \norm{H}^2 \,\d\rho \right)^\frac12 \qquad (0 \leq r \leq t \leq T).
\]
Since $\int_r^t H[\d X] = \int_0^t H[\d X] - \int_0^r H[\d X]$, it follows that $\into H[\d X] \in R_{a,p}$.
If $\rho$ were atomless, then we would have $\into H[\d X] \in C_{a,p}$.
\end{remark}

\begin{theorem}[Construction of the NC stochastic integral]\label{thm.constructNCstochint}
Under {\rm Assumption \ref{ass.integ}}, $X$ is an $L^p$-integrator with respect to $\G$ with the property that
\[
\sup_{0 \leq s \leq t} \norm{I_X[H](s)}_p \leq a_T \left(\int_0^t \norm{H}^2\,\d\rho\right)^\frac12 \qquad (0 \leq t \leq T)
\]
for all $H \in \cI_{\G}$ and $T \geq 0$.
\end{theorem}

To prove Theorem \ref{thm.constructNCstochint}, we extend the elementary stochastic integral against $X$ to a map defined on a $\rho$-dependent space containing $\cI_{\G}$.

\begin{notation}\label{nota.IGrho}
Let $\overline{\cG}_{\infty}$ be the $\norm{\cdot}$-completion of $\cG_{\infty}$.
(Recall that $\cG_{\infty}$ is not assumed~to~be complete.)
Write $\cL_{\G,\rho}$ for the set of strongly measurable functions $H \colon \R_+ \to \overline{\cG}_{\infty}$ such that
\[
\norm{H}_{\rho,t} \coloneqq \left(\int_{(0,t]} \norm{H}^2 \,\d\rho\right)^{\frac{1}{2}} < \infty \qquad (t \geq 0).
\]
Endow $\cL_{\G,\rho}$ with the locally convex topology generated by the collection $\{\norm{\cdot}_{\rho,t} : t \geq 0\}$ of seminorms, and write $\cI_{\G,\rho}$ for the closure of $\EP_{\G}$ in $\cL_{\G,\rho}$.
\end{notation}

\begin{proposition}\label{prop.IGrhoextension}
Under {\rm Assumption \ref{ass.integ}}, $I_X^0$ extends uniquely to a continuous complex-linear map $I_X^{\G,\rho} \colon \cI_{\G,\rho} \to R_{a,p}(\E_{\scB})$.
Furthermore, for all $H \in \cI_{\G,\rho}$ and $T \geq 0$,
\begin{equation}
    \sup_{0 \leq s \leq t}\norm{I_X^{\G,\rho}[H](s)}_p  \leq a_T\norm{H}_{\rho,t} \qquad (0 \leq t \leq T).\label{eq.IGrhobound}
\end{equation}
\end{proposition}

\begin{proof}
Since $\EP_{\G}$ is dense in $\cI_{\G,\rho}$ and $R_{a,p}(\E_{\scB})$ is Hausdorff, there is at most one continuous extension $\cI_{\G,\rho} \to R_{a,p}$ of $I_X^0$.
To construct the extension, let
\[
\cN \coloneqq \big\{H \in \cI_{\G,\rho} : \norm{H}_{\rho,t} = 0 \; \forall t \geq0\big\}, \; \tilde{\cI}_{\G,\rho} \coloneqq \cI_{\G,\rho}/\cN, \; \text{ and } \; \EP_{\G,\rho} \coloneqq \EP_{\G}/\cN.
\]
By Assumption \ref{ass.integ}\ref{item.rhoass}, $I_X^0 \colon \EP_{\G} \to R_{a,p}$ descends to a continuous complex-linear map $J \colon \EP_{\G,\rho} \to R_{a,p}$.
Since $\tilde{\cI}_{\G,\rho}$ and $R_{a,p}$ are Hausdorff, $R_{a,p}$ is complete, and $\EP_{\G,\rho}$ is dense in $\tilde{\cI}_{\G,\rho}$, $J$ extends to a continuous complex-linear map $\tilde{J} \colon \tilde{\cI}_{\G,\rho} \to R_{a,p}$ by \cite[Thm.\ 5.1]{Treves1967}.
If $\pi \colon \cI_{\G,\rho} \to \tilde{\cI}_{\G,\rho}$ is the natural projection, then $I_X^{\G,\rho} \coloneqq \tilde{J} \circ \pi$ is the desired extension.
Inequality \eqref{eq.IGrhobound} follows easily from Assumption \ref{ass.integ}\ref{item.rhoass} and our construction of $I_X^{\G,\rho}$.
\end{proof}

\begin{remark}\label{rem.IGrhoLagniappe}
The reader may find it useful to note that $I_X^{\G,\rho}$ also satisfies the extra properties from Theorem \ref{thm.abstractstochint}.
Specifically, if $H \in \cI_{\G,\rho}$, then
\[
I_X^{\G,\rho}[1_{(s,t]}H](T) = I_X^{\G,\rho}[H](t) - I_X^{\G,\rho}[H](s) \qquad (0 \leq s \leq t \leq T),
\]
$I_X^{\G,\rho}[H]$ is an $L^p$-martingale whenever $X$ is an $L^{p_0}$-martingale, and $I_X^{\G,\rho}[H] \in C_{a,p}$ whenever $\into K[\d X] \in C_{a,p}$ for all $K \in \EP_{\G}$.
Also, it is worth noting that Proposition \ref{prop.IGrhoextension} is a generalization of \cite[Thm.\ 4.19(i)]{JKN2026}.
\end{remark}

\begin{proof}[Proof of Theorem \ref{thm.constructNCstochint}]
First, observe that $\cI_{\G,\rho}$ is closed under locally bounded convergence by the dominated convergence theorem.
Since $\EP_{\G} \subseteq \cI_{\G,\rho}$ by definition, we conclude that $\cI_{\G} \subseteq \cI_{\G,\rho}$.
Define $I_X[H] \coloneqq I_X^{\G,\rho}[H]$ for all $H \in \cI_{\G}$.
By \eqref{eq.IGrhobound} and the dominated convergence theorem, $I_X$ is continuous with respect to locally bounded convergence.
Thus, $X$ is an $L^p$-integrator with respect to $\G$.
The claim bound on $I_X$ follows immediately from \eqref{eq.IGrhobound} and the fact that $I_X = I_X^{\G,\rho}|_{\cI_{\G}}$.
\end{proof}

\subsection{Measured decomposable processes}\label{subsec.MD}

We now introduce and study a large class of noncommutative integrators satisfying Assumption \ref{ass.integ} with $(\G,\norm{\cdot}) = (\F^{p_0;p},\norm{\cdot}_{p_0;p})$:
measured martingales and decomposable processes.
(Recall that $\F^{p_0;p}$ was defined in Example \ref{ex.LpLq}.)
We begin by recording some material related to functions of bounded $\alpha$-variation.

\begin{definition}[Bounded $\alpha$-variation]\label{def.alphavar}
Fix $a \in \R$ and $b \in \R \cup \{\infty\}$ such that $a \leq b$, $\alpha \in [1,\infty)$, and a normed vector space $\cV$.
Let $I \coloneqq [a,b] \cap \R$ and $F \colon I \to \cV$ be a function.
If $s,t \in I$ and $s \leq t$, then
\[
V_{\alpha}(F : [s,t]) = V_{\alpha,\cV}(F : [s,t]) \coloneqq \sup_{\Pi \in \cP_{[s,t]}}\Bigg(\sum_{r \in \Pi}\norm{\Delta_rF}_{\cV}^{\alpha}\Bigg)^{\frac1\alpha} \in [0,\infty].
\]
If $b = \infty$, then $V_{\alpha}(F : I) \coloneqq \sup\left\{V_{\alpha}(F : [a,c]) : c \in I \right\}$.
If $V_{\alpha}(F:I) < \infty$, then $F$ has \textbf{bounded $\boldsymbol{\alpha}$-variation};
if $V_{\alpha}(F : [a,c]) < \infty$ for all $c \in I$, then $F$ has \textbf{locally bounded $\boldsymbol{\alpha}$-variation}.
We often drop $\alpha$ in the notation and terminology when $\alpha=1$.
\end{definition}

A function of locally bounded $\alpha$-variation comes with a measure that we shall use in a crucial way later.

\begin{lemma}\label{lem.control}
Retain the setup of {\rm Definition \ref{def.alphavar}}.
If $F \colon I \to \cV$ is right-continuous and has locally bounded $\alpha$-variation, then there exists a unique Borel measure $\mu$ on $I$ such that
\[
\mu(\{a\}) = 0 \; \text{ and } \; \mu((s,t]) = V_{\alpha}^{\alpha}(F : [a,t]) - V_{\alpha}^{\alpha}(F : [a,s]) \qquad (s,t \in I, \; s < t),
\]
where $V_{\alpha}^{\alpha}(F : [a,t]) \coloneqq V_{\alpha}(F : [a,t])^{\alpha}$.
The measure $\mu$ satisfies 
\[
\norm{F(t) - F(s)}_{\cV}^{\alpha} \leq \mu((s,t]) \qquad (s,t \in I, \; s \leq t)
\]
and is atomless if $F$ is continuous.
\end{lemma}

\begin{proof}
Define $f(t) \coloneqq V_{\alpha}^{\alpha}(F : [a,t]) = V_{\alpha}(F : [a,t])^{\alpha}$ for all $t \in I$.
Clearly, $f \colon I \to \R_+$ is increasing.
Since $F$ is right-continuous, $f$ is right-continuous (and $f$ is continuous if $F$ is), as can be seen by examining the proof of \cite[Prop.\ 5.8]{FV2010}.
The existence and uniqueness of $\mu$, as well as the fact that $\mu$ is atomless if $F$ is continuous, follow.
Finally, if $s,t \in I$ and $s < t$, then
\begin{align*}
    f(s) + \norm{F(t) - F(s)}_{\cV}^{\alpha} & = V_{\alpha}^{\alpha}(F : [a,s]) + \norm{F(t) - F(s)}_{\cV}^{\alpha} \\
    & \leq  V_{\alpha}^{\alpha}(F : [a,s]) +  V_{\alpha}^{\alpha}(F : [s,t]) \\
    & \leq  V_{\alpha}^{\alpha}(F : [a,t]) = f(t),
\end{align*}
which rearranges to the desired inequality.
\end{proof}

\begin{example}\label{ex.alpha=1}
If $\alpha=1$, then $\mu$ is determined by $\mu(\{a\}) = 0$ and $\mu((s,t]) = V(F : [s,t])$ because the $1$-variation is additive:
$V(F : [s,u]) = V(F : [s,t]) + V(F : [t,u])$ whenever $a \leq s \leq t \leq u \in I$.
In this case, we write $\norm{\d F(t)}_{\cV} \coloneqq \mu(\d t)$.
\end{example}

Next, we review some terminology and notation from \cite[\S3]{JKN2026}.
Suppose $1 \leq p,q \leq \infty$.
A process $X \colon \R_+ \to L^p(\E)$ is \textbf{$\boldsymbol{L^p}$-FV} if it is adapted and has locally bounded variation with respect to $\norm{\cdot}_p$.
Write $\M^p = \M_{\scA}^p \subseteq C_{a,p}(\E)$ for the closed subspace of continuous $L^p$-martingales and $\FV^p = \FV_{\scA}^p$ for the space of continuous $L^p$-FV processes with the topology induced by the collection
\[
\{X \mapsto \norm{X(0)}_p + V_{1,L^p}(X : [0,t]) : t \geq 0\}
\]
of seminorms.
(Note that the inclusion of $\FV^p$ into $C_{a,p}(\E)$ is continuous.)
Finally, a process $X \colon \R_+ \to L^{p \wedge q}(\E)$ is \textbf{$\boldsymbol{(L^p,L^q)}$-decomposable} if $X = X(0)+ M+A$ for some $M \in \M^p$ and some $A \in \FV^q$ with $M(0) = A(0)=0$.
By \cite[Cor.\ 3.22]{JKN2026}, if $1/p+1/q \leq 1$, then the decomposition $X = X(0) + M+A$ is unique;
in this case, we call $X^{\mathrm{m}} \coloneqq M$ the \textbf{martingale part} of $X$ and $X^{\mathrm{fv}} \coloneqq A$ the \textbf{FV part} of $X$.

\begin{definition}\label{def.measured}
Suppose $1 \leq p,q \leq \infty$.
A process $M \colon \R_+ \to L^p(\E)$ is an \textbf{$\boldsymbol{L^p}$-measured martingale} if $M \in \M^p$ and $V_{2,L^p}(M : [0,t]) < \infty$ for all $t \geq 0$.
An \textbf{$\boldsymbol{(L^p,L^q)}$-measured decomposable process} is a process $X \colon \R_+ \to L^{p \wedge q}(\E)$ such that there exists an $L^p$-measured martingale $M$ with $M(0) = 0$ and $X-X(0)-M \in \FV^q$.
We shorten ``$(L^p,L^p)$-'' to ``$L^p$-'' in this term.
\end{definition}

When $1/p+1/q \leq 1$, an $(L^p,L^q)$-decomposable process is $(L^p,L^q)$-measured if and only if its martingale part is an $L^p$-measured martingale.
\pagebreak

\begin{example}[Measured is automatic when $p=2$]\label{ex.2measured}
If $M$ is an $L^2$-martingale, then
\[
V_{2,L^2}(M : [s,t]) = \norm{M(t)-M(s)}_2 < \infty \qquad ( 0 \leq s < t)
\]
by \cite[Lem.\ 3.20]{JKN2026}.
Thus, every continuous $L^2$-martingale is an $L^2$-measured martingale, and every $L^2$-decomposable process is an $L^2$-measured decomposable process.
Also, we note for later use that if we take $(r,\cV) = (2,L^2(\E))$ and $F = M \colon \R_+ \to L^2(\E)$ to be a right-continuous $L^2$-martingale in Lemma \ref{lem.control}, then the measure $\mu$ is determined by
\[
\mu(\{0\}) = 0\; \text{ and } \; \mu((s,t]) = \norm{M(t) - M(s)}_2^2 = \norm{M(t)}_2^2 - \norm{M(s)}_2^2 \qquad (0 \leq s < t).
\]
In \cite{JKN2026}, this measure $\mu$ was denoted by $\kappa_M$ and was used in an essential way to construct the $L^2$-valued stochastic integral against an $L^2$-decomposable process.
\end{example}

\begin{example}[Classical Brownian motion]
Let $(\Om,\sF,(\sF_t)_{t \geq 0},P)$ be a classical filtered probability space and $(\cA,(\cA_t)_{t \geq 0},\E)$ be as in \eqref{eq.classicalfWps}.
If $X \colon \R_+ \times \Om \to \R$ is a classical Brownian motion and $1 \leq p < \infty$, then the noncommutative process $\R_+ \ni t \mapsto x(t) \coloneqq X_t \in L^p(\E)$ is an $L^p$-martingale.
We claim that $x$ is an $L^p$-measured martingale.
Indeed, if $0 \leq s  < t$, then $X_t  - X_s \sim N(0,t-s)$, so
\[
\norm{x(t) - x(s)}_{L^p(\E)} = \norm{X_t - X_s}_{L^p(\E_P)} = \norm{N(0,1)}_{L^p(P)}\sqrt{t-s} \qquad (1 \leq p < \infty).
\]
Consequently, if $1 \leq p < \infty$, then $x$ is $L^p$-continuous, and
\[
V_{2,L^p}(x : [0,t]) \leq \norm{N(0,1)}_{L^p(P)}\sqrt{t} < \infty \qquad (t \geq 0).
\]
In particular, $x$ is an $L^p$-measured martingale, as claimed.
\end{example}

\begin{example}[$q$-Brownian motion]\label{ex.qBm}
If $X \colon \R_+ \to \cA_{\sa}$ is a free Brownian motion (Definition \ref{def.ndimfBM}), then $X$ is a martingale, and $\norm{X(t) - X(s)}_{\infty}^2 = 4|t-s|$ for all $s, t \geq 0$.
In particular, $X$ is $L^{\infty}$-continuous, and $V_{2,L^{\infty}}(X : [0,t]) = 2\sqrt{t} < \infty$ for all $t \geq 0$.
Thus, $X$ is an $L^{\infty}$-measured martingale.
More generally, if $-1 \leq q < 1$ and $X \colon \R_+ \to \cA_{\sa}$ is a $q$-Brownian motion (vid.\ \cite{BKS1997,DonatiMartin2003,DS2013,DS2018}), then $X$ is a martingale, and
\[
\norm{X(t) - X(s)}_{\infty} = 2\sqrt{\frac{|t-s|}{1-q}} \implies V_{2,L^{\infty}}(X : [0,t]) = 2\sqrt{\frac{t}{1-q}} < \infty \qquad (s,t \geq 0).
\]
Thus, $X$ is an $L^{\infty}$-measured martingale.
\end{example}

\begin{example}[$q$-Hermite martingales]\label{ex.qHermitemart}
Suppose $- 1 \leq q < 1$ and $X \colon \R_+ \to \cA_{\sa}$ is a $q$-Brownian motion.
If $n \in \N_0$, $H_n^{(q)}$ is the $n^{\text{th}}$ $q$-Hermite polynomial (vid.\ \cite[Def.\ 1.9]{BKS1997}), and $M_n \colon \R_+ \to \cA_{\sa}$ is defined by $M_0 \equiv 1$ and for $n \in \N$, 
\[
M_n(0) \coloneqq 0 \; \text{ and } \; M_n(t) \coloneqq t^{\frac{n}{2}}H_n^{(q)}\left(t^{-\frac12}X(t)\right) \qquad (t > 0),
\]
then $M_n$ is a martingale (vid.\ \cite[Cor.\ 4.7]{BKS1997}).
We claim that $M_n$ is an $L^{\infty}$-measured martingale.
To see this, we prove by induction on $n$ that if $T \geq 0$, then there exists a constant $C_{n,T}$ such that
\begin{equation}
    \norm{M_n(t) - M_n(s)}_{\infty}^2 \leq C_{n,T}|t-s| \qquad (0 \leq s, t \leq T).\label{eq.Mnmeasured}
\end{equation}
Since $H_0^{(q)}(x) = 1$, the $n=0$ case is obvious;
since $H_1^{(q)}(x) = x$, the $n=1$ case was the content of Example \ref{ex.qBm}.
Now, let $n \geq 1$, and assume the desired conclusion for $M_n$ and $M_{n-1}$.
By definition of the $q$-Hermite polynomials,
\[
H_{n+1}^{(q)}(x) = xH_n^{(q)}(x) - [n]_qH_{n-1}^{(q)}(x),
\]
where $[n]_q = (1-q^n)/(1-q)$.
This translates into the relation
\[
M_{n+1}(t) = X(t)\,M_n(t) - [n]_q\,t\,M_{n-1}(t) \qquad (t \geq 0).
\]
Consequently, if $0 \leq s,t \leq T$ and $a_{i,T} \coloneqq \sup\left\{ \norm{M_i(r)}_{\infty} : 0 \leq r \leq T\right\}$ ($i=0,\ldots,n$), then
\begin{align*}
    \|M_{n+1}(t) - M_{n+1} &(s)\|_{\infty} \leq \norm{M_n(t)}_{\infty}\norm{X(t) - X(s)}_{\infty} + \norm{X(s)}_{\infty} \norm{M_n(t) - M_n(s)}_{\infty} \\
    & \hspace{20mm} + [n]_q|t-s|\norm{M_{n-1}(t)}_{\infty} + [n]_qs\norm{M_{n-1}(t) - M_{n-1}(s)}_{\infty} \\
    & \leq \left(\frac{2a_{n,T}}{\sqrt{1-q}} + 2\sqrt{\frac{TC_{n,T}}{1-q}} + [n]_qa_{n-1,T}\sqrt{T} + [n]_qT\sqrt{C_{n-1,T}}\right)\sqrt{|t-s|}.
\end{align*}
Therefore, we may take
\[
C_{n+1,T} \coloneqq \left(\frac{2a_{n,T}}{\sqrt{1-q}} + 2\sqrt{\frac{TC_{n,T}}{1-q}} + [n]_qa_{n-1,T}\sqrt{T} + [n]_qT\sqrt{C_{n-1,T}}\right)^2.
\]
This completes the induction and therefore the proof of \eqref{eq.Mnmeasured}.
Since \eqref{eq.Mnmeasured} implies that $M_n$ is $L^{\infty}$-continuous and $V_{2,L^{\infty}}(M_n : [0,T]) \leq \sqrt{C_{n,T}T} < \infty$ for all $T \geq 0$, we conclude that $M_n$ is an $L^{\infty}$-measured martingale, as claimed.
\end{example}

Next, we endow the space of $(L^p,L^q)$-measured decomposable processes with a useful topology when $1/p+1/q \leq 1$.

\begin{notation}[Space of measured decomposable processes]\label{nota.GD}
Suppose $1 \leq p,q \leq \infty$ and $1/p + 1/q \leq 1$.
Write $\GD^{p,q} = \GD_{\scA}^{p,q} \subseteq C_{a, p \wedge q}$ for the space of $(L^p,L^q)$-measured decomposable processes.
Endow $\GD^{p,q}$ with the topology induced by the collection
\[
\Big\{X \mapsto \norm{X}_{p,q,t} \coloneqq \norm{X(0)}_{p \wedge q} + V_{2,L^p}\big(X^{\mathrm{m}} : [0,t]\big) + V_{1,L^q}\big(X^{\mathrm{fv}} : [0,t]\big) : t \geq 0\Big\}
\]
of seminorms.
As in other such situations, we write $\GD^p \coloneqq \GD^{p,p}$ and $\norm{\cdot}_{p,t} \coloneqq \norm{\cdot}_{p,p,t}$.
\end{notation}

\begin{proposition}\label{prop.GD}
Suppose $1 \leq p,q \leq \infty$ and $1/p + 1/q \leq 1$.
\begin{enumerate}[label=(\roman*),font=\normalfont]
\setlength{\itemsep}{2pt}
    \item If $X \in \GD^{p,q}$, then
    \[
    \norm{X}_{p,q,t} = \norm{X(0)}_{p \wedge q} + \norm{X^{\mathrm{m}}}_{p,\infty,t} + \big\|X^{\mathrm{fv}}\big\|_{\infty, q,t} \qquad (t \geq 0).
    \]
    In particular, the martingale-part map $\GD^{p,q} \ni X \mapsto \mathrm{m}(X) \coloneqq X^{\mathrm{m}} \in \GD^{p,\infty}$ and the FV-part map $\GD^{p,q} \ni X \mapsto \mathrm{fv}(X) \coloneqq X^{\mathrm{fv}} \in \GD^{\infty,q}$ are continuous.\label{item.martpartFVpart}
    \item If $X \in \GD^{p,q}$, then
    \[
    \sup_{0 \leq s \leq t}\norm{X(s)}_{p \wedge q} \leq \norm{X}_{p,q,t} \qquad (t \geq 0).
    \]
    In particular, the inclusion $\GD^{p,q} \hookrightarrow C_{a,p \wedge q}$ is continuous.\label{item.incl}
    \item $\GD^{p,q}$ is a complex Fr\'{e}chet space.\label{item.Frechet}
\end{enumerate}
\end{proposition}
\pagebreak

\begin{proof}
Let $X \in \GD^{p,q}$.
By definition of the martingale and FV parts, $X^{\mathrm{m}}(0) = X^{\mathrm{fv}}(0) = 0$, $(X^{\mathrm{m}})^{\mathrm{m}} = X^{\mathrm{m}}$, $(X^{\mathrm{fv}})^{\mathrm{fv}} = X^{\mathrm{fv}}$, and $(X^{\mathrm{m}})^{\mathrm{fv}} = (X^{\mathrm{fv}})^{\mathrm{m}} \equiv 0$.
Item \ref{item.martpartFVpart} follows easily from these observations and the definition of $\norm{\cdot}_{p,q,t}$.

\ref{item.incl} If $t \geq 0$, then
\begin{align*}
    \norm{X(t)}_{p \wedge q}  & = \norm{X(0) + X(t) - X(0)}_{p \wedge q} = \big\|X(0) + X^{\mathrm{m}}(t) + X^{\mathrm{fv}}(t)\big\|_{p \wedge q} \\
    & = \big\|X(0) + X^{\mathrm{m}}(t) - X^{\mathrm{m}}(0) + X^{\mathrm{fv}}(t) - X^{\mathrm{fv}}(0)\big\|_{p \wedge q} \\
    & \leq \norm{X(0)}_{p \wedge q} + \norm{X^{\mathrm{m}}(t) - X^{\mathrm{m}}(0)}_{p \wedge q} + \big\|X^{\mathrm{fv}}(t) - X^{\mathrm{fv}}(0)\big\|_{p \wedge q} \\
    & \leq \norm{X(0)}_{p \wedge q} + \sqrt{\norm{X^{\mathrm{m}}(t) - X^{\mathrm{m}}(0)}_p^2} + \big\|X^{\mathrm{fv}}(t) - X^{\mathrm{fv}}(0)\big\|_q \\
    & \leq \norm{X(0)}_{p \wedge q} + V_{2,L^p}\big(X^{\mathrm{m}} : [0,t]\big) + V_{1,L^q}\big(X^{\mathrm{fv}} : [0,t]\big) = \norm{X}_{p, q,t}.
\end{align*}
Since $\norm{X}_{p,q,t}$ increases in $t$, the result follows.

\ref{item.Frechet} The only nontrivial claim is that $\GD^{p,q}$ is complete.
This may be deduced by standard arguments from the first two items and the following well-known fact, which we leave as an exercise for the unfamiliar reader:
If $\cV$ is a Banach space and $1 \leq \alpha < \infty$, then the vector space $\{F \in C(\R_+;\cV) : V_{\alpha}(F : [0,t]) < \infty \; \forall t \geq 0\}$ is a Fr\'{e}chet space with respect to the topology generated by the collection $\{F \mapsto \norm{F(0)}_{\cV} + V_{\alpha}(F : [0,t]) : t \geq 0\}$ of seminorms.
\end{proof}

We now move toward the main result of this subsection:
$L^{p_0}$-measured decomposable processes satisfy Assumption \ref{ass.integ} with $(\G,\norm{\cdot}) = (\F^{p_0;p},\norm{\cdot}_{p_0;p})$ whenever $2 \leq p < \infty$.
The key to this result is to control the $L^p$ norm of the integral against the martingale part using the discrete-time noncommutative (NC) Burkholder--Davis--Gundy (BDG) inequalities of G.\ Pisier and Q.\ Xu (partially) recalled below.

\begin{theorem}[NC BDG inequalities, \cite{PX1997}]\label{thm.NCBDG}
There exist families of positive constants $(\alpha_p)_{p \geq 2}$ and $(\beta_p)_{p \geq 2}$ such that for each $C^*$-probability space $(\cM,\tau)$, filtration $(\cM_n)_{n=0}^N$, $p \in [2,\infty)$, and discrete-time $(\cM_n)_{n=0}^N$-martingale $m = (m_n)_{n=0}^N \colon \{0,\ldots,N\} \to L^p(\tau)$,
\[
\alpha_p\norm{m}_{\cH_p(\tau)} \leq \norm{m_N}_p = \max_{0 \leq n \leq N}\norm{m_n}_p \leq \beta_p\norm{m}_{\cH_p(\tau)},
\]
where
\begin{align*}
    \norm{m}_{\cH_p(\tau)} & \coloneqq \max\Biggr\{\Bigg\|m_0^*m_0 + \sum_{n=1}^N(m_n-m_{n-1})^*(m_n-m_{n-1})\Bigg\|_{\frac{p}{2}}^{\frac12}, \\
    & \hspace{25mm} \Bigg\|m_0m_0^* + \sum_{n=1}^N(m_n-m_{n-1})(m_n-m_{n-1})^*\Bigg\|_{\frac{p}{2}}^{\frac12}  \Biggr\}.
\end{align*}
Also, $\alpha_2=\beta_2 = 1$.
\end{theorem}

If the reader has not already done so, now would be a good time to review the notation and terminology in Example \ref{ex.LpLq}.

\begin{theorem}\label{thm.MDass}
Suppose $1 \leq p_0 \leq \infty$ and $2 \leq p < \infty$.
Let $X \colon \R_+ \to L^{p_0}(\E_{\scA})$ be an $L^{p_0}$-measured decomposable process with decomposition $X = X(0) + M + A$ as in {\rm Definition~\ref{def.measured}}.
Let $\nu(\d t) \coloneqq \norm{\d A(t)}_{p_0}$ and $\mu$ be as in {\rm Lemma \ref{lem.control}} with $(\alpha,\cV,F) = (2,L^{p_0}(\E_{\cA}),M)$.
If $H \in \EP^{p_0;p}$, then $I_X^0[H]$ is an $L^p$-measured decomposable process, $I_X^0[H]^{\mathrm{m}} = I_M^0[H]$, $I_X^0[H]^{\mathrm{fv}} = I_A^0[H]$, and
\begin{align}
    \norm{I_X^0[H]}_{p,t} & \leq \beta_p \left( \int_{(0,t]} \norm{H}_{p_0;p}^2 \,\d\mu\right)^{\frac12} + \int_{(0,t]}\norm{H}_{p_0;p}\,\d \nu \label{eq.elemintbd} \\
    & \leq \left(\beta_p+\nu((0,t])^{\frac12}\right) \left( \int_{(0,t]} \norm{H}_{p_0;p}^2 \,\d(\mu+\nu)\right)^{\frac12}, \label{eq.elemintbd2}
\end{align}
where $\beta_p$ is as in {\rm Theorem \ref{thm.NCBDG}}.
In particular, as a result of {\rm Proposition \ref{prop.GD}\ref{item.incl}}, $X$ satisfies {\rm Assumption \ref{ass.integ}} with $(\G,\norm{\cdot}) = (\F^{p_0;p},\norm{\cdot}_{p_0;p})$, $\rho = \mu+\nu$, and $a_T = \beta_p+\nu((0,T])^{1/2}$ for all $T \geq 0$.
(Note that $\rho$ is atomless in this case.)
Consequently, by {\rm Theorem \ref{thm.constructNCstochint}}, $X$ is an $L^p$-integrator with respect to $\F^{p_0;p}$.
\end{theorem}

\begin{proof}
Let $H \in \EP^{p_0;p}$, $N \coloneqq I_M^0[H]$, and $B \coloneqq I_A^0[H]$, in which case $I_X^0[H] = N+B$.
By \cite[Lem.\ 4.12(ii)]{JKN2026}, $N \in \M_{\scB}^p$.
By \cite[Cor.\ 4.8]{JKN2026}, $B \in \FV_{\scB}^p$, and
\[
V_{1,L^p}(B : [0,t]) \leq \int_{(0,t]}\norm{H}_{p_0;p} \,\d\nu \qquad (t \geq 0).
\]
Now, we claim that
\begin{equation}
    V_{2,L^p}(N : [0,t]) \leq \beta_p \left( \int_{(0,t]} \norm{H}_{p_0;p}^2 \,\d\mu\right)^{\frac12} \qquad (t \geq 0),\label{eq.V2elembd}
\end{equation}
which will complete the proof of \eqref{eq.elemintbd}.
By the Cauchy--Schwarz inequality and the fact that $\max\left\{\mu,\nu\right\} \leq \mu+\nu$, \eqref{eq.elemintbd2} follows from \eqref{eq.elemintbd}.

To prove \eqref{eq.V2elembd}, note that since $H \in \EP^{p_0;p}$, there exist times $0 = t_0 < \cdots < t_k < \infty$ and $H_i \in \cF_{t_{i-1}}^{p_0;p}$ ($i=1,\ldots,k$, with $t_{-1} \coloneqq 0$) such that
\[
H = 1_{\{0\}} H_0 + \sum_{i=1}^k 1_{(t_{i-1},t_i]}H_i.
\]
Write
\[
\int_0^{\infty} H[\d M] \coloneqq \sum_{i=1}^k H_i[M(t_i) - M(t_{i-1})] = \int_0^T H[\d M] \qquad (T \geq t_k).
\]
Now, define $x_0 \coloneqq 0$, $\Delta_i M \coloneqq M(t_i) - M(t_{i-1})$ ($i=1,\ldots,k$), and
\[
x_i \coloneqq \sum_{j=1}^i H_j[\Delta_jM] = \int_0^{t_i} H[\d M] = N(t_i) \qquad (i=1,\ldots,k).
\]
Then $x = (x_i)_{i=0}^k$ is a discrete-time $L^p$-martingale with respect to $(\cB_{t_i})_{i=0}^k$.
Also, by construction, $x_k = \int_0^{\infty} H[\d M]$, and $x_i - x_{i-1} = H_i[\Delta_iM]$ for all $i = 1,\ldots,k$.
Therefore, by the NC BDG inequalities and our choice of $\mu$,
\begin{align*}
    \norm{\int_0^{\infty} H[\d M]}_p^2 & \leq \beta_p^2 \max\left\{\norm{\sum_{i=1}^k H_i[\Delta_iM]^*H_i[\Delta_iM]}_{\frac{p}{2}}, \norm{\sum_{i=1}^k H_i[\Delta_iM]\,H_i[\Delta_iM]^*}_{\frac{p}{2}}\right\} \\
    & \leq \beta_p^2 \max\left\{\sum_{i=1}^k \norm{H_i[\Delta_iM]^*H_i[\Delta_iM]}_{\frac{p}{2}}, \sum_{i=1}^k \norm{H_i[\Delta_iM]\,H_i[\Delta_iM]^*}_{\frac{p}{2}}\right\} \\
    & = \beta_p^2\sum_{i=1}^k \norm{H_i[\Delta_iM]}_p^2 \leq \beta_p^2\sum_{i=1}^k \norm{H_i}_{p_0;p}^2 \norm{\Delta_iM}_{p_0}^2\\
    & \leq \beta_p^2\sum_{i=1}^k \norm{H_i}_{p_0;p}^2 \, \mu((t_{i-1},t_i]) = \beta_p^2\int_{(0,\infty)} \norm{H}_{p_0;p}^2 \,\d\mu.
\end{align*}
Applying this bound to $1_{(r,s]}H$ in place of $H$ and using $\int_r^s H[\d M] = \int_0^s H[\d M] - \int_0^r H[\d M]$, we obtain \eqref{eq.V2elembd}.
\end{proof}

Consequently, if $X \colon \R_+ \to L^{p_0}(\E_{\scA})$ is an $L^{p_0}$-decomposable process and $2 \leq p < \infty$, then we may speak of the stochastic integral $I_X \colon \cI^{p_0;p} \to C_{a,p}$.
This noncommutative stochastic integral has a few more notable properties than a completely general one.
Before going through them, we record an application not requiring those additional properties:
continuous-time NC BDG inequalities generalizing \cite[Thm.\ 1.11]{JKN2026}.
We keep the arguments brief because they are very similar to those used to prove \cite[Thm.\ 1.11]{JKN2026}.

\begin{notation}\label{nota.dXdY}
Suppose $1 \leq p,q \leq \infty$ and $1/p+1/q \leq 1$.
For processes $X \colon \R_+ \to L^p(\E)$ and $Y \colon \R_+ \to L^q(\E)$ and $t \geq 0$, write
\[
\int_0^t \d X(s)\,\d Y(s) \coloneqq L^1\text{-}\lim_{\Pi \in \cP_{[0,t]}} \sum_{s \in \Pi} \Delta_s X\,\Delta_s Y
\]
if this limit exists.
\end{notation}

\begin{proposition}\label{prop.dXdY}
Suppose $2 \leq p \leq p_0 \leq \infty$, $2 \leq q \leq q_0 \leq \infty$, $1/p_0+1/q \leq 1/2$, and $1/p+1/q_0 \leq 1/2$.
If $X \colon \R_+ \to L^{p_0}(\E)$ is an $L^{p_0}$-continuous $L^p$-measured decomposable process and $Y \colon \R_+ \to L^{q_0}(\E)$ is an $L^{q_0}$-continuous $L^q$-measured decomposable process, then
\begin{align*}
    \sum_{t \in \Pi} (X(t \wedge \cdot) - X&(t_- \wedge \cdot))(Y(t \wedge \cdot) - Y(t_- \wedge \cdot)) \\
    & \xrightarrow[\Pi \in \cP_{\R_+}]{|\Pi| \to 0} X\,Y - X(0)\,Y(0) - \into \d X(t)\,Y(t) - \into X(t)\,\d Y(t)
\end{align*}
in the space $C_{a,2} = C_a(\R_+;L^2(\E))$.
In particular,
\[
\into \d X(t)\,\d Y(t) = X\,Y - X(0)\,Y(0) - \into \d X(t)\,Y(t) - \into X(t)\,\d Y(t).
\]
\end{proposition}

\begin{proof}[Sketch of proof]
If $\Pi$ is a partition of $\R_+$ and $t \geq 0$, then
\begin{align*}
    X(t)\,Y(t) - X(0)\,Y(0) & = \sum_{s \in \Pi} \Delta_s(XY)^t = \sum_{s \in \Pi} \left(\Delta_sX^t\, Y(s_-) + X(s_-) \, \Delta_s Y^t + \Delta_sX^t\,\Delta_sY^t\right).
\end{align*}
The result then follows from two applications of Proposition \ref{prop.IXRS} (and Theorem \ref{thm.MDass}).
\end{proof}

\begin{corollary}\label{cor.epsilon/5}
Suppose the indices $p$, $p_0$, $q$, and $q_0$ are as in {\rm Proposition \ref{prop.dXdY}} and $M,N \colon \R_+ \to L^2(\E)$ are $L^2$-martingales.
If $M$ is the limit in $C_{a,2}$ of a sequence of $L^{p_0}$-continuous $L^p$-measured martingales and $N$ is the limit in $C_{a,2}$ of a sequence of $L^{q_0}$-continuous $L^q$-measured martingales, then $\int_0^t \d M(s)\,\d N(s)$ exists for all $t \geq 0$.
\end{corollary}

\begin{proof}[Sketch of proof]
Suppose $M_1,M_2 \colon \R_+ \to L^2(\E)$ are $L^2$-martingales.
If $t \geq 0$ and $\Pi$ is a partition of $[0,t]$, define 
\[
\RS_{\Pi}^{M_1,M_2} \coloneqq \sum_{s \in \Pi} \Delta_sM_1\,\Delta_s M_2.
\]
By \cite[Lem.\ 5.6]{JKN2026},
\[
\norm{\RS_{\Pi}^{M_1,M_2}}_1 \leq \norm{M_1(t) - M_1(0)}_2\norm{M_2(t) - M_2(0)}_2.
\]
Using this inequality and the hypotheses on $M$ and $N$, the fact that $\big(\RS_{\Pi}^{M,N}\big)_{\Pi \in \cP_{[0,t]}}$ is Cauchy and therefore convergent in $L^1(\E)$ then follows from Proposition \ref{prop.dXdY} via a standard ``$\e/5$ argument.''
\end{proof}

Corollary \ref{cor.epsilon/5} and \cite[Thm.\ 5.23]{JKN2026} together directly yield the following generalization of \cite[Thm.\ 1.11]{JKN2026}, which is the $(q,q_0) = (2,\infty)$ case.

\begin{theorem}[Continuous-time NC BDG inequalities]\label{thm.contimeNCDG}
Suppose $2 \leq p,q_0,q \leq \infty$, $p \neq \infty$, and $1/q+1/q_0 \leq 1/2$.
If $M \colon \R_+ \to L^p(\E)$ is an $L^p$-martingale that is the limit in $C_{a,2}$ of a sequence $L^{q_0}$-continuous $L^q$-measured decomposable martingales, then
$\int_0^t \d M^*(s)\,\d M(s)$ and $\int_0^t \d M(s)\,\d M^*(s)$ belong to $L^{p/2}(\E)$, and
\[
\alpha_p^{-1}\norm{M}_{\cH_t^p(\cA)} \leq \norm{M(t)}_p = \sup_{0 \leq s \leq t} \norm{M(s)}_p \leq \beta_p\norm{M}_{\cH_t^p(\cA)},
\]
where
\begin{align*}
    \norm{M}_{\cH_t^p(\cA)} & \coloneqq \max\Biggr\{\Bigg\|M(0)^*M(0) + \int_0^t \d M^*(s)\,\d M(s)\Bigg\|_{\frac{p}{2}}^{\frac12}, \\
    & \hspace{25mm} \Bigg\|M(0)M(0)^* + \int_0^t \d M(s)\,\d M^*(s)\Bigg\|_{\frac{p}{2}}^{\frac12}  \Biggr\}
\end{align*}
and $\alpha_p$ and $\beta_p$ are as in {\rm Theorem \ref{thm.NCBDG}}. \qed
\end{theorem}

We now turn back to general properties of the stochastic integral against measured decomposable processes.

\begin{theorem}\label{thm.stochint}
Retain the setup of {\rm Theorem \ref{thm.MDass}}.
\begin{enumerate}[label=(\roman*),font=\normalfont]
    \item If $H \in \cI^{p_0;p}$, then $I_X[H] \in \GD_{\scB}^p$, $I_X[H]^{\mathrm{m}} = I_M[H]$, $I_X[H]^{\mathrm{fv}} = I_A[H]$, and\label{item.MDp0p}
    \[
    \norm{I_X[H]}_{p,t} \leq \beta_p\left( \int_{(0,t]} \norm{H}_{p_0;p}^2 \,\d \mu\right)^{\frac12} + \int_{(0,t]} \norm{H}_{p_0;p} \,\d\nu \qquad (t \geq 0).
    \]
    \item $I_X^{p_0;p} \coloneqq I_X$ is continuous with respect to locally bounded convergence as a map from $\cI^{p_0;p}$ to $\GD_{\scB}^p$.\label{item.MDp0pcont}
    \item If $p_0 \geq 2$, then the bilinear map $I \colon \GD_{\scA}^{p_0} \times \cI^{p_0;p} \to \GD_{\scB}^p$ defined by $(X,H) \mapsto I_X[H]$ satisfies the following additional continuity property:
    If $(X_n)_{n \in \N}$ is a sequence in $\GD_{\scA}^{p_0}$ converging to $X \in \GD_{\scA}^{p_0}$ and $(H_n)_{n \in \N}$ is a sequence in $\cI^{p_0;p}$ converging locally boundedly to $H \in \cI^{p_0;p}$, then $I[X_n,H_n] \to I[X,H]$ in $\GD_{\scB}^p$ as $n \to \infty$.\label{item.bilincont}
\end{enumerate}
\end{theorem}

\begin{proof}
We prove items \ref{item.MDp0p} and \ref{item.MDp0pcont} simultaneously.
Let $\cL_X$ be the space of strongly measurable maps $H \colon \R_+ \to B^{p_0;p}$ such that
\[
\norm{H}_{X,t} \coloneqq \left(\int_{(0,t]} \norm{H}_{p_0;p}^2 \,\d\mu\right)^{\frac{1}{2}} + \int_{(0,t]} \norm{H}_{p_0;p} \,\d\nu < \infty \qquad (t \geq 0),
\]
endowed with the locally convex topology generated by the collection $\{\norm{\cdot}_{X,t} : t \geq 0\}$ of seminorms, and write $\cI_X$ for the closure of $\EP^{p_0;p}$ in $\cL_X$.
Since $\cI_X$ is closed under locally bounded convergence by the dominated convergence theorem, $\cI^{p_0;p} \subseteq \cI_X$.
By the same arguments that proved Proposition \ref{prop.IGrhoextension}, Theorem \ref{thm.MDass} implies that $I_X^0$ extends uniquely to a continuous complex-linear map $J_X \colon \cI_X \to \GD_{\scB}^p$ such that for all $H \in \cI_X$ and $t \geq 0$,
\begin{align*}
    \norm{J_X[H]}_{p,t} & \leq \beta_p\left( \int_{(0,t]} \norm{H}_{p_0;p}^2 \,\d \mu\right)^{\frac12} + \int_{(0,t]} \norm{H}_{p_0;p} \,\d\nu \leq \max\left\{\beta_p,1\right\}  \norm{H}_{X,t} \\
    & \leq \max\left\{\beta_p,1\right\} \left(\mu((0,t])^{\frac12} + \nu((0,t])\right) \sup_{0 < s \leq t} \norm{H(s)}_{p_0;p} \\
    & = \max\left\{\beta_p,1\right\} \left(V_{2,L^{p_0}}(M : [0,t])^{\frac12} + V_{1,L^{p_0}}(A : [0,t])\right) \sup_{0 < s \leq t} \norm{H(s)}_{p_0;p}.
\end{align*}
Consequently, the dominated convergence theorem implies $J_X|_{\cI^{p_0;p}}$ is continuous with respect to locally bounded convergence (as a map from $\cI^{p_0;p}$ to $\GD_{\scB}^p$).
By Proposition \ref{prop.GD}\ref{item.incl} and the uniqueness property of $I_X$, the previous sentence implies $J_X|_{\cI^{p_0;p}} = I_X$.
Also,
\[
\cS \coloneqq \left\{ H \in \cI_X : J_X[H]^{\mathrm{m}} = J_M[H] \text{ and } J_X[H]^{\mathrm{fv}} = J_A[H]\right\}
\]
contains $\EP^{p_0;p}$ by Theorem \ref{thm.MDass} and is closed in $\cI_X$ by the continuity of the $J$ maps and Proposition \ref{prop.GD}\ref{item.martpartFVpart}.
Since $\EP^{p_0;p}$ is dense in $\cI_X$, we obtain $\cS = \cI_X$.
This establishes all the assertions of the first two items.

To prove \ref{item.bilincont}, let $(X_n)_{n \in \N}$ and $(H_n)_{n \in \N}$ be as in the statement, and define
\[
c_t \coloneqq \sup_{n \in \N}\sup_{0 \leq s \leq t}\norm{H_n(s)}_{p_0;p} < \infty \qquad (t \geq 0).
\]
By the bounds above and the dominated convergence theorem, if $t \geq 0$, then
\begin{align*}
    \|I[X_n,H_n]& - I[X,H]\|_{p,t} \leq \norm{I[X_n,H_n] - I[X,H_n]}_{p,t} + \norm{I[X,H_n] - I[X,H]}_{p,t} \\
    & = \norm{I_{X_n-X}[H_n]}_{p,t} + \norm{I_X[H_n - H]}_{p,t} \\
    & \leq \max\left\{\beta_p,1\right\}\left( \norm{H_n}_{X_n-X,t} + \norm{H_n-H}_{X,t}\right) \\
    & \leq \max\left\{\beta_p,1\right\}\left(c_t\left(\norm{X_n-X}_{p_0,t}^{\frac12} + \norm{X_n-X}_{p_0,t}\right) + \norm{H_n-H}_{X,t}\right) \xrightarrow{n \to \infty} 0,
\end{align*}
as desired.
\end{proof}

We end this subsection with the substitution formula, a result similar to \cite[Thm.\ 4.23]{JKN2026} (in which $p_0=p=2$).

\begin{proposition}\label{prop.alg}
Let $p,q,r \in [1,\infty]$ and $(\cC,(\cC_t)_{t \geq 0},\E_{\scC})$ be another conditionable filtered $\mathrm{C}^*$-probability space.
If $H \in \cI^{p;q}(\E_{\scA};\E_{\scB})$ and $K \in \cI^{q;r}(\E_{\scB};\E_{\scC})$, then $KH \in \cI^{p;r}(\E_{\scA};\E_{\scC})$.
\end{proposition}

\begin{proof}
We leave it to the reader to check that if $T \in \cF_t^{p;q}(\E_{\scA};\E_{\scB})$ and $S \in \cF_t^{q;r}(\E_{\scB};\E_{\scC})$, then $ST \in \cF_t^{p;r}(\E_{\scA};\E_{\scC})$.
It follows that if $H \in \EP^{p;q}(\E_{\scA};\E_{\scB})$ and $K \in \EP^{q;r}(\E_{\scB};\E_{\scC})$, then $KH \in \EP^{p;r}(\E_{\scA};\E_{\scC})$.
In particular, if we define $\cI_H$ to be the set of $K \in \cI^{q;r}(\E_{\scB};\E_{\scC})$ such that $KH \in \cI^{p;r}(\E_{\scA};\E_{\scC})$, then  $\EP^{p;q}(\E_{\scB};\E_{\scC}) \subseteq \cI_H$ for all $H \in \EP^{p;q}(\E_{\scA};\E_{\scB})$.
Since $\cI_H$ is clearly closed under locally bounded convergence, it follows that $\cI^{p;q}(\E_{\scB};\E_{\scC}) = \cI_H$ for all $H \in \EP^{p;q}(\E_{\scA};\E_{\scB})$.
Next, define $\cI^K \coloneqq \{H \in \cI^{p;q}(\E_{\scA};\E_{\scB}) : KH \in \cI^{p;r}(\E_{\scA};\E_{\scC})\}$ for all $K \in \cI^{q;r}(\E_{\scB};\E_{\scC})$.
We just showed that $\EP^{p;q}(\E_{\scA};\E_{\scB}) \subseteq \cI^K$ for all $K \in \cI^{q;r}(\E_{\scB};\E_{\scC})$.
Since $\cI^K$ is closed under locally bounded convergence, we conclude that $\cI^K = \cI^{p;q}(\E_{\scA};\E_{\scB})$ for all $K \in \cI^{q;r}(\E_{\scB};\E_{\scC})$, as desired.
\end{proof}

\begin{theorem}[Substitution formula]\label{thm.subform}
Suppose $(\cC,(\cC_t)_{t \geq 0},\E_{\scC})$ is another conditionable filtered $\mathrm{C}^*$-probability space, $1 \leq p \leq \infty$, and $2 \leq q,r < \infty$.
If $X \colon \R_+ \to L^p(\E_{\scA})$ is an $L^p$-measured decomposable process, $H \in \cI^{p;q}(\E_{\scA};\E_{\scB})$, $Y_H \coloneqq \into H[\d X] \in \GD_{\scB}^q$, and $K \in \cI^{q;r}(\E_{\scB};\E_{\scC})$, then
\begin{equation}
    \into K[\d Y_H] = \into KH[\d X].\label{eq.subform}
\end{equation}
\end{theorem}

\begin{proof}
For each $H \in \cI^{p;q}(\E_{\scA};\E_{\scB})$, define $\cS_H \coloneqq \left\{K \in \cI^{q;r}(\E_{\scB};\E_{\scC}) : \text{\eqref{eq.subform} holds}\right\}$.
Since $I_X$ and $I_{Y_H}$ are continuous with respect to locally bounded convergence, $\cS_H$ is closed under locally bounded convergence.
Also, it is a matter of algebra to show that $\EP^{q;r}(\E_{\scB};\E_{\scC}) \subseteq \cS_H$ for all $H \in \EP^{p;q}(\E_{\scA};\E_{\scB})$.
Thus, $\cS_H = \cI^{q;r}(\E_{\scB};\E_{\scC})$ for all $H \in \EP^{p;q}(\E_{\scA};\E_{\scB})$.

Next, let $K \in \cI^{q;r}(\E_{\scB};\E_{\scC})$, and define $\cS^K \coloneqq \left\{H \in \cI^{p;q}(\E_{\scA};\E_{\scB}) : \text{\eqref{eq.subform} holds}\right\}$.
By the previous paragraph, $\EP^{p;q}(\E_{\scA};\E_{\scB}) \subseteq \cS^K$.
We claim that $\cS^K$ is closed under locally bounded convergence.
Indeed, let $(H_n)_{n \in \N}$ be a sequence in $\cS^K$ converging locally boundedly to $H \in \cI^{p;q}(\E_{\scA};\E_{\scB})$.
By Theorem \ref{thm.stochint}\ref{item.MDp0pcont}, $(I_X[KH_n])_{n \in \N}$ converges to $I_X[KH]$ in $\GD_{\scB}^q$.
For the same reason, $(Y_{H_n})_{n \in \N} = (I_X[H_n])_{n \in \N}$ converges to $I_X[H] = Y_H$ in $\GD_{\scB}^q$.
Consequently, by Theorem \ref{thm.stochint}\ref{item.bilincont}, $(I[Y_{H_n},K])_{n \in \N}$ converges to $I[Y_H,K]$ in $\GD_{\scC}^r$.
By definition of $\cS_K$, $I_X[KH_n] = I[Y_{H_n},K]$ for all $n \in \N$.
Thus, \eqref{eq.subform} holds, i.e., $H \in \cS^K$, which proves the claim.
Consequently, $\cS^K = \cI^{p;q}(\E_{\scA};\E_{\scB})$ for all $K \in \cI^{q;r}(\E_{\scB};\E_{\scC})$, which is what we set out to prove.
\end{proof}

\subsection{Operator-norm bound for free Brownian integrator}\label{subsec.fBMintegrator}

Take $(\cB,(\cB_t)_{t \geq 0},\E_{\scB}) = (\cA,(\cA_t)_{t \geq 0},\E)$ for the duration of this subsection.

\begin{notation}\label{nota.finrank}
Fix the following notation.
\begin{enumerate}[label=(\roman*), font=\normalfont]
    \item $\otimes$ is the algebraic tensor product over $\C$, and $\otimes_{\min}$ is the minimal (or spatial) $\mathrm{C}^*$-tensor product (vid.\ \cite[Ch.\ 3]{BO2008}).
    \item $\cA^{\op}$ is the opposite ($\mathrm{C}^*$-)algebra of $\cA$.
    \item $\# \colon \cA \otimes \cA^{\op} \to \B(\cA)$ is the complex-linear map determined by
    \[
    \#(a \otimes b)x = axb \qquad (a,b,x \in \cA).
    \]
    Write\label{item.sh}
    \[
    u \sh x \coloneqq \#(u)x \qquad \big(u \in \cA \otimes \cA^{\op}, \; x \in \cA\big).
    \]
    \item $\mathfrak{t} \colon \cA \otimes \cA^{\op} \to \B(\cA)$ is the complex-linear map determined by
    \[
    \mathfrak{t}(a \otimes b)x = \E[bx]a \qquad (a,b,x \in \cA).
    \]
    Write\label{item.finrank}
    \[
    v[x] \coloneqq \mathfrak{t}(v)x \qquad \big(v \in \cA \otimes \cA^{\op}, \; x \in \cA\big).
    \]
\end{enumerate}

\end{notation}

In \cite{BS1998}, Biane and Speicher proved an operator-norm bound on the free stochastic integral of an \textbf{elementary adapted biprocess}, i.e., a map $U \colon \R_+ \to \cA \otimes \cA^{\op}$ of the form
\begin{equation}\label{eq.elembip}
\begin{split}
    U & = 1_{\{0\}} u_0 +  \sum_{k=1}^{\ell} 1_{(s_k,t_k]} u_k, \; \text{ where} \\
    u_k & \in \cA_{s_k} \otimes \cA_{s_k}^{\op} \quad (k=0,\ldots,\ell; s_0 \coloneqq 0).
\end{split}
\end{equation}
Specifically, \cite[Thm.\ 3.2.1]{BS1998} establishes that if $S \colon \R_+ \to \cA_{\sa}$ is a free Brownian motion and $U \colon \R_+ \to \cA \otimes \cA^{\op}$ is as in \eqref{eq.elembip}, then
\begin{align*}
    \norm{\int_0^{\infty} U(t)\sh \d S(t)}_{\infty} & = \norm{\sum_{k=1}^{\ell} u_k\sh(S(t_k) - S(s_k))}_{\infty} \\
    & \leq 2\sqrt2 \left(\int_0^{\infty} \norm{U(t)}_{\cA \otimes_{\min} \cA^{\op}}^2\,\d t\right)^{\frac12}.  \numberthis\label{eq.BSelembip}
\end{align*}
Converting a multidimensional free Brownian motion into a one-dimensional one in a clever way results in a multidimensional version of this inequality, stated below and proven in appendix \ref{app.multidimBSineq}.

\begin{definition}\label{def.ndimfBM}
An $n$-tuple $(S_1,\ldots,S_n) \colon \R_+ \to \cA_{\sa}^n$ is an \textbf{$\boldsymbol{n}$-dimensional free Brownian motion} if $S_1(0) = \cdots = S_n(0) = 0$, and for all $t > s \geq 0$, $S_i(t) - S_i(s)$ is a semicircular element of variance $t-s$ ($i=1,\ldots, n$), $(S_1(t)-S_1(s),\ldots,S_n(t)-S_n(s))$ is a free family, and $\{S_i(t) - S_i(s) : i=1,\ldots,n\}$ is freely independent of $\cA_s$.
If $n=1$, we drop ``$n$-dimensional'' from the terminology.
\end{definition}

\begin{theorem}\label{thm.multidimBSineq}
Let $(S_1,\ldots,S_n) \colon \R_+ \to \cA_{\sa}^n$ be an $n$-dimensional free Brownian motion.
If $U_j \colon \R_+ \to \cA \otimes \cA^{\op}$ is an elementary biprocess for each $j=1,\ldots,n$, then
\[
\norm{\sum_{j=1}^n \int_0^{\infty} U_j(t)\sh \d S_j(t)}_{\infty} \leq 2\sqrt2 \left(\int_0^{\infty}\sum_{j=1}^n\norm{U_j(t)}_{\cA \otimes_{\min} \cA^{\op}}^2\,\d t\right)^\frac12.
\]
\end{theorem}

The remainder of this section is devoted to converting this result into an operator-norm bound on the stochastic integrals of certain multidimensional trace biprocesses against multidimensional free Brownian motion.
To begin, we prove a key estimate relating the tensor norm appearing in Theorem \ref{thm.multidimBSineq} to a linear-map norm.

\begin{proposition}\label{prop.mintensornormbound}
Let $t \geq 0$, and suppose there exists a non-zero $c \in \cA$ such that $\E[c] = 0$ and $c$ is $\ast$-freely independent of $\cA_t$.
If $u \in \cA_t \otimes \cA_t^{\op}$, then
\[
\norm{u}_{\cA \otimes_{\min} \cA^{\op}} \leq \inf \big\{\norm{\#(u) + \mathfrak{t}(v)}_{2;2} : v \in \cA_t \otimes \cA_t^{\op}\big\}.
\]
\end{proposition}

\begin{proof}[First proof]
Write $\E_{\otimes} \coloneqq \E \otimes_{\min} \E^{\op} \colon \cA \otimes_{\min} \cA^{\op} \to \C$ for the trace on $\cA \otimes_{\min} \cA^{\op}$ determined by $\E_{\otimes}[a \otimes b] = \E[a]\,\E[b]$ ($a,b \in \cA$).
Also, assume that $\norm{c}_2 = 1$.
By a brief calculation using the fact that $c$ is $\ast$-freely independent of $\cA_t$, if $u,v \in \cA_t \otimes \cA_t^{\op}$, then
\begin{equation}
    v[c] = 0 \; \text{ and } \; \norm{u \sh c}_{L^2(\E)} = \norm{u}_{L^2(\E_{\otimes})}.\label{eq.keyobs}
\end{equation}
Now, define $\cC_t$ to be the closure in $B_{\C}(L^2(\E))$ of $\{\#(u) + \mathfrak{t}(v) : u,v \in \cA_t \otimes \cA_t^{\op}\}$.
It is not difficult to see that $\cC_t$ is a unital $\mathrm{C}^*$-algebra.
Let $\rho \colon \cC_t \to \C$ be the state $\rho(T) \coloneqq \la Tc, c \ra_{L^2(\E)}$.
If $\cN_{\rho} \coloneqq \{T \in \cC_t : \rho(T^*T) = 0\}$, then the completion of the space $H_{\rho} \coloneqq \cC_t/\cN_{\rho}$ with the inner product $\la T+\cN_{\rho}, S+\cN_{\rho} \ra_{\rho} \coloneqq \rho(S^*T)$ is the Hilbert space in the Gelfand--Naimark--Segal (GNS) construction for $\rho$.
Write $\pi \colon \cC_t \to B_{\C}(H_{\rho})$ for the associated GNS representation.
By definition and \eqref{eq.keyobs}, if $u,v \in \cA_t \otimes \cA_t^{\op}$, then
\[
\norm{\#(u) + \mathfrak{t}(v) + \cN_{\rho}}_{\rho} = \norm{u \sh c+v[c]}_{L^2(\E)} = \norm{u}_{L^2(\E_{\otimes})}.
\]
This shows that the map $(\cA_t \otimes \cA_t^{\op},\ip{\cdot,\cdot}_{L^2(\E_{\otimes})}) \to (H_{\rho},\ip{\cdot,\cdot}_{\rho})$ obtained by composing $\#$ with the natural projection $\cC_t \to H_{\rho}$ is an isometry with dense image.
It therefore extends to a unitary $U \colon L^2(\cA_t \otimes_{\min} \cA_t^{\op}, \E_{\otimes}) \to H_{\rho}$.
We claim that
\[
U^*\pi(\#(u) + \mathfrak{t}(v)) U = u \qquad \big(u,v \in \cA_t \otimes \cA_t^{\op}\big),
\]
where $u$ is represented as a left-multiplication operator on $L^2(\cA_t \otimes_{\min} \cA_t^{\op}, \E_{\otimes})$ via the standard representation.
Indeed, if $w \in \cA_t \otimes \cA_t^{\op} \subseteq \cA_t \otimes_{\min} \cA_t^{\op} \subseteq L^2(\cA_t \otimes_{\min} \cA_t^{\op}, \E_{\otimes})$,~then
\begin{align*}
    U^*\pi(\#(u) + \mathfrak{t}(v))Uw & = U^*\pi(\#(u))Uw = U^*\pi(\#(u))(\#(w) + \cN_{\rho}) \\
    & = U^*(\#(u)\#(w) + \cN_{\rho}) = U^*(\#(u w) + \cN_{\rho}) = uw,
\end{align*}
which establishes the claim.
Thus,
\begin{align*}
    \norm{u}_{\cA \otimes_{\min} \cA^{\op}} & = \norm{u}_{\cA_t \otimes_{\min} \cA_t^{\op}} = \norm{U^*\pi(\#(u) + \mathfrak{t}(v))U}_{\cA_t \otimes_{\min} \cA_t^{\op}} \\
    & \leq \norm{\#(u) + \mathfrak{t}(v)}_{\cC_t} = \norm{\#(u) + \mathfrak{t}(v)}_{2;2}
\end{align*}
for all $u,v \in \cA_t \otimes \cA_t^{\op}$, as desired.
\end{proof}

\begin{proof}[Second proof]
Let $\cC_t$ be as in the first proof, and define $\cI_t \subseteq \cC_t$ to be the smallest closed $\ast$-ideal of $\cC_t$ containing $\{\mathfrak{t}(v) : v \in \cA_t \otimes \cA_t^{\op}\}$.
The quotient $\cC_t/\cI_t$ is then a $\mathrm{C}^*$-algebra with the usual quotient norm, and one can use \eqref{eq.keyobs} and basic properties of $\#$ to show that the map $\cA_t \otimes \cA_t^{\op} \ni u \mapsto \#(u) + \cI_t \in \cC_t/\cI_t$ is an injective $\ast$-homomorphism.
In particular,
\[
\norm{u} \coloneqq \norm{\#(u)+\cI_t}_{\cC_t/\cI_t} \qquad \big(u \in \cA_t \otimes \cA_t^{\op}\big)
\]
is a $\mathrm{C}^*$-tensor norm on $\cA_t \otimes \cA_t^{\op}$.
It then follows from Takesaki's theorem on the minimality of $\norm{\cdot}_{\cA_t \otimes_{\min} \cA_t^{\op}}$ (vid.\ \cite[\S3.4]{BO2008}) that $\|u\|_{\cA_t \otimes_{\min} \cA_t^{\op}} \leq \|u\|$ for all $u \in \cA_t \otimes \cA_t^{\op}$.
Since it is clear that $\|u\| \leq \inf\big\{\norm{\#(u)+\mathfrak{t}(v)}_{2;2} : v \in \cA_t \otimes \cA_t^{\op}\big\}$, we are done.
\end{proof}

Next, we recall the direct sum construction for $\mathrm{C}^*$-probability spaces.
If $(\cA_i,\E_i)$ is a $\mathrm{C}^*$-probability space for each $i=1,\ldots,n$, then $\cC \coloneqq \cA_1 \oplus \cdots \oplus \cA_n$ is a unital $\mathrm{C}^*$-algebra with the norm $\norm{\mathbf{a}} = \max\left\{\norm{a_i} : i=1,\ldots,n\right\}$ ($\mathbf{a} = (a_1,\ldots,a_n) \in \cC$) and componentwise operations, and $\cC \ni \mathbf{a} \mapsto \E^{\oplus}[\mathbf{a}] \coloneqq n^{-1}(\E_1[a_1] + \cdots \E_n[a_n]) \in \C$ is a trace.
In this case,
\[
\norm{\mathbf{a}}_{L^p(\E^{\oplus})} = \left( \frac1n \sum_{i=1}^n \norm{a_i}_{L^p(\E_i)}^p\right)^\frac1p \qquad \big( \mathbf{a} = (a_1,\ldots,a_n) \in \cC, \; 1 \leq p < \infty\big).
\]
For the remainder of this subsection, we shall write $\E^{\oplus n}$ for the direct-sum trace on $\cA^n$.
Observe that $(\cA^n,(\cA_t^n)_{t \geq 0},\E^{\oplus n})$ is a conditionable filtered $\mathrm{C}^*$-probability space.

At this time, the reader may wish to review Notation \ref{nota.trpolyfiltration}--Example \ref{ex.trbip}.

\begin{notation}\label{nota.mntrbip}
Suppose $m,n \in \N$ and $1 \leq p,q \leq \infty$.
\begin{enumerate}[label=(\roman*), font=\normalfont]
    \item If $\cV$ and $\cW$ are vector spaces, then $\mathrm{M}_{m \times n}(\cV)$ is the set of $m \times n$ matrices with entries from $\cV$, $\Hom\left(\cV;\cW\right)$ is the space of linear maps from $\cV$ to $\cW$, and we view $\mathrm{M}_{m \times n}(\Hom\left(\cV;\cW\right))$ as $\Hom\left(\cV^n;\cW^m\right)$ in the obvious way.
    \item Let $(\cC_1,\E_1)$ and $(\cC_2,\E_2)$ be $\mathrm{C}^*$-probability spaces.
    If $T \colon L^p(\E_1) \to L^q(\E_2)$ is a real-linear map, define
    \[
    T^{\C}[x] \coloneqq T[\cRe x] + iT[\cIm x] \in L^q(\E_2) \qquad (x \in L^p(\E_1)),
    \]
    where $\cRe x = (x+x^*)/2$ and $\cIm x = -i(x-x^*)/2$.
    Note that $T^{\C}$ is complex linear and $\norm{T^{\C}}_{p;q} \leq 2 \norm{T}_{p;q}$.
\end{enumerate}
\end{notation}

\begin{theorem}[$L^{\infty}$-norm bound on free stochastic integral]\label{thm.Linfbound}
Let $\mathbf{S} \colon \R_+ \to \cA_{\sa}^n$ be an $n$-dimensional free Brownian motion.
If $\mathbf{H} \in \big(\cI_{\cT}^{2;2}\big)^{m \times n}$ and $t \geq 0$, then $\int_0^t \mathbf{H}[\d \mathbf{S}] \in \cA_t^m$,~and
\begin{align*}
    \norm{\int_0^t \mathbf{H}(r)[\d \mathbf{S}(r)]}_{L^{\infty}(\E^{\oplus m})} & \leq 2\sqrt 2\max_{i=1,\ldots,m}\left(\int_0^t \sum_{j=1}^n \big\|H_{ij}^{\C}(s)\big\|_{L^2(\E) \to L^2(\E)}^2\,\d s \right)^\frac12 \\
    & \leq 2\sqrt{2n} \left( \int_0^t \big\|\mathbf{H}^{\C}(s)\big\|_{L^2(\E^{\oplus n}) \to L^2(\E^{\oplus m})}^2 \,\d s \right)^\frac12 \\
    & \leq 4\sqrt{2n} \left( \int_0^t \norm{\mathbf{H}(s)}_{L^2(\E^{\oplus n}) \to L^2(\E^{\oplus m})}^2 \,\d s \right)^\frac12.
\end{align*}
In particular, $\mathbf{S}$ satisfies {\rm Assumption \ref{ass.integ}} with $p_0 = 2$, $p=\infty$, the normed filtration
\[
(\G,\norm{\cdot}) = \big(\big(\big(\cT_t^{2;2}\big)^{m \times n}\big)_{t \geq 0},\norm{\cdot}_{L^2(\E^{\oplus n}) \to L^2(\E^{\oplus m})}\big)
\]
of $B^{2;2}(\E^{\oplus n};\E^{\oplus m})$, $\rho$ equal to the Lebesgue measure, and $a_T = 4\sqrt{2n}$ for all $T \geq 0$.
Consequently, by {\rm Theorem \ref{thm.constructNCstochint}}, $\mathbf{S}$ is an $L^{\infty}$-integrator with respect to $\G$.
\end{theorem}

\begin{proof}[Sketch of proof]
By breaking into components, it suffices to treat the $m = 1$ case.
In this case, we begin by arguing that if $\mathbf{H} = \begin{bmatrix}
    H_1 & \cdots & H_n
\end{bmatrix} \in \big(\EP_{\cT}^{2;2}\big)^{1 \times n}$, then
\begin{equation}
    \norm{\int_0^{\infty} \mathbf{H}(t)[\d \mathbf{S}(t)]}_{\infty} \leq 2\sqrt2\left(\int_0^{\infty} \sum_{j=1}^n \big\|H_j^{\C}(t)\big\|_{2;2}^2\,\d t\right)^\frac12,\label{eq.semicircineq}
\end{equation}
where $\into \mathbf{H}[\d \mathbf{S}] = \lim_{T \to \infty}\int_0^T \mathbf{H}[\d \mathbf{S}]$.
To this end, let $\mathbf{H} = \begin{bmatrix}
    H_1 & \cdots & H_n
\end{bmatrix} \in \big(\EP_{\cT}^{2;2}\big)^{1 \times n}$, and observe that
\[
\int_0^{\infty} \mathbf{H}(t)[\d \mathbf{S}(t)] = \int_0^{\infty} \mathbf{H}^{\C}(t)[\d \mathbf{S}(t)]
\]
because $\mathbf{S}$ is self-adjoint.
Now, by the definitions of the elementary stochastic integral and $\big(\EP_{\cT}^{2;2}\big)^{1 \times n}$ and a brief limiting argument, to establish \eqref{eq.semicircineq} for all $\mathbf{H} \in \big(\EP_{\cT}^{2;2}\big)^{1 \times n}$, it suffices to consider $\mathbf{H}$ such that $\mathbf{H}^{\C}$ is of the form
\[
\mathbf{H}^{\C} = \sum_{k=1}^{\ell} 1_{(t_{k-1},t_{k}]} \big(\mathbf{H}_{k}^1 + \mathbf{H}_{k}^2\big),
\]
where $0 = t_0 < \cdots < t_{\ell} < \infty$, $\mathbf{H}_{k}^1 = \begin{bmatrix} \#\big(u_1^{k}\big) & \cdots & \#\big(u_n^{k}\big)\end{bmatrix}$, $\mathbf{H}_{k}^2 = \begin{bmatrix} \mathfrak{t}\big(v_1^{k}\big) & \cdots & \mathfrak{t}\big(v_n^{k}\big)\end{bmatrix}$, and $u_j^{k}, v_j^{k} \in \cA_{t_{k-1}} \otimes \cA_{t_{k-1}}^{\op}$ ($k=1,\ldots,\ell$; $j=1,\ldots,n$).
Next, a brief calculation reveals that
\[
\int_0^{\infty} \big(1_{(t_{k-1},t_{k}]} \mathbf{H}_{k}^2\big)[\d \mathbf{S}] = 0 \qquad (k = 1,\ldots,\ell)
\]
because $S_j$ is a martingale for all $j=1,\ldots,n$.
Consequently, if
\[
U_j \coloneqq \sum_{k=1}^{\ell} 1_{(t_{k-1},t_k]} u_j^k \; \text{ and } \; V_j \coloneqq \sum_{k=1}^{\ell} 1_{(t_{k-1},t_k]} v_j^k \qquad (j=1,\ldots,n),
\]
then Theorem \ref{thm.multidimBSineq} and Proposition \ref{prop.mintensornormbound} yield
\begin{align*}
    \norm{\int_0^{\infty} \mathbf{H}(t)[\d \mathbf{S}(t)]}_{\infty}^2 & = \norm{\sum_{j=1}^n\int_0^{\infty} U_j(t) \sh \d S_j(t)}_{\infty}^2 \\
    & \leq 8\int_0^{\infty} \sum_{j=1}^n\norm{U_j(t)}_{\cA \otimes_{\min} \cA^{\op}}^2 \,\d t \\
    & \leq 8\int_0^{\infty} \sum_{j=1}^n\norm{\#(U_j(t)) + \mathfrak{t}(V_j(t))}_{2;2}^2 \,\d t \\
    & = 8\int_0^{\infty} \sum_{j=1}^n\big\|H_j^{\C}(t)\big\|_{2;2}^2 \,\d t,
\end{align*}
which is \eqref{eq.semicircineq}.
\pagebreak

As a result of the previous paragraph and Theorem \ref{thm.constructNCstochint}, we may speak of the stochastic integral $I_{\mathbf{S}} \colon \big(\cI_{\cT}^{2;2}\big)^{1 \times n} \to C_{a,\infty}$.
From its construction (in the proof of Theorem \ref{thm.constructNCstochint}) and \eqref{eq.semicircineq}, it is also easy to see that for all $\mathbf{H} \in \big(\cI_{\cT}^{2;2}\big)^{1 \times n}$ and $t \geq 0$,
\[
\norm{I_{\mathbf{S}}[\mathbf{H}](t)}_{\infty} \leq 2\sqrt{2}\left(\int_0^t \sum_{j=1}^n \big\|H_j^{\C}(s)\big\|_{2;2}^2\,\d s\right)^\frac12,
\]
as desired.
The remaining detail is to check that $I_{\mathbf{S}} \colon $ agrees with the stochastic integral $I_{\mathbf{S}}^{2;2} \colon \big(\cI^{2;2}\big)^{1 \times n} \to \GD_{\scA}^2$ from subsection \ref{subsec.MD} on $\big(\cI_{\cT}^{2;2}\big)^{1 \times n} \subseteq \big(\cI^{2;2}\big)^{1 \times n}$.
Indeed, the set
\[
\cS \coloneqq \big\{\mathbf{H} \in \big(\cI_{\cT}^{2;2}\big)^{1 \times n} : I_{\mathbf{S}}[\mathbf{H}] = I_{\mathbf{S}}^{2;2}[\mathbf{H}]\big\}
\]
contains $\big(\EP_{\cT}^{2;2}\big)^{1 \times n}$ and is closed under locally bounded convergence, so $\cS = \big(\cI_{\cT}^{2;2}\big)^{1 \times n}$.
\end{proof}

\subsection{The classical case}\label{subsec.classical}

Finally, we establish the relationship between our noncommutative notions of predictability and integrator to their classical counterparts.
Doing so requires some rather serious stochastic analysis.
We recommend reviewing subsection \ref{subsec.clintegrators} and consulting \cite{CW1990,RY1999,Bichteler2002,Protter2005} as needed.
Bichteler's perspective and methods from \cite{Bichteler2002} will be particularly important.

For the duration of this subsection, let $(\Om,\sF,(\sF_t)_{t \geq 0},P)$ be a classical filtered probability space satisfying the usual conditions, and take
\[
(\cA,(\cA_t)_{t \geq 0},\E) = (\cB,(\cB_t)_{t \geq 0}, \E_{\scB}) = (L^{\infty}(\Om,\sF,P), (L^{\infty}(\Om,\sF_t,P))_{t \geq 0}, \E_P).
\]
For a measurable space $(S,\sS)$, an $S$-valued stochastic process $X \colon \R_+ \times \Om \to S$ is called (\textbf{jointly}) \textbf{measurable} if $X$ is $\cB_{\R_+} \otimes \sF$--$\sS$ measurable.

At this time, the reader should also review the notation in \eqref{eq.Moperator}, \eqref{eq.HY}, and Example~\ref{ex.clpred}.

\begin{theorem}\label{thm.clpred}
Suppose $1 \leq q \leq p \leq \infty$ and $r = r(p,q) \in [1,\infty]$ satisfies $1/r+1/p = 1/q$.
Let $X \colon \R_+ \times \Om \to \C$ be a stochastic process, and assume that $\sup\left\{\norm{X_s}_r : 0 \leq s \leq t\right\} < \infty$ for all $t \geq 0$.
\begin{enumerate}[font=\normalfont,label=(\roman*)]
    \item If $X$ has a predictable modification and $p > q$ (i.e., $r \neq \infty$), then $H_X$ is an $\M^{p;q}$-predictable linear process.\label{item.Xpred=>HXpred}
    \item If $H_X$ is an $\M^{p;q}$-predictable linear process, then $X$ has a predictable modification.\label{item.HXpred=>Xpred}
\end{enumerate}
\end{theorem}

\begin{proof}
We take both items one at a time.

\ref{item.Xpred=>HXpred} We begin by showing that if $X$ is bounded and predictable, then $H_X \in \cI_{\M^{p;q}}$.
To this end, let $\cH$ be the set of bounded stochastic processes $X$ such that $H_X \in \cI_{\M^{p;q}}$, and let $\sP_0$ be as in \eqref{eq.P0}.
Clearly, $\cH$ is a complex-linear subspace of $\ell^{\infty}(\R_+ \times \Om)$ containing the constant functions.
Using the fact that $J_q\cM_t^{p;q}J_p = \cM_t^{p;q}$, where $J_{\alpha} \colon L^{\alpha}(\E) \to L^{\alpha}(\E)$ is the complex-conjugation operation, it is easy to check that $\cH$ is closed under complex conjugation.
Most importantly, the dominated convergence theorem and the fact that $\cI_{\M^{p;q}} \subseteq \ell_{\loc}^{\infty}(\R_+;B^{p;q})$ is closed under locally bounded convergence guarantee that $\cH \subseteq \ell^{\infty}(\R_+ \times \Om)$ is closed under bounded convergence.
(This is where we use that $r \neq \infty$.)
Since $\{1_A : A \in \sP_0\} \subseteq \cH$ and $\sP_0$ is a $\pi$-system (closed under intersection), the multiplicative system theorem (vid.\ \cite[Thm.\ I.21]{DM1975}) yields that $\cH$ contains all bounded $\sigma(\sP_0)$-measurable functions.
Since $\sigma(\sP_0) = \sP$, $\cH$ contains all bounded predictable processes.
\pagebreak

Finally, note that $H_X = H_Y$ if $Y$ is a modification of $X$, so it suffices to show that if $X$ is predictable and $\sup\left\{\norm{X_s}_r : 0 \leq s \leq t \right\} < \infty$ for all $t \geq 0$, then $H_X \in \cI_{\M^{p;q}}$.
To this end, let $n \in \N$ and $X^n \coloneqq 1_{\{(t,\om) \in \R_+ \times \Om : |X_t(\om)| \leq n\}} X$.
Since $X^n$ is bounded and predictable, $H_{X^n} \in \cI_{\M^{p;q}}$ by the previous paragraph.
Since $|X^n| \leq |X|$ and $X^n \to X$ pointwise as $n \to \infty$, the dominated convergence theorem and the fact that $\sup\left\{\norm{X_s}_r : 0 \leq s \leq t\right\} < \infty$ for all $t \geq 0$ guarantee that $(H_{X^n})_{n \in \N}$ converges locally boundedly to $H_X$.
(Here, we used again that $r \neq \infty$.)
Since $\cI_{\M^{p;q}}$ is closed under locally bounded convergence, we conclude that $H_X \in \cI_{\M^{p;q}}$, as desired.

\ref{item.HXpred=>Xpred} We break the technical proof of this part into four steps.
Suppose $H_X \in \cI_{\M^{p;q}}$.

\textbf{Step 1:} \emph{Reduce to the $p > q$ (i.e., $r \neq \infty$) case.}
To this end, observe that
\begin{equation}
    1 \leq \tilde{q} \leq q \leq p \leq \tilde{p} \leq \infty \implies \cI_{\M^{p;q}} \subseteq \cI_{\M^{\tilde{p},\tilde{q}}}.\label{eq.Mobs}
\end{equation}
(Note, however, the slight abuse of notation that amounts to viewing $B^{p;q}$ as contained in $B^{\tilde{p};\tilde{q}}$.)
Now, assume we know the desired result when $p > q$, i.e., $r \neq \infty$.
Let $X \colon \R_+ \times \Om \to \C$ be a stochastic process such that $\sup\left\{ \norm{X_s}_{\infty} : 0 \leq s \leq t \right\} < \infty$ for all $t \geq 0$ and $H_X \in \cI_{\M^{p;p}}$.
Choose $p_0$ and $p_1$ such that $1 \leq p_1 \leq p \leq p_0 \leq \infty$ and $p_0 > p_1$, and let $r_0 \coloneqq r(p_0,p_1)$ so that $1/r_0 +1/p_0 = 1/p_1$.
Of course, $1 \leq r_0 < \infty$, $\sup\left\{\norm{X_s}_{r_0} : 0 \leq s \leq t \right\} < \infty$ for all $t \geq 0$, and $H_X \in \cI_{\M^{p_0;p_1}}$ by \eqref{eq.Mobs}.
By assumption, we therefore know that $X$ has a predictable modification, as desired.

Henceforth, assume $p > q$, i.e., $r \neq \infty$.

\textbf{Step 2:} \emph{Show that $X$ is $(\sF_{t-})_{t \geq 0}$-adapted, where $\sF_{0-} \coloneqq \sF_0$ and $\sF_{t-} \coloneqq \sigma\big(\bigcup_{0 \leq s < t} \sF_s\big)$ for $t > 0$.}
To this end, recall from Proposition \ref{prop.GpredisGadaptandmeas} that $H_X$ is $\M_-^{p;q}$-adapted.
By a standard measure theory exercise (using that $r \neq \infty$), $\bigcup_{0 \leq s < t} L^r(\Om,\sF_s,P)$ is dense in $L^r(\Om,\sF_{t-},P)$ for all $t > 0$.
It follows that $\M_-^{p;q} = (\{M_a : a \in L^r(\Om,\sF_{t-},P)\})_{t \geq 0}$.
Unraveling the definitions, we see that $X$ has an $(\sF_{t-})_{t \geq 0}$-adapted modification.
Since $(\sF_t)_{t \geq 0}$ is complete, $(\sF_{t-})_{t \geq 0}$ is complete.
Thus, $X$ itself is $(\sF_{t-})_{t \geq 0}$-adapted, as desired.

\textbf{Step 3:} \emph{Show that $X$ has a measurable modification.}
To begin, we claim that the function $x \colon \R_+ \to L^r \coloneqq L^r(\Om,\sF,P)$ defined by $t \mapsto X_t$ is strongly measurable.
Indeed, by Proposition \ref{prop.GpredisGadaptandmeas}, $H_X \colon \R_+ \to \cM_{\infty}^{p;q}$ is strongly measurable (with respect to $\norm{\cdot}_{p;q}$).
Since $\cM_{\infty}^{p;q} \subseteq \{M_a : a \in L^r\} =\vcentcolon \cM$ and the map $M \colon L^r \to \cM$ defined by $a \mapsto M_a$ is an isometric isomorphism, we conclude that $M^{-1} \circ \iota \circ H_X \colon \R_+ \to L^r$ is strongly measurable, where $\iota$ is the inclusion of $\cM_{\infty}^{p;q}$ into $\cM$.
Unraveling the definitions, we see that $x = M^{-1} \circ \iota \circ H_X$, so the claim is true.

Next, let $L^0$ be the space of $P$-a.e.\ equivalence classes of $\sF$--$\cB_{\C}$-measurable functions from $\Om$ to $\C$, endowed with the topology of convergence in probability.
The inclusion $\iota_r \colon L^r \hookrightarrow L^0$ is continuous and thus Borel measurable.
Also, since $x \colon \R_+ \to L^r$ strongly measurable, it is Borel measurable and has a separable image.
Thus, $\iota_r \circ x \colon \R_+ \to L^0$ is Borel measurable and has a separable image.
Since $\iota_r \circ x$ is just the function $\R_+ \ni t \mapsto X_t \in L^0$, we conclude from \cite[Thm.\ 3]{Cohn1972} that $X$ has a measurable modification, as desired.

\textbf{Step 4:} \emph{Complete the proof by showing that $X$ has a predictable modification.}
By the third step, $X$ has a measurable modification, so we may and do assume that $X$ itself is measurable.
By the second step (and the completeness of the filtration), $X$ is also $(\sF_{t-})_{t \geq 0}$-adapted.
We show that these two conditions, measurability and $(\sF_{t-})_{t \geq 0}$-adaptedness, imply the existence of a predictable modification.
By working with real and imaginary parts, we may and do assume without loss of generality that $X$ is real valued.

Let $n \in \N$, and define $X^n \coloneqq 1_{\{(t,\om) \in \R_+ \times \Om : |X_t(\om)| \leq n\}} X$.
Then $X^n$ is bounded and measurable.
By the predictable projection theorem (vid.\ \cite[Thm.\ V.5.6]{RY1999}), a predictable projection ${}^{\mathrm{p}}X^n$ of $X^n$ exists.
For the purposes of this proof, what is important to know about ${}^{\mathrm{p}}X^n$ is that it is a predictable process and for each $t \geq 0$,
\[
{}^{\mathrm{p}}X_t^n = \E[X_t^n \mid \sF_{t-}]
\]
almost surely.
Since $X^n$ is $(\sF_{t-})_{t \geq 0}$-adapted, $X_t^n =\E[X_t^n \mid \sF_{t-}]$.
Thus, ${}^{\mathrm{p}}X^n$ is a predictable modification of $X^n$.

Finally, let $Z \coloneqq \limsup_{n \to \infty} {}^{\mathrm{p}}X^n$ and $Y \coloneqq 1_{\{(t,\om) \in \R_+ \times \Om : |Z(t,\om)| < \infty\}} Z$.
We claim that $Y$ is a predictable modification of $X$.
Certainly, $Y$ is a predictable process.
Now, let $t \geq 0$.
Almost surely, for all $n \in \N$, ${}^{\mathrm{p}}X_t^n = X_t^n$.
Since $X^n \to X$ pointwise as $n \to \infty$, it follows that for $P$-a.e.\ $\om \in \Om$,
\[
Y_t(\om) = Z_t(\om) = \lim_{n \to \infty} {}^{\mathrm{p}}X_t^n(\om) = \lim_{n \to \infty} X_t^n(\om) = X_t(\om).
\]
Thus, $Y$ is, indeed, a modification of $X$.
Finally, we are done.
\end{proof}

\begin{remark}\label{rem.completefiltration}
The proof of Theorem \ref{thm.clpred}\ref{item.Xpred=>HXpred} did not use either of the usual conditions.
They were used in the proof of Theorem \ref{thm.clpred}\ref{item.HXpred=>Xpred} only to invoke the predictable projection theorem.
Though all standard references with which we are familiar assume the filtration is right-continuous as well as complete in the statement of the predictable projection theorem, only the completeness assumption is necessary.
To our knowledge, the blog post \cite{Lowther2017} is the only place in which this is written down.
Thus, Theorem \ref{thm.clpred}\ref{item.HXpred=>Xpred} is true if $(\sF_t)_{t \geq 0}$ is only assumed to be complete.
\end{remark}

\begin{theorem}\label{thm.clvsncstochint}
Suppose $1 \leq p < \infty$ and $1 \leq q \leq \infty$.
Let $X \colon \R_+ \times \Om \to \C$ be an $L^p$ stochastic process, and write $x \colon \R_+ \to L^p(\E)$ for the map $t \mapsto X_t$.
\begin{enumerate}[font=\normalfont,label=(\roman*)]
    \item If $x$ is a noncommutative $L^p$-integrator with respect to $\M^{q;q}$, then $X$ has a modification that is a classical $L^p$-integrator.\label{item.ncint=>clint}
    \item If $X$ is a classical $L^p$-integrator such that $I_X^c[Y] = I_X^c[Z]$ whenever $Y,Z \in \sP_{\mathrm{bct}}$ are modifications of each other, then $x$ is a noncommutative $L^p$-integrator with respect to $\M^{q;q}$.
    Furthermore, if $H \in \cI_{\M^{q;q}}$, then there exists a $Y \in \sP_{\mathrm{bct}}$ such that $H=H_Y$, and $I_x[H](t) = I_X^c[Y]_t$ for all $t \geq 0$ and any such $Y$.\label{item.clint=>ncint}
\end{enumerate}
\end{theorem}

\begin{proof}
We take both items one at a time.

\ref{item.ncint=>clint} By definition of noncommutative integrator, $x \colon \R_+ \to L^p(\E) = L^p(P)$ is continuous.
In other words, if $t \geq 0$, then $\norm{X_t-X_s}_p = \norm{x(t) - x(s)}_p \to 0$ as $s \to t$.
Since $L^p$ convergence implies convergence in probability, $X$ is continuous in probability.

To see that $X$ is a classical $L^p$-integrator, we use Bichteler's characterization of $L^p$-integrators from \cite{Bichteler2002}.
Specifically, it suffices to prove the following contrapositive:
If there is some $t \geq 0$ such that $\big\{I_X^{c,0}[Y]_t : Y \in \sEP \text{ and } |Y| \leq 1\big\}$ is not bounded in $L^p(P)$, then $x$ is not a noncommutative $(\M^{q;q},L^p)$-integrator.
Indeed, if so, there exists a sequence $(Y^n)_{n \in \N}$ in $\sEP$ such that $|Y^n| \leq 1$ and $\big\|I_X^{c,0}[Y^n]_t\big\|_p > n$ for all $n \in \N$.
If $Z^n \coloneqq Y^n/n$ and $H_n \coloneqq H_{Z^n} \in \EP_{\M^{q;q}}$ for all $n \in \N$, then $H_n \to 0$ boundedly as $n \to \infty$, and
\[
\norm{I_x^0[H_n](t)}_p = \big\|I_X^{c,0}[Z^n]_t\big\|_p > 1 \qquad (n \in \N).
\]
Consequently, it is impossible that $I_x^0$ extends to a map defined on $\cI_{\M^{q;q}}$ that is continuous with respect to locally bounded convergence.
In other words, $x$ is not a noncommutative $(\M^{q;q},L^p)$-integrator, as desired.

\ref{item.clint=>ncint} Let $H \in \cI_{\M^{q;q}}$, and define $C_t \coloneqq \sup \big\{ \norm{H(s)}_{q;q} : 0 \leq s \leq t \big\}$ for all $t \geq 0$.
By definition of $\M^{q;q}$, if $t \geq 0$, there exists a $y(t) \in L^{\infty}(\Om,\sF_t,P)$ such that $H(t) = M_{y(t)}$.
Let $U \colon \R_+ \times \Om \to \C$ be a stochastic process such that for all $t \geq 0$, $U_t = y(t)$ in $L^{\infty}(\E) = L^{\infty}(P)$.
Since $H_U = H \in \cI_{\M^{q;q}}$, Theorem \ref{thm.clpred}\ref{item.HXpred=>Xpred} provides a predictable modification $Z$ of $U$.
If
\[
Y \coloneqq Z\,1_{\{(t,\om) \in \R_+ \times \Om : |Z_t(\om)| \leq C_t\}},
\]
then $Y$ is predictable, and $\sup\left\{|Y_s| : 0 \leq s \leq t\right\} \leq C_t$.
We claim that $Y$ is a modification of $Z$.
Indeed, if $t \geq 0$, then $Y_t = Z_t\,1_{\{\om \in \Om : |Z_t(\om)| \leq C_t\}}$.
Since $H_Z = H$, $\norm{Z_t}_{\infty} = \norm{H(t)}_{q;q} \leq C_t$.
Thus, $|Z_t| \leq C_t$ almost surely, so $Y_t = Z_t$ almost surely, as claimed.
In particular, $H = H_Y$.
This establishes that there exists a $Y \in \mathrm{b}\sP_{\loc}$ such that $H = H_Y$ and $|Y_t| \leq C_t$ for all~$t \geq 0$.

Next, for any $Y \in \mathrm{b}\sP_{\loc}$ such that $H=H_Y$, define $I_x[H](t) \coloneqq I_X^c[Y]_t \in L^p(\Om,\sF_t,P)$ for all $t \geq 0$.
This definition is independent of the choice of $Y$ by the assumption that $I_X^c[Y] = I_X^c[Z]$ whenever $Z \in \mathrm{b}\sP_{\loc}$ is a modification of $Y$.
Also, it follows easily from the definitions that if $H \in \EP_{\M^{q;q}}$, then $I_x[H] = I_x^0[H]$.

It remains to show that $I_x[H] \colon \R_+ \to L^p(\E) = L^p(P)$ is continuous for all $H \in \cI_{\M^{q;q}}$ and $I_x \colon \cI_{\M^{q;q}} \to C_a(\R_+;L^p(\E))$ is continuous with respect to locally bounded convergence.
The right-continuity of $I_x[H]$ follows easily from the fact that $I_X^c[Y] \in \cR_{a,p}$, as we encourage the reader to ponder.
To establish the left-continuity of $I_x[H]$, let $t > 0$ and $(s_n)_{n \in \N}$ be a sequence in $\R_+$ such that $s_n \nearrow t$ as $n \to \infty$.
By the bounded convergence theorem for $L^p$-integrators,
\[
I_x[H](t) - I_x[H](s_n) = I_X^c[Y]_t - I_X^c[Y]_{s_n} = I_X^c[Y\,1_{(s_n,t] \times \Om}]_t \xrightarrow{n \to \infty} I_X^c[Y\,1_{\{t\} \times \Om}]_t
\]
in $L^p(\E) = L^p(P)$.
By another bounded convergence theorem argument, the fact that $X$ is continuous in probability implies that for any $Z \in \mathrm{b}\sP_{\loc}$, $I_X^c[Z\,1_{\{t\} \times \Om}]_t = 0$ almost surely.
(Sketch of proof:
First, check this when $Z \in \sEP$.
Then upgrade it to the general case using the multiplicative system theorem and the bounded convergence theorem for $L^p$-integrators.)
Thus, $I_x[H] \colon \R_+ \to L^p(\E)$ is continuous.

Finally, suppose $(H_n)_{n \in \N}$ is a sequence in $\cI_{\M^{q;q}}$ converging locally boundedly to $H$, let $D_t \coloneqq \sup\big\{ \norm{H_n(s)}_{q;q} : 0 \leq s \leq t\big\}$ for all $t \geq 0$, let $(Z^n)_{n \in \N}$ be a sequence in $\mathrm{b}\sP_{\loc}$ such that $|Z_t^n| \leq D_t$ for all $t \geq 0$ and $n \in \N$, and let $Z \in \mathrm{b}\sP_{\loc}$ be such that $|Z_t| \leq D_t$ for all $t \geq 0$.
Define
\[
Y \coloneqq Z\,1_{\left\{(t,\om) \in \R_+ \times \Om : \lim_{k \to \infty} Z_t^k(\om) = Z_t(\om)  \right\}} \; \text{ and } \; Y^n \coloneqq Z^n \,1_{\left\{(t,\om) \in \R_+ \times \Om : \lim_{k \to \infty} Z_t^k(\om) = Z_t(\om)  \right\}}
\]
for all $n \in \N$.
Note that $Y,Y^n \in \mathrm{b}\sP_{\loc}$ and, for each $t \geq 0$, $Y^n\,1_{[0,t] \times \Om} \to Y\,1_{[0,t] \times \Om}$ boundedly as $n \to \infty$ by design.
Next, note that if $t \geq 0$, then
\[
\norm{Z_t^n-Z_t}_{\infty} = \norm{H_{Z^n}(t) - H_Z(t)}_{q;q} = \norm{H_n(t) - H(t)}_{q;q} \xrightarrow{n \to \infty} 0.
\]
In particular, $Z_t^n \to Z_t$ almost surely as $n \to \infty$, i.e.,
\[
\left\{ \om \in \Om : \lim_{n \to \infty} Z_t^n(\om) = Z_t(\om)\right\}
\]
has probability one.
It follows that $Y$ is a modification of $Z$ and $Y^k$ is a modification of $Z^k$.
Consequently, by the bounded convergence theorem for $L^p$-integrators,
\[
\sup_{0 \leq s \leq t}\norm{I_x[H_n](s) - I_x[H](s)}_p = \sup_{0 \leq s \leq t} \norm{I_X^c[Y^n]_s - I_X^c[Y]_s} \leq \norm{I_X^c[Y^n-Y]_t^*}_p \xrightarrow{n \to \infty} 0,
\]
as desired.
This completes the proof.
\end{proof}

\begin{remark}\label{rem.fullconverse?}
We suspect that the full conclusion in \ref{item.ncint=>clint} should be that $X$ has a modification $U$ that is a classical $L^p$-integrator with the property that $I_U^c[Y] = I_U^c[Z]$ whenever $Y,Z \in \sP_{\mathrm{bct}}$ are modifications of each other and that $I_U^c[Y]_t = I_x[H_Y](t)$ for all $t \geq 0$ and $Y \in \sP_{\mathrm{bct}}$ such that $H_Y \in \M^{q;q}$.
However, it seems that proving this would require reworking much of the development in \cite[Ch.\ 3]{Bichteler2002}.
We have opted not to attempt to do so in this paper, as it would take us too far afield.
\end{remark}

\begin{example}\label{ex.Doleanmeasac}
Let $X \colon \R_+ \times \Om \to \R$ be a continuous $L^2$-martingale and $\mu_X$ be its Dol\'eans measure (vid.\ \cite[\S2.4]{CW1990}).
If $\mu_X \ll \lambda \otimes P$, where $\lambda$ is the Lebesgue measure on $\R_+$, then $I_X^c[Y] = I_X^c[Z]$ whenever $Y,Z \in \mathrm{b}\sP_{\loc}$ are modifications of each other.
Equivalently, $I_X^c[Y] = 0$ in $\cR_{a,2}$ whenever $Y \in \mathrm{b}\sP_{\loc}$ is a modification of the zero process.
Indeed, if $Y$ is a modification of the zero process, then $Y=0$ $\lambda\otimes P$-a.e.
By assumption, this means $Y = 0$ $\mu_X$-a.e.
Consequently, by the It\^o isometry,
\[
\norm{I_X^c[Y]_t}_{L^2(\Om,\sF,P)} = \norm{Y\,1_{[0,t] \times \Om}}_{L^2(\R_+ \times \Om,\sP,\mu_X)} = 0.
\]
Since RCLL processes are modifications of each other if and only if they are indistinguishable from each other, it follows that $I_X^c[Y] = 0$ in $\cR_{a,p}$.
Thus, Theorem \ref{thm.clvsncstochint}\ref{item.clint=>ncint} applies to $X$.
In particular, since the Dol\'eans measure of Brownian motion is $\lambda \otimes P$, Theorem \ref{thm.clvsncstochint}\ref{item.clint=>ncint} applies to Brownian motion.
\end{example}

Actually, by choosing a different norm, one can do better in the case of Brownian motion.
To see this, we recall the classical Burkholder--Davis--Gundy (BDG) inequalities.

\begin{theorem}[BDG inequalities {\cite[Thm.\ 1.1]{MR2016}}]\label{thm.BDG}
There exist increasing families $(\gamma_p)_{p \geq 1}$ and $(\delta_p)_{p \geq 1}$ of strictly positive constants such that if $1 < p < \infty$ and $M \colon \R_+ \times \Om \to \C$ is an RCLL local martingale such that $M_0 = 0$, then
\[
\gamma_p^{-1}\norm{\big[\overline{M},M\big]_t}_{\frac{p}{2}}^{\frac12} \leq \norm{\left|M\right|_t^*}_{p} \leq \delta_p\norm{\big[\overline{M},M\big]_t}_{\frac{p}{2}}^{\frac12} \qquad (t \geq 0),
\]
where $\left|M\right|_t^* = \sup\left\{ |M_s| : 0 \leq s \leq t\right\}$ and
\[
\big[\overline{M},M\big]_t = L^0\text{-}\lim_{\Pi \in \cP_{[0,t]}} \sum_{s \in \Pi} \Delta_s\overline{M} \Delta_sM = L^0\text{-}\lim_{\Pi \in \cP_{[0,t]}} \sum_{s \in \Pi} |\Delta_sM|^2 \qquad (t \geq 0)
\]
is the quadratic covariation of $\overline{M}$ with $M$.
(Above, $L^0$\text{-}$\lim$ denotes a limit in probability.)
\end{theorem}

Now, suppose $2 \leq p  < \infty$ and $B \colon \R_+ \times \Om \to \R$ is a standard real Brownian motion.
If $Y \in \sEP$, then $M \coloneqq I_B^c[Y] = \into Y\,\d B$ is a continuous $L^p$-martingale, and
\[
\big[\overline{M},M\big]_t = \int_0^t |Y_s|^2 \,\d s \qquad (t \geq 0).
\]
Consequently, if $b \colon \R_+ \to L^p(\E)$ is the map $t\mapsto B_t$ and $H_Y \colon \R_+ \to B^{p;1}$ is the map $t \mapsto M_{Y_t} = (x \mapsto Y_tx)$, then BDG inequalities and Minkowski's integral inequality yield
\begin{align*}
    \norm{I_b[H_Y](t)}_p & \leq \norm{\left| \into Y\,\d B \right|_t^*}_p \leq \delta_p\norm{\int_0^t |Y_s|^2\,\d s}_{\frac{p}{2}}^\frac12 \\
    & \leq \delta_p\left(\int_0^t \norm{Y_s}_p^2\,\d s\right)^\frac12 = \delta_p\left(\int_0^t \norm{H_Y(s)}_{\infty;p}^2\,\d s\right)^\frac12
\end{align*}
for all $t \geq 0$.
It follows from Theorem \ref{thm.constructNCstochint} that $b$ is an $L^p$-integrator with respect to $\M^{\infty;p}$, viewed as a compatible normed filtration of $B^{p;1}$ with the norm $\norm{M_y} \coloneqq \norm{y}_p = \norm{M_y}_{\infty;p}$.
(Note that $M_yx = yx \in L^1(P)$ whenever $x,y \in L^p(P)$ because $p \geq 2$, so viewing $\M^{\infty;p}$ as a compatible normed filtration of $B^{p;1}$ is possible.)
Also, write $\sP^p$ for the set of all predictable processes $Y \colon \R_+ \times \Om \to \C$ with the property that $\sup\big\{ \norm{Y_s}_p : 0 \leq s \leq t\big\} < \infty$ for all $t \geq 0$.
Since
\[
\{H_Y : Y \in \mathrm{b}\sP_{\loc}\} \subseteq \left\{ H_Y : Y \in \sP^p\right\}=  \cI_{\M^{\infty;p}}
\]
by Theorem \ref{thm.clpred}, a multiplicative-system-theorem argument yields that $I_b[H_Y](t) = I_B^c[Y]_t$ for all $Y \in \mathrm{b}\sP_{\loc}$ and $t \geq 0$.
Actually, $I_b$ and $I_B^c$ ``agree on all of $\cI_{\M^{\infty;p}}$'' in the sense that, by the above estimates, $I_B^c$ extends uniquely to a linear map $\tilde{I}_B^c \colon \sP^p \to \cR_{a,p}$ that is continuous with respect to locally ($L^p$ norm--)bounded convergence and $I_b[H_Y](t) = I_B^c[Y]_t$ for all $t \geq 0$ and $Y \in \sP^p$.
We leave the details to the reader.

\section{Noncommutative stochastic integral equations in \texorpdfstring{$L^p$}{}}\label{sec.SIEp}

For the duration of this section, fix two indices $p_0,p \in [1,\infty]$, two conditionable filtered $\mathrm{C}^*$-probability spaces $(\cB,(\cB_t)_{t \geq 0},\E_{\scB})$ and $(\cC,(\cC_t)_{t \geq 0},\E_{\scC})$, and a $\beta \in \{\emptyset,\sa\}$ indicating a choice of $L^p(\E_{\scC})_{\emptyset} \coloneqq L^p(\E_{\scC})$ or $L^p(\E_{\scC})_{\sa}$.
Also, we write
\[
L_{\alg}^p(\cC_{\infty},\E_{\scC}) \coloneqq \bigcup_{t \geq 0} L^p(\cC_t,\E_{\scC})
\]
and endow $L_{\alg}^p(\cC_{\infty},\E_{\scC})$ with the topology of the $L^p$ norm.

\subsection{Setup and statement of the main result}\label{subsec.SIEpsetup}

We now set up our main result on noncommutative stochastic integral equations in $L^p$ (Theorem \ref{thm.SIE} below, proven in the next subsection).
In the next section, we study a few large classes of examples, including those discussed in subsection \ref{subsec.NCSDEs}.

To begin, we explain what it means to solve a noncommutative stochastic integral equation (driven by a continuous integrator).
To this end, we extend some notation and terminology from earlier in the paper.
Let $I \subseteq \R_+$ be a sub-interval.
A process $Y \colon I \to L^p(\E_{\scC})_{\beta}$ is \textbf{adapted} if $Y(t) \in L^p(\cC_t,\E_{\scC})_{\beta}$ for all $t \in I$.
Write $C_{a,p}(I) = C_a(I;L^p(\E_{\scC})_{\beta})$ for the space of continuous adapted functions $Y \colon I \to L^p(\E_{\scC})_{\beta}$.
Next, let $(\G,\norm{\cdot})$ be a normed filtration of $B^{p_0;1} = B^{p_0;1}(\E_{\scB};\E_{\scC})$.
A linear process $H \colon I \to \cG_{\infty}$ is \textbf{$\boldsymbol{\G}$-adapted} if $H(t) \in \cG_t$ for all $t \in I$.
Finally, suppose $I = [r,T] \cap \R_+$ or $I = (r,T] \cap \R_+$ with $0 \leq r < T \leq \infty$.
A $\G$-adapted linear process $H \colon I \to \cG_{\infty}$ is \textbf{$\boldsymbol{\G}$-predictable}, written $H \in \cI_{\G}(I)$, if $1_{(r,T]}H \in \cI_{\G}$, in which case we define
\[
\int_s^u H[\d X] \coloneqq \int_s^u (1_{(r,T]}H)[\d X] \qquad (s,u \in I, \; s \leq u)
\]
whenever $X \colon \R_+ \to L^{p_0}(\E_{\scB})$ is an $L^p$-integrator with respect to $\G$.

\begin{definition}[Solutions to NCSIEs]\label{def.SIEsol}
Suppose $(\G,\norm{\cdot})$ is a compatible normed filtration of $B^{p_0;1} = B^{p_0;1}(\E_{\scB};\E_{\scC})$ and $X \colon \R_+ \to L^{p_0}(\E_{\scB})$ is an $L^p$-integrator with respect to $\G$.
Let $F \colon \R_+ \times L_{\alg}^p(\cC_{\infty},\E_{\scC})_{\beta} \to \cG_{\infty}$ be a continuous function with the property that $F(t,y) \in \cG_t$ for all $t \geq 0$ and $y \in L^p(\cC_t,\E_{\scC})_{\beta}$, and let $C \in C_a(\R_+;L^p(\E_{\scC}))$.

A \textbf{local solution} to the \textbf{noncommutative stochastic integral equation} (NCSIE)
\begin{equation}
    Y = C + \into F(t,Y(t))[\d X(t)]\label{eq.SIE}
\end{equation}
is an interval $I \subseteq \R_+$  containing $0$ and a process $Y \in C_a(I;L^p(\E_{\scC})_{\beta})$ such that 
\[
Y(t) = C(t) + \int_0^t F(s,Y(s))[\d X(s)] \qquad (t \in I).
\]
If $I = \R_+$, then $Y$ is a \textbf{global solution} to \eqref{eq.SIE}.

If $C \equiv C(0) =\vcentcolon y_0$, then the NCSIE
\[
Y = C + \into F(t,Y(t))[\d X(t)] = y_0 + \into F(t,Y(t))[\d X(t)]
\]
is called a \textbf{noncommutative stochastic differential equation} (NCSDE) and is written
\[
\begin{cases}
    \d Y(t) = F(t,Y(t))[\d X(t)] & \\
    \;\, Y(0) = y_0
\end{cases}
\]
in ``stochastic differential notation.''
\end{definition}

\begin{remark}\label{rem.dt}
The reader may be disturbed by the lack of a ``$\d t$ term'' in \eqref{eq.SIE}.
There is, however, nothing to worry about.
Under reasonable conditions, such a term may be included in the formalism by considering an integrator of the form $Z(t) \coloneqq (X(t),t)$, which takes values in $L^{p_0}$ of the $\mathrm{C}^*$-probability space $\cB \oplus \C$ with the direct sum trace (and filtration), and a coefficient of the form
\[
H(t,y)[z] \coloneqq F(t,y)[x] + G(t,y)\lambda \quad \big(t \geq  0, \; y \in L_{\alg}^p(\cC_{\infty},\E_{\scC})_{\beta}, \; z = (x,\lambda) \in L^{p_0}(\E_{\scB}) \times \C\big),
\]
where $G \colon \R_+ \times L_{\alg}^p(\cC_{\infty},\E_{\scC})_{\beta} \to L^p(\E_{\scC})$ is a continuous function with the property that $G(t,y) \in L^p(\cC_t,\E_{\scC})$ for all $t \geq 0$ and $y \in L^p(\cC_t,\E_{\scC})_{\beta}$.
In this case, consider the NCSIE
\[
Y = C + \into H(t,Y(t))[\d Z(t)] = C + \into F(t,Y(t))[\d X(t)] + \into G(t,Y(t))\,\d t.
\]
We leave the details to the reader.
\end{remark}

We now lay out the assumptions for our main result on NCSIEs in $L^p$, one set of assumptions on the driving integrator and another on the coefficient of the NCSIE.
They will involve two different filtrations with two different norms;
we discussed the reason this is useful in the paragraph after the statement of Theorem \ref{thm.SDESCBMdriver}.
\pagebreak

\begin{assumption}[Integrator assumptions]\label{ass.integrator}
We have:
\begin{enumerate}[font=\normalfont,label=(\roman*)]
    \item $X \in C_a(\R_+;L^{p_0}(\E_{\scB}))$;\label{item.Xass}
    \item $(\H,\norm{\cdot}_{\H})$ and $(\G,\norm{\cdot}_{\G})$ are compatible normed filtrations of $B^{p_0;1} = B^{p_0;1}(\E_{\scB};\E_{\scC})$ such that $\cG_t \subseteq \cH_t$ for all $t \geq 0$ and $\norm{\cdot}_{\H} \leq \norm{\cdot}_{\G}$ on $\cG_{\infty}$;\label{item.HandGass}
    \item there exists a $q$ such that $1 \leq q \leq p \leq \infty$, a locally finite, atomless Borel measure $\rho$ on $\R_+$, and an increasing family $(a_T)_{T \geq 0}$ of positive real numbers such that
    \[
    \norm{\int_0^t H[\d X]}_q \leq a_T\left(\int_{(0,t]} \norm{H}_{\H}^2\,\d\rho \right)^\frac12 \; \text{ and } \; \norm{\int_0^t G[\d X]}_p \leq a_T\left(\int_{(0,t]} \norm{G}_{\G}^2\,\d\rho \right)^\frac12
    \]
    for all $T \geq t \geq 0$, $H \in \EP_{\H}$, and $G \in \EP_{\G}$;\label{item.integralboundass}
    \item if $p=\infty$ and $q < \infty$, then $(\cC,(\cC_t)_{t \geq 0},\E_{\scC})$ is a filtered $\mathrm{W}^*$-probability space.\label{item.W*ass}
\end{enumerate}
Henceforth, we shall write $\int_s^t \boldsymbol{\cdot} \,\d\rho \coloneqq \int_{(s,t]} \boldsymbol{\cdot} \,\d\rho$ ($0 \leq s \leq t$).
Note that, since $\rho$ is atomless, $\int_s^t \boldsymbol{\cdot}\,\d\rho = \int_{(s,t)} \boldsymbol{\cdot} \,\d\rho = \int_{[s,t)} \boldsymbol{\cdot} \,\d\rho = \int_{[s,t]} \boldsymbol{\cdot} \,\d\rho$ ($0 \leq s \leq t$).
\end{assumption}

Under Assumption \ref{ass.integrator}, $X$ is a continuous $L^q$-integrator with respect to $\H$ and a continuous $L^p$-integrator with respect to $\G$ by Theorem \ref{thm.constructNCstochint}.
Furthermore, using Assumption \ref{ass.integrator}\ref{item.HandGass}, it is easy to see that $\cI_{\G} \subseteq \cI_{\H}$ and $I_X^{\H}[G] = I_X^{\G}[G]$ for all $G \in \cI_{\G}$.
Consequently, we may suppress $\H$ and $\G$ from the notation for stochastic integrals against $X$ without issues.

\begin{assumption}[Coefficient assumptions]\label{ass.coefficient}
Let $(\G,\norm{\cdot})$ and $(\H,\norm{\cdot}_{\H})$ be compatible normed filtrations of $B^{p_0;1} = B^{p_0;1}(\E_{\scB};\E_{\scC})$ and $X \colon \R_+ \to L^{p_0}(\E_{\scB})$ be an $L^p$-integrator with respect to $\G$.
We have:
\begin{enumerate}[font=\normalfont,label=(\roman*)]
    \item $F \colon \R_+ \times L_{\alg}^p(\cC_{\infty},\E_{\scC})_{\beta} \to \cG_{\infty}$ is continuous;
    \item $F(t,y) \in \cG_t$ for all $t \geq 0$ and $y \in L^p(\cC_t,\E_{\scC})_{\beta}$;
    \item if $t \geq r \geq 0$, then $\displaystyle \int_r^t F(s,Y(s))[\d X(s)] \in L^p(\cC_t,\E_{\scC})_{\beta}$ for all $Y \in C_a([r,t];L^p(\E_{\scC})_{\beta})$;\label{item.sareq}
    \item for all $R \geq 0$, there exists an increasing family $(c_{T,R})_{T \geq 0}$ of positive real numbers such that for all $y,z \in L_{\alg}^p(\cC_{\infty},\E_{\scC})_{\beta}$ with $\norm{y}_p \leq R$ and $\norm{z}_p \leq R$,\label{item.Lip}
    \[
    \norm{F(t,y) - F(t,z)}_{\H} \leq c_{T,R}\norm{y-z}_q \qquad (0 \leq t \leq T);
    \]
    \item for all $T,R \geq 0$,\label{item.bddass}
    \[
    M_{T,R} \coloneqq \sup\left\{ \norm{F(t,y)}_{\G} : 0 \leq t \leq T, \, y \in L_{\alg}^p(\cC_{\infty},\E_{\scC})_{\beta},  \,\norm{y}_p \leq R \right\} <\infty.
    \]
\end{enumerate}
\end{assumption}

\begin{remark}\label{rem.ass}
If $\beta = \emptyset$, then Assumption \ref{ass.coefficient}\ref{item.sareq} is devoid of content.
If $\beta = \sa$,  $X^*=X$, and $F(t,y)[x]^* = F(t,y)[x^*]$ for all $t \geq 0$, $y \in L^p(\cC_t,\E_{\scC})_{\sa}$, and $x \in L^{p_0}(\cB,\E_{\scB})$, then Assumption \ref{ass.coefficient}\ref{item.sareq} holds.
Finally, if $q=p$ and $(\H,\norm{\cdot}_{\H}) = (\G,\norm{\cdot}_{\G})$, then Assumption~\ref{ass.coefficient}\ref{item.bddass} is redundant.
Indeed, in this case, if $0 \leq t \leq T$, $y \in L_{\alg}^p(\cC_{\infty},\E_{\scC})_{\beta}$, and $\norm{y}_p \leq R$, then
\begin{align*}
    \norm{F(t,y)}_{\G} & \leq \norm{F(t,y) - F(t,0)}_{\G} + \norm{F(t,0)}_{\G} \\
    & \leq c_{T,R}\norm{y}_p + \norm{F(t,0)}_{\G} \leq c_{T,R}R + \norm{F(t,0)}_{\G}.
\end{align*}
Thus, $M_{T,R} \leq c_{T,R}R + \sup\left\{\norm{F(t,0)}_{\G} : 0 \leq t \leq T\right\} < \infty$.
\end{remark}

\begin{theorem}[Existence and uniqueness of maximal solutions to NCSIEs in $L^p$]\label{thm.SIE}
Under {\rm Assumptions \ref{ass.integrator}} and {\rm \ref{ass.coefficient}}, if $C \in C_a(\R_+;L^p(\E_{\scC})_{\beta})$, then there exists a $T_{\ell} \in (0,\infty]$ and a local solution $Y \colon [0,T_{\ell}) \to L^p(\E_{\scC})_{\beta}$ to NCSIE \eqref{eq.SIE} such that
\begin{enumerate}[label=(\roman*),font=\normalfont]
    \item $Y$ is maximal and unique in the sense that if $Y_0 \colon I \to L^p(\E_{\scC})_{\beta}$ is another local solution to \eqref{eq.SIE}, then $I \subseteq [0,T_{\ell})$, and $Y|_I = Y_0$; and\label{item.maxsol}
    \item if $T_{\ell} < \infty$, then $\norm{Y(t)}_p \to \infty$ as $t \nearrow T_{\ell}$.\label{item.blowup}
\end{enumerate}
Owing to {\rm\ref{item.maxsol}}, $Y$ is called the \textbf{maximal solution} to \eqref{eq.SIE}, and $T_{\ell}$ is its \textbf{lifetime}.
Owing to {\rm\ref{item.blowup}}, when $T_{\ell} < \infty$, we also call $T_{\ell}$ the \textbf{blow-up} or \textbf{explosion time} of $Y$.
\end{theorem}

As a consequence of the blow-up phenomenon in Theorem \ref{thm.SIE}\ref{item.blowup}, if one imposes a sublinear growth condition on $F$ in addition to imposing Assumptions \ref{ass.integrator} and \ref{ass.coefficient}, then the maximal solution to \eqref{eq.SIE} is global.

\begin{corollary}\label{cor.sublingrowth}
Under {\rm Assumptions \ref{ass.integrator}} and {\rm\ref{ass.coefficient}}, if for each $T \geq 0$, there exists an $M_T < \infty$ such that
\[
\sup_{0 \leq t \leq T}\norm{F(t,y)}_{\G} \leq M_T\left(\norm{y}_p + 1\right) \qquad \big(y \in L_{\alg}^p(\cC_{\infty},\E_{\scC})_{\beta}\big),
\]
then the maximal solution to \eqref{eq.SIE} is global, i.e., $T_{\ell} = \infty$.
\end{corollary}

\begin{proof}
Suppose $0 < T < \infty$ and $Y \colon [0,T) \to L^p(\E_{\scC})_{\beta}$ is a local solution to \eqref{eq.SIE}.
If $0 \leq t < T$, then
\begin{align*}
    \norm{Y(t)}_p^2 & = \norm{C(t) + \int_0^t F(s,Y(s))[\d X(s)]}_p^2 \\
    & \leq 2\norm{C(t)}_p^2 + 2a_T^2\int_0^t \norm{F(s,Y(s))}_{\G}^2\,\rho(\d s) \\
    & \leq 2\norm{C(t)}_p^2 + 2a_T^2M_T^2\int_0^t \left(1+\norm{Y(s)}_p\right)^2\rho(\d s) \\
    & \leq 2\norm{C(t)}_p^2 + 4a_T^2M_T^2\int_0^t \left(1+\norm{Y(s)}_p^2\right)\rho(\d s) \\
    & = 2\norm{C(t)}_p^2 + 4a_T^2M_T^2\rho([0,t])+4a_T^2M_T^2\int_0^t \norm{Y(s)}_p^2\,\rho(\d s).
\end{align*}
Consequently, by Gr\"onwall's inequality (vid.\ \cite[Lem.\ 13]{Seifert2016}), if we write
\[
\alpha(t) \coloneqq 2\norm{C(t)}_p^2 + 4a_T^2M_T^2\rho([0,t]) \qquad (t \geq 0),
\]
then
\begin{align*}
    \norm{Y(t)}_p^2 & \leq \alpha(t)+4a_T^2M_T^2\int_0^t \alpha(s)\,e^{4a_T^2M_T^2\rho([s,t])}\,\rho(\d s) \\
    & \leq \sup_{0 \leq s \leq T}\alpha(s)\left(1+4a_T^2M_T^2e^{4a_T^2M_T^2\rho([0,T])}\rho([0,T])\right)
\end{align*}
whenever $0 \leq t < T$.
Thus,
\[
\sup_{0 \leq t < T}\norm{Y(t)}_p \leq \sqrt{\sup_{0 \leq s \leq T}\alpha(s)\left(1+4a_T^2M_T^2e^{4a_T^2M_T^2\rho([0,T])}\rho([0,T])\right)} < \infty.
\]
In particular, $\norm{Y(t)}_p \not\to \infty$ as $t \nearrow T$.
It follows from item \ref{item.blowup} of Theorem \ref{thm.SIE} that $T_{\ell} = \infty$ in this case, as desired.
\end{proof}

Note that if $q=p$, $(\H,\norm{\cdot}_{\H}) = (\G,\norm{\cdot}_{\G})$, and we can take $c_{T,R} = c_T$ to be independent of $R$ in Assumption \ref{ass.coefficient}\ref{item.Lip}, then the sublinear growth condition in Corollary \ref{cor.sublingrowth} is automatic.
Indeed, in this case, if $0 \leq t \leq T$ and $y \in L_{\alg}^p(\cC_{\infty},\E_{\scC})_{\beta}$, then
\[
\norm{F(t,y)}_{\G} \leq  \norm{F(t,y) - F(t,0)}_{\G} + \norm{F(t,0)}_{\G} \leq c_T \norm{y}_p + \sup_{0 \leq s \leq T}\norm{F(s,0)}_{\G},
\]
so we may take $M_T \coloneqq \max\left\{c_T,\sup\left\{\norm{F(t,0)}_{\G} : 0 \leq t \leq T\right\}\right\}$ in Corollary \ref{cor.sublingrowth}.

\begin{example}
Let $(\cA,(\cA_t)_{t \geq 0},\E)$ be a conditionable filtered $\mathrm{C}^*$-probability space, and take $(\cB,(\cB_t)_{t \geq 0},\E_{\scB}) = (\cC,(\cC_t)_{t \geq 0},\E_{\scC}) = (\cA,(\cA_t)_{t \geq 0},\E)$.
Suppose $G \colon \R_+ \times L^p(\E)_{\beta} \to L^p(\E)_{\beta}$ is continuous; $G(t,y) \in L^p(\cA_t,\E)_{\beta}$ for all $t \geq 0$ and $y \in L^p(\cA_t,\E)_{\beta}$; and for every $R \geq 0$, there exists an increasing family $(c_{T,R})_{T \geq 0}$ of positive real numbers such that for all $y,z \in L^p(\E)_{\beta}$ with $\norm{y}_p \leq R$ and $\norm{z}_p \leq R$,
\[
\norm{G(t,y) - G(t,z)}_p \leq c_{T,R}\norm{y-z}_p \qquad (0 \leq t \leq T).
\]
Let $C \in C_a(\R_+;L^p(\E)_{\beta})$.
As a consequence of Theorem \ref{thm.SIE}, there exists a unique maximal solution to
\[
Y = C + \into G(t,Y(t))\,\d t.
\]
More precisely, there exists a time $T_{\ell} < \infty$ and a process $Y \in C_a([0,T_{\ell});L^p(\E)_{\beta})$ such that $Y(t) = C(t) + \int_0^t G(s,Y(s))\,\d s$ for all $t \in [0,T_{\ell})$;
and if $I \subseteq \R_+$ is another sub-interval containing $0$ and $Z \in C_a(I;L^p(\E)_{\beta})$ satisfies $Z = C+\into G(t,Y(t))\,\d t$ on $I$, then $I \subseteq [0,T_{\ell})$, and $Z = Y|_I$.
Furthermore, if $T_{\ell} < \infty$, then $\norm{Y(t)}_p \to \infty$ as $t \nearrow T_{\ell}$.
Finally, by Corollary \ref{cor.sublingrowth}, if for every $T \geq 0$, there exists an $M_T < \infty$ such that
\[
\sup_{0 \leq t \leq T}\norm{G(t,y)}_p \leq M_T\left(\norm{y}_p+1\right) \qquad \big(y \in L^p(\E)_{\beta}\big),
\]
e.g., if $c_{T,R}$ can be taken to be independent of $R$, then $T_{\ell} = \infty$.
Very special cases of such NCSIEs were considered in \cite{BS2001,Gao2008}.
\end{example}

\subsection{Proof of Theorem \ref{thm.SIE}}\label{sec.pfSIEp}

Throughout this subsection, Assumptions \ref{ass.integrator} and \ref{ass.coefficient} will be in force.

The key to Theorem \ref{thm.SIE} is a quantitative local existence result (Theorem \ref{thm.exist} below).
It will allow us to show that (1) the NCSIE \eqref{eq.SIE} always admits a local solution defined on an interval of non-zero length and (2) a local solution to \eqref{eq.SIE} that does not blow up in finite time can be extended to a local solution on a larger interval.

\begin{theorem}[Quantitative local existence]\label{thm.exist}
Suppose $R > 0$ and $r \geq 0$.
If $T > r$ and $D \in C_a([r,\infty);L^p(\E_{\scC})_{\beta})$ satisfy
\[
a_T^2M_{T,R}^2\rho([r,T]) \leq \frac{R^2}{4} \; \text{ and } \; \sup_{r \leq t \leq T}\norm{D(t)}_p  \leq \frac{R}{2},
\]
then there exists a $Y \in C_a([r,T];L^p(\E_{\scC})_{\beta})$ such that
\begin{equation}
    Y(t) = D(t) + \int_r^t F(s,Y(s))[\d X(s)] \qquad (r \leq t \leq T).\label{eq.SIEfromr}
\end{equation}
\end{theorem}

\begin{proof}
Define $\Phi \colon C_a([r,\infty);L^p(\E_{\scC})_{\beta}) \to C_a([r,\infty);L^p(\E_{\scC})_{\beta})$ by
\[
\Phi(Z) \coloneqq D + \int_r^{\boldsymbol{\cdot}} F(t,Z(t))[\d X(t)].
\]
Note that $\Phi(Z)(t) \in L^p(\cC_t,\E_{\scC})_{\beta}$ for all $t \geq r$ by Assumption \ref{ass.integrator}\ref{item.sareq}.

If $Z \in C_a([r,\infty);L^p(\E_{\scC})_{\beta})$ and $\sup\big\{\norm{Z(t)}_p : r \leq t \leq T\big\} \leq R$, then
\begin{align*}
    \sup_{r \leq t \leq T} \norm{\Phi(Z)(t)}_p^2 & \leq 2\sup_{r \leq t \leq T}\norm{D(t)}_p^2 + 2\sup_{r \leq t \leq T}\norm{\int_r^tF(s,Z(s))[\d X(s)]}_p^2 \\
    & \leq \frac{R^2}{2} + 2a_T^2\int_r^T\norm{F(s,Z(s))}_{\G}^2\,\rho(\d s)\\
    & \leq \frac{R^2}{2} + 2a_T^2\rho([r,T])\sup_{r \leq t \leq T}\norm{F(t,Z(t))}_{\G}^2 \\
    & \leq \frac{R^2}{2} + 2a_T^2M_{T,R}^2\rho([r,T]) \leq R^2.
\end{align*}
Consequently, if $n \in \N_0$, then $\sup\big\{\norm{\Phi^{\circ n}(Z)(t)}_p : r \leq t \leq T\big\} \leq R$ as well.

Taking $Z=D$ in the previous paragraph and $(Z_n)_{n \in \N_0} \coloneqq \big(\Phi^{\circ n}(D)\big)_{n \in \N_0}$, we have that
\[
\sup_{r \leq t \leq T} \norm{Z_n(t)}_p \leq R \qquad (n \in \N_0).
\]
We claim that $(Z_n|_{[r,T]})_{n \in \N_0}$ converges in $C_a([r,T];L^q(\E_{\scC})_{\beta})$.
To this end, let $n \in \N$.
If $r \leq t \leq T$, then
\begin{align*}
    \norm{Z_{n+1}(t) - Z_n(t)}_q^2 & = \norm{\Phi(Z_n)(t) - \Phi(Z_{n-1})(t)}_q^2 \\
    & = \norm{\int_r^t(F(s,Z_n(s)) - F(s,Z_{n-1}(s)))[\d X(s)]}_q^2 \\
    & \leq a_T^2\int_r^t \norm{F(s,Z_n(s)) - F(s,Z_{n-1}(s))}_{\H}^2\,\rho(\d s) \\
    & \leq a_T^2\,c_{T,R}^2\int_r^t \norm{Z_n(s) - Z_{n-1}(s)}_q^2\,\rho(\d s),
\end{align*}
where $c_{T,R}$ comes from Assumption \ref{ass.coefficient}\ref{item.Lip}.
Applying this estimate inductively yields
\begin{align*}
    \sup_{r \leq t \leq T} & \norm{Z_{n+1}(t) - Z_n(t)}_q^2 \\
    & \leq a_T^{2n}c_{T,R}^{2n}\int_r^T \int_r^{t_1}\cdots \int_r^{t_{n-1}} \norm{Z_1(t_n) - Z_0(t_n)}_q^2\,\rho(\d t_n)\cdots \rho(\d t_2)\,\rho(\d t_1) \\
    & \leq 4R^2a_T^{2n}c_{T,R}^{2n}\int_r^T \int_r^{t_1}\cdots \int_r^{t_{n-1}} \rho(\d t_n)\cdots \rho(\d t_2)\,\rho(\d t_1) \\
    & \leq 4R^2a_T^{2n}c_{T,R}^{2n}\frac{\rho([r,T])^n}{n!}.
\end{align*}
Please see part (ii) of the proof of \cite[Lem.\ 13]{Seifert2016} for the last inequality (and recall that $\rho$-integrals over $[s,t]$ are the same $\rho$-integrals over $(s,t)$ because $\rho$ is atomless).
Therefore,
\[
\sum_{n=0}^{\infty}\sup_{r \leq t \leq T} \norm{Z_{n+1}(t) - Z_n(t)}_q \leq 2R\sum_{n=0}^{\infty}\frac{\left(a_Tc_{T,R}\sqrt{\rho([r,T])}\right)^n}{\sqrt{n!}} < \infty.
\]
It follows that $(Z_n|_{[r,T]})_{n \in \N}$ is Cauchy and thus convergent in $C_a([r,T];L^q(\E_{\scC})_{\beta})$, as claimed.

Let $Y \colon [r,T] \to L^q(\E_{\scC})_{\beta}$ be the uniform limit of $\big(Z_n|_{[r,T]}\big)_{n \in \N_0}$.
We claim that $Y$ does the job.
To begin to see this, we must first show that $Y \in C_a([r,T];L^p(\E_{\scC})_{\beta})$.
If $p=q$, then there is nothing to prove, so we assume $p > q$.
In this case, we first argue that $Y(t) \in L^p(\cC_t,\E_{\scC})_{\beta}$ whenever $r \leq t \leq T$.
It suffices to show that $Y(t) \in L^p(\cC_t,\E_{\scC})$ whenever $r \leq t \leq T$.
To this end, note that $\sup \big\{ \norm{Z_n(t)}_p : 0 \leq t \leq T, \, n \in \N\big\} \leq R$.
Also, if $r \leq t \leq T$, then $(Z_n(t))_{n \in \N}$ converges to $Y(t)$ in $L^q(\cC_t,\E_{\scC})$.
Consequently, if $a \in \cC_t$ and $p' \in [1,\infty)$ is the H\"older conjugate of $p$, then
\[
\left|\E_{\scC}[ aY(t)] \right| = \lim_{n \to \infty}\left|\E_{\scC}[ aZ_n(t)] \right| \leq \norm{a}_{p'}\liminf_{n \to \infty}\norm{Z_n(t)}_p \leq R\norm{a}_{p'}
\]
by noncommutative H\"older's inequality.
Since $b \mapsto (a \mapsto \E_{\scC}[ab])$ is an isometric isomorphism from $L^p(\cC_t,\E_{\scC})$ onto $L^{p'}(\cC_t,\E_{\scC})^*$, it follows that $Y(t) \in L^p(\cC_t,\E_{\scC})$, as desired.
(Here, we used Assumption \ref{ass.integrator}\ref{item.W*ass} for the $p=\infty$ case.)

Next, we argue that $Y \in C_a([r,T];L^p(\E_{\scC})_{\beta})$.
Indeed, if $r \leq s < t \leq T$ and $a \in \cC_t$, then
\begin{align*}
    | \E_{\scC}[ a(Y(t) - Y(s))] | & = \lim_{n \to \infty} \left| \E_{\scC}[ a(Z_n(t) - Z_n(s))]\right| \\
    & \leq \norm{a}_{p'}\liminf_{n \to \infty} \norm{Z_n(t) - Z_n(s)}_p \\
    & = \norm{a}_{p'}\liminf_{n \to \infty} \norm{\int_s^t F(s_0,Z_{n-1}(s_0))[\d X(s_0)]}_p \\
    & \leq a_T\norm{a}_{p'}\liminf_{n \to \infty}\left(\int_s^t \norm{F(s_0,Z_{n-1}(s_0))}_{\G}^2 \,\rho(\d s_0)\right)^\frac12 \\
    & \leq a_TM_{T,R}\norm{a}_{p'}\rho([s,t])^{\frac12}.
\end{align*}
Thus, by duality again,
\[
\norm{Y(t)-Y(s)}_p \leq a_TM_{T,R}\rho([s,t])^{\frac12},
\]
from which the desired conclusion follows.
\pagebreak

It remains to show that \eqref{eq.SIEfromr} holds.
To this end, let $H(t) \coloneqq F(t,Y(t)) \in \cG_t$ and $H_n(t) \coloneqq F(t,Z_n(t)) \in \cG_t$ for all $t \in [r,T]$ and $n \in \N$.
Of course, $H \in \cI_{\G} \subseteq \cI_{\H}$ and $H_n \in \cI_{\G}([r,T]) \subseteq \cI_{\H}([r,T])$ for all $n \in \N$.
By definition of $Z_n$,
\begin{equation}
    Z_n = D + \int_r^{\boldsymbol{\cdot}} F(t,Z_{n-1}(t))[\d X(t)] = D+I_X^{\H}[1_{(r,T]}H_{n-1}] \; \text{ on } \; [r,T]. \label{eq.ZnSIE}
\end{equation}
Taking $n \to \infty$ on the left-hand side gives $Y$.
Also,
\begin{align*}
    \sup_{r \leq t \leq T} \norm{H_{n-1}(t) - H(t) }_{\H} & = \sup_{r \leq t \leq T}\norm{F(t,Z_{n-1}(t)) - F(t,Y(t))}_{\H} \\
    & \leq c_{T,R} \sup_{r \leq t \leq T}\norm{Z_{n-1}(t) - Y(t)}_q \xrightarrow{n \to \infty} 0.
\end{align*}
In particular, $(1_{(r,T]}H_{n-1})_{n \geq 2}$ converges locally boundedly to $1_{(r,T]}H$.
By the continuity of $I_X^{\H}$ with respect to locally bounded convergence, $\big(I_X^{\H}[1_{(r,T]}H_{n-1}]\big)_{n \geq 2}$ converges to $I_X^{\H}[1_{(r,T]}H]$ in $C_a([r,T];L^q(\E_{\scC})_{\beta})$.
Consequently, after taking $n \to \infty$ in \eqref{eq.ZnSIE}, we obtain that $Y = D+ I_X^{\H}[1_{(r,T]}H]$ on $[r,T]$.
Since $I_X^{\H}$ agrees with $I_X^{\G}$ on $\cI_{\G}$, we conclude that \eqref{eq.SIEfromr} holds.
This completes the proof.
\end{proof}

\begin{corollary}\label{cor.exist}
For any $C \in C_a(\R_+;L^p(\E_{\scC})_{\beta})$, there exists a local solution to \eqref{eq.SIE} defined on an interval of non-zero length.
\end{corollary}

\begin{proof}
Let $r = 0$ and $R = 2\norm{C(0)}_p+1$.
Since
\[
\lim_{T \searrow 0} a_T^2M_{T,R}^2\rho([0,T]) = 0 \; \text{ and } \;\ \lim_{T \searrow 0}\norm{C(T)}_p = \norm{C(0)}_p,
\]
there exists a $T > 0$ small enough that
\[
a_T^2M_{T,R}^2\rho([0,T]) \leq \frac{R^2}{4} \; \text{ and } \; \sup_{0 \leq t \leq T}\norm{C(t)}_p \leq \frac{R}{2}.
\]
By Theorem \ref{thm.exist}, there exists a $Y \in C_a([0,T];L^p(\E_{\scC})_{\beta})$ such that \eqref{eq.SIE} holds on $[0,T]$.
\end{proof}

Having now established the local existence of solutions, we use Gr\"onwall's inequality in the usual way to establish a uniqueness result.

\begin{proposition}[Uniqueness]\label{prop.uniqueness}
Suppose $0 \leq r \leq T$.
If $C,D,Y,Z \in C_a([r,T];L^p(\E_{\scC})_{\beta})$ are such that $Y = C+\int_r^{\boldsymbol{\cdot}} F(t,Y(t))[\d X(t)]$ and $Z = D+\int_r^{\boldsymbol{\cdot}} F(t,Z(t))[\d X(t)]$ on $[r,T]$, then
\[
\sup_{r \leq s \leq t} \norm{Y(s) - Z(s)}_q \leq \sqrt2\sup_{r \leq s \leq t}\norm{C(s)-D(s)}_qe^{2a_T^2c_{T,R}^2\rho([r,t])} \qquad (r \leq t \leq T),
\]
where
\[
R = \max\left\{ \sup_{r \leq t \leq T}\norm{Y(t)}_p, \sup_{r \leq t \leq T}\norm{Z(t)}_p\right\}.
\]
In particular, if $C\equiv D$, then $Y \equiv Z$.
\end{proposition}

\begin{proof}
If $r \leq t \leq T$, then
\begin{align*}
    \norm{Y(t) - Z(t)}_q^2 & \leq 2\norm{C(t)-D(t)}_q^2 + 2\norm{\int_r^t (F(s,Y(s)) - F(s,Z(s)))[\d X(s)]}_q^2 \\
    & \leq 2\norm{C(t)-D(t)}_q^2 + 2a_T^2\int_r^t\norm{F(s,Y(s)) - F(s,Z(s))}_{\H}^2\,\rho(\d s) \\
    & \leq 2\norm{C(t)-D(t)}_q^2 + 2a_T^2c_{T,R}^2\int_r^t\norm{Y(s) - Z(s)}_q^2\,\rho(\d s).
\end{align*}
Consequently, we obtain from Gr\"onwall's inequality (vid.\ \cite[Lem.\ 13]{Seifert2016}) and the elementary inequality $1+xe^x \leq e^{2x}$ ($x \geq 0$) that
\begin{align*}
    \norm{Y(t) - Z(t)}_q^2 & \leq 2\norm{C(t) - D(t)}_q^2 + 4a_T^2c_{T,R}^2\int_r^t \norm{C(s) - D(s)}_q^2e^{2a_T^2c_{T,R}^2\rho([s,t])} \,\rho(\d s) \\
    & \leq 2\sup_{r \leq s \leq t}\norm{C(s)-D(s)}_q^2\left(1 + 2a_T^2c_{T,R}^2\rho([r,t])e^{2a_T^2c_{T,R}^2\rho([r,t])} \right) \\
    & \leq 2\sup_{r \leq s \leq t}\norm{C(s)-D(s)}_q^2e^{4a_T^2c_{T,R}^2\rho([r,t])}
\end{align*}
whenever $r \leq t \leq T$.
Taking square roots on both sides yields the desired inequality.
\end{proof}

Combining the existence and uniqueness results we have proven thus far enables us to extend solutions that do not blow up in finite time.

\begin{theorem}[Extending solutions that do not blow up]\label{thm.extendbddsol}
Suppose $0 < T_0 < \infty$ and $Y_0 \in C_a( [0,T_0); L^p(\E_{\scC})_{\beta})$ is a local solution to \eqref{eq.SIE}.
If $\norm{Y_0(t)}_p \not\to \infty$ as $t \nearrow T_0$, then there exists a $T > T_0$ and a local solution $Y \in C_a([0,T]; L^p(\E_{\scC})_{\beta})$ to \eqref{eq.SIE} such that $Y|_{[0,T_0)} = Y_0$.
\end{theorem}

\begin{proof}
If $\norm{Y_0(t)}_p \not\to \infty$ as $t \nearrow T_0$, then there exists a sequence $(t_n)_{n \in \N}$ in $[0,T_0)$ such that $t_n \nearrow T_0$ as $n \to \infty$ and $\sup\big\{\norm{Y_0(t_n)}_p : n \in \N\big\} < \infty$.
Now, define
\[
M_1 \coloneqq \sup_{n \in \N}\norm{Y_0(t_n)}_p, \;\;  M_2 \coloneqq \sup_{0 \leq t \leq T_0+1}\norm{C(t)}_p, \; \text{ and } \; R \coloneqq 2(2M_2+M_1)+1.
\]
Also, choose a $T \in (T_0,T_0+1)$ and an $n \in \N$ such that
\[
a_T^2M_{T,R}^2\rho((t_n,T]) \leq \frac{R^2}{4}.
\]
This is possible because $t_n \nearrow T_0$ as $n \to \infty$ and
\[
\lim_{\underset{u \searrow T_0}{t \nearrow T_0}}a_u^2M_{u,R}^2 \rho([t,u]) = 0.
\]
Now, if $D \coloneqq C - C(t_n) + Y_0(t_n) \in C_a([t_n,\infty);L^p(\E_{\scC})_{\beta})$, then
\[
\sup_{t_n \leq t \leq T} \norm{D(t)}_p  \leq 2 \sup_{0 \leq t \leq T_0+1} \norm{C(t)}_p + \sup_{m \in \N} \norm{Y_0(t_m)}_p = 2M_2+M_1 < \frac{R}{2}.
\]
By Theorem \ref{thm.exist} with $r = t_n$, there exists a $\tilde{Y} \in C_a([t_n,T];L^p(\E_{\scC})_{\beta})$ such that
\[
\tilde{Y}(t) = D(t) + \int_{t_n}^t F\big(s,\tilde{Y}(s)\big)[\d X(s)] \qquad (t_n \leq t \leq T).\pagebreak
\]
We claim that $\tilde{Y}|_{[t_n,T_0)} = Y_0|_{[t_n,T_0)}$.
Indeed, if $t_n \leq t < T_0$, then
\begin{align*}
    Y_0(t) & = C(t) + \int_0^t F(s,Y_0(s))[\d X(s)] \\
    & = C(t) - C(t_n) + C(t_n) + \int_0^{t_n} F(s,Y_0(s))[\d X(s)] + \int_{t_n}^t F(s,Y_0(s))[\d X(s)] \\
    & = C(t) - C(t_n) + Y_0(t_n) + \int_{t_n}^t F(s,Y_0(s))[\d X(s)] \\
    & = D(t) + \int_{t_n}^t F(s,Y_0(s))[\d X(s)].
\end{align*}
In other words, $Y_0$ also solves $Z = D+\int_{t_n}^{\boldsymbol{\cdot}} F(t,Z(t))[\d X(t)]$ on $[t_n,T_0)$.
It follows from Proposition \ref{prop.uniqueness} that $Y_0 \equiv \tilde{Y}$ on $[t_n,T_0)$, as claimed.
In particular, if
\[
Y(t) \coloneqq \begin{cases}
    Y_0(t) & \text{if } 0 \leq t < T_0, \\
    \tilde{Y}(t) & \text{if } T_0 \leq t \leq T,
\end{cases}
\]
then $Y \in C_a([0,T];L^p(\E_{\scC})_{\beta})$.

We complete the proof by showing that $Y$ solves \eqref{eq.SIE} on $[0,T]$.
Since $Y|_{[0,T_0)} = Y_0$, we already know $Y$ solves \eqref{eq.SIE} on $[0,T_0)$.
Now, if $T_0 \leq t \leq T$, then
\begin{align*}
    Y(t) & = \tilde{Y}(t) = D(t) + \int_{t_n}^t F\big(s,\tilde{Y}(s)\big)[\d X(s)] \\
    & = C(t) - C(t_n) + Y_0(t_n) + \int_{t_n}^t F(s,Y(s))[\d X(s)] \\
    & = C(t) + \int_0^{t_n} F(s,Y_0(s))[\d X(s)] + \int_{t_n}^t F(s,Y(s))[\d X(s)] \\
    & = C(t) + \int_0^{t_n} F(s,Y(s))[\d X(s)] + \int_{t_n}^t F(s,Y(s))[\d X(s)] \\
    & = C(t) + \int_0^t F(s,Y(s))[\d X(s)],
\end{align*}
as claimed.
\end{proof}

We are finally ready to prove Theorem \ref{thm.SIE}.

\begin{proof}[Proof of Theorem \ref{thm.SIE}]
Let $J$ be the set of all $T \in (0,\infty]$ such that there exists a local solution $Y_T  \colon [0,T) \to L^p(\E_{\scC})_{\beta}$ to \eqref{eq.SIE}.
Note that if $T \in J$ and $0 < S \leq T$, then $S \in J$.
By Proposition \ref{prop.uniqueness}, if $S,T \in J$ and $S < T$, then $Y_S = Y_T|_{[0,S)}$.
By Corollary \ref{cor.exist}, $J$ is not empty.
Define $T_{\ell} \coloneqq \sup J \in (0,\infty]$ and $Y \colon [0,T_{\ell}) \to L^p(\E_{\scC})_{\beta}$ by $Y|_{[0,T)} \coloneqq Y_T$ for all $T \in [0,T_{\ell})$.
By the third sentence of this paragraph, $Y$ is well defined.
It is also clear that $Y \in C_a([0,T_{\ell});L^p(\E_{\scC})_{\beta})$.
We claim that $Y$ solves \eqref{eq.SIE} on $[0,T_{\ell})$.
Indeed, suppose $0 \leq t < T < T_{\ell}$.
Since $Y_T$ solves \eqref{eq.SIE} on $[0,T)$ and $Y|_{[0,t]} = Y_T|_{[0,t]}$, $Y$ solves \eqref{eq.SIE} on $[0,t]$.
Since $t \in [0, T_{\ell})$ was arbitrary, the claim is~proven.

Next, we show that $Y$ is maximal and unique.
Suppose $Z \colon I \to L^p(\E_{\scC})_{\beta}$ is any local solution to \eqref{eq.SIE}.
We begin by arguing that $I \subseteq [0,T_{\ell})$.
If $I = [0,\infty)$, then $T_{\ell} = \infty$, so $I = [0,T_{\ell})$ trivially.
If $I \subsetneq [0,\infty)$ and we define $T \coloneqq \sup I < \infty$, then $I = [0,T)$ or $I = [0,T]$.
If $I = [0,T)$, then $T \in J$, so $T \leq T_{\ell}$, in which case $I = [0,T) \subseteq [0,T_{\ell})$.
If $I = [0,T]$, then by Theorem \ref{thm.extendbddsol}, there exists an $\e > 0$ and a local solution $Z_1 \colon [0,T+\e] \to L^p(\E_{\scC})_{\beta}$ to \eqref{eq.SIE} such that $Z_1 = Z$ on $[0,T)$.
Consequently, $T+\e \in J$, so $T+\e \leq T_{\ell}$;
thus, $I = [0,T] \subseteq [0,T+\e) \subseteq [0,T_{\ell})$, as desired.
In either case, it follows at once from Proposition \ref{prop.uniqueness} that $Z = Y|_I$.

Finally, suppose $T_{\ell} < \infty$.
If it were the case that $\norm{Y(t)}_p \not\to \infty$ as $t \nearrow T_{\ell}$, then Theorem \ref{thm.extendbddsol} would say that $Y$ could be extended to a local solution on $[0,T]$ for some $T > T_{\ell}$, which would contradict the fact that $T_{\ell} = \sup J$.
We conclude that if $T_{\ell} < \infty$, then $\norm{Y(t)}_p \to \infty$ as $t \nearrow T_{\ell}$.
This completes the proof.\vspace{-0.75mm}
\end{proof}

\section{Examples}\label{sec.SIEpexamples}\vspace{-0.75mm}

Retain the objects and notation set at the beginning of section \ref{sec.SIEp}.\vspace{-0.75mm}

\subsection{Measured decomposable driver}\label{subsec.MDSIE}

In this subsection, we take:
\begin{itemize}
    \item $X = X(0) + M + A \colon \R_+ \to L^{p_0}(\E_{\scB})$ to be an $L^{p_0}$-measured decomposable process;
    \item $2 \leq q=p < \infty$;
    \item $(\H,\norm{\cdot}_{\H}) = (\G,\norm{\cdot}_{\G}) = (\F^{p_0;p},\norm{\cdot}_{p_0;p}) = \big(\big(\cF_t^{p_0;p}(\E_{\scB};\E_{\scC})\big)_{t \geq 0},\norm{\cdot}_{p_0;p}\big)$ (Example \ref{ex.LpLq});
    \item $\nu(\d t) \coloneqq \norm{\d A(t)}_{p_0}$, $\mu$ as in Lemma \ref{lem.control} with $(\alpha,F,\cV) = (2,M,L^{p_0}(\E_{\scB}))$, $\rho \coloneqq \mu+\nu$, and $a_T \coloneqq \beta_p+\nu([0,T])^{1/2}$ for all $T \geq 0$, where $\beta_p$ is as in Theorem \ref{thm.NCBDG}.
\end{itemize}
Then Assumption \ref{ass.integrator} is satisfied by definition and Theorem \ref{thm.MDass}.
Theorem \ref{thm.SIE}, Remark \ref{rem.ass}, and Corollary \ref{cor.sublingrowth} then yield the following result.

\begin{theorem}\label{thm.SIEMDdriver}
Suppose $F \colon \R_+ \times L_{\alg}^p(\cC_{\infty},\E_{\scC})_{\beta} \to B^{p_0;p}$ is continuous and satisfies:
\begin{enumerate}[font=\normalfont,label=(\roman*)]
    \item $F(t,y) \in \cF_t^{p_0;p}$ for all $t\geq 0$ and $y \in L^p(\cC_t,\E_{\scC})_{\beta}$;
    \item $\displaystyle \int_r^t F(s,Y(s))[\d X(s)] \in L^p(\cC_t,\E_{\scC})_{\beta}$ for all $t \geq r \geq 0$ and $Y \in C_a([r,t];L^p(\E_{\scC})_{\beta})$;
    \item for all $R > 0$, there exists an increasing family $(c_{T,R})_{T \geq 0}$ of positive real numbers such that for all $y,z \in L_{\alg}^p(\cC_{\infty},\E_{\scC})_{\beta}$ with $\norm{y}_p \leq R$ and $\norm{z}_p \leq R$,
    \[
    \norm{F(t,y) - F(t,z)}_{p_0;p} \leq c_{T,R}\norm{y-z}_p \qquad (0 \leq t \leq T).
    \]
\end{enumerate}
For each $C \in C_a(\R_+;L^p(\E_{\scC})_{\beta})$, there exists a unique maximal solution $Y \colon [0,T_{\ell}) \to L^p(\E_{\scC})_{\beta}$ to \eqref{eq.SIE} such that $\lim_{t \nearrow T_{\ell}}\norm{Y(t)}_p = \infty$ whenever $T_{\ell} < \infty$.
Furthermore, if for each $T \geq 0$, there exists an $M_T <\infty$ such that 
\[
\sup_{0 \leq t \leq T}\norm{F(t,y)}_{p_0;p} \leq M_T\left(\norm{y}_p+1\right) \qquad \big(y \in L_{\alg}^p(\cC_{\infty},\E_{\scC})_{\beta}\big),
\]
e.g., if $c_{T,R}$ may be taken to be independent of $R$, then $T_{\ell}=\infty$.\qed
\end{theorem}
\pagebreak

Let us examine two classes of examples of such coefficient functions $F$ when $p_0 = \infty$.

\begin{example}\label{ex.linpoly}
Let $(\cA,(\cA_t)_{t \geq 0},\E)$ be a conditionable filtered $\mathrm{C}^*$-probability space.
Take $(\cB,(\cB_t)_{t \geq 0},\E_{\scB})$ to be the direct sum $(\cA\oplus \C,(\cA_t\oplus \C)_{t \geq 0},\E_{\mathsmaller{\cA \oplus \C}})$ and $(\cC,(\cC_t)_{t \geq 0},\E_{\scC})$ to be $(\cA,(\cA_t)_{t \geq 0},\E)$.
Also, let $X \colon \R_+ \to \cA$ be an $L^{\infty}$-measured decomposable process, and define $Z(t) \coloneqq (X(t),t)$ for all $t \geq 0$.
Then $Z \colon \R_+ \to \cB$ is an $L^{\infty}$-measured decomposable process.
Finally, take $\beta = \emptyset$.

Consider a function $P \colon \R_+ \times L^p(\E) \times L^{\infty}(\E) \to L^p(\E)$ of the form
\begin{equation}
\begin{split}
    P(t,y,x) = \sum_{\e,\delta \in \{1,\ast\}} \big((&a_{\e\delta}(t)y^{\e}b_{\e\delta}(t) + c_{\e\delta}(t))x^{\delta}d_{\e\delta}(t) \\
    & + e_{\e\delta}(t)x^{\delta}(f_{\e\delta}(t)y^{\e}g_{\e\delta}(t) + h_{\e\delta}(t)) \big),
\end{split}\label{eq.P}
\end{equation}
where $a_{\e\delta},b_{\e\delta},c_{\e\delta},d_{\e\delta},e_{\e\delta},f_{\e\delta},g_{\e\delta},h_{\e\delta} \in C_a(\R_+;L^{\infty}(\E))$ ($\e,\delta \in \{1,\ast\}$).
Consider also a function $Q \colon \R_+ \times L^p(\E) \to L^p(\E)$ of the form
\begin{equation}
    Q(t,y) = \sum_{i=1}^k\left( a_i(t) y b_i(t) + c_i(t) y^* d_i(t) \right),\label{eq.Q}
\end{equation}
where $a_i,b_i,c_i,d_i \in C_a(\R_+;L^{\infty}(\E))$ ($i=1,\ldots,k$).
If $F \colon \R_+ \times L^p(\E) \to B^{\infty;p} = B^{\infty;p}(\E_{\scB};\E)$ is defined by
\[
F(t,y)[z] \coloneqq P(t,y,x) + Q(t,y)\lambda \qquad \big(t \geq 0, \,y \in L^p(\E), \,z = (x,\lambda) \in \cB\big),
\]
then $F$ is continuous, and $F(t,y) \in \cF_t^{\infty;p}$ for all $t \geq 0$ and $y \in L^p(\cA_t,\E_{\scC})$.
Furthermore,
\[
\sup_{0 \leq t \leq T} \norm{F(t,y)-F(t,z)}_{\infty;p} \leq c_T\norm{y-z}_p \qquad (T \geq 0, \; y,z \in L^p(\E)),
\]
where
\begin{align*}
    c_T = \sup_{0 \leq t \leq T}\left\{\sum_{\e,\delta \in \{1,\ast\}} \big(\|a_{\e\delta}\hspace{-1.25mm}\right.&\left.(t)\|_{\infty}\norm{b_{\e\delta}(t)}_{\infty}\norm{d_{\e\delta}(t)}_{\infty} + \norm{e_{\e\delta}(t)}_{\infty}\norm{f_{\e\delta}(t)}_{\infty}\norm{g_{\e\delta}(t)}_{\infty} \big) \right. \\
    & \left. + \sum_{i=1}^k\big(\norm{a_i(t)}_{\infty}\norm{b_i(t)}_{\infty} + \norm{c_i(t)}_{\infty}\norm{d_i(t)}_{\infty}\big) \right\}
\end{align*}
Consequently, if $C \in C_a(\R_+;L^p(\E))$, then Theorem \ref{thm.SIEMDdriver} tells us that the NCSIE
\[
Y = C + \into F(t,Y(t))[\d Z(t)] = C+\into P(t,Y(t),\d X(t)) + \into Q(t,Y(t))\,\d t
\]
has a unique global solution.
In particular, the NCSDE
\[
\begin{cases}
    \d Y(t) = P(t,Y(t),\d X(t)) + Q(t,Y(t))\,\d t & \\
    \;\, Y(0) = y_0
\end{cases}
\]
has a unique global solution for each initial condition $y_0 \in L^p(\cA_0,\E)$.

The example in the previous paragraph generalizes to ``multiple dimensions.''
Indeed, let $n,m \in \N$, and take $(\cB,(\cB_t)_{t \geq 0},\E_{\scB})$ to be the direct sum $(\cA^n \oplus \C,(\cA_t^n \oplus \C)_{t \geq 0},\E_{\mathsmaller{\cA^n \oplus \C}})$ and $(\cC,(\cC_t)_{t \geq 0},\E_{\scC})$ to be the direct sum $(\cA^m,(\cA_t^m)_{t \geq 0},\E_{\mathsmaller{\cA^m}})$.
Define
\[
\mathbf{P}(t,\mathbf{y},\mathbf{x}) \coloneqq \begin{bmatrix}
    \displaystyle \sum_{j=1}^m\sum_{k=1}^n P_{1jk}(t,y_j,x_k) \\
    \vdots \\
    \displaystyle \sum_{j=1}^m\sum_{k=1}^n P_{mjk}(t,y_j,x_k)
\end{bmatrix} \in L^p(\E)^m = L^p(\E_{\scC})
\]
for all $t \geq 0$, $\mathbf{y} = (y_1,\ldots,y_m) \in L^p(\E)^m$, and $\mathbf{x} = (x_1,\ldots,x_n) \in L^{\infty}(\E)^n$, where $P_{ijk}$ is a function of the form \eqref{eq.P} for each $i,j=1,\ldots,m$ and $k=1,\ldots,n$.
Also, define
\[
\mathbf{Q}(t,\mathbf{y}) \coloneqq \begin{bmatrix}
    \displaystyle \sum_{j=1}^mQ_{1j}(t,y_j) \\
    \vdots \\
    \displaystyle \sum_{j=1}^mQ_{mj}(t,y_j)
\end{bmatrix} \in L^p(\E)^m = L^p(\E_{\scC})
\]
for all $t \geq 0$ and $\mathbf{y} = (y_1,\ldots,y_m) \in L^p(\E)^m$, where $Q_{ij}$ is a function of the form \eqref{eq.Q} for each $i,j=1,\ldots,m$.
If $\mathbf{C} \in C_a(\R_+;L^p(\E)^m)$ and $\mathbf{X} \colon \R_+ \to \cA^n$ is an $L^{\infty}$-measured decomposable process, then the NCSIE
\[
\mathbf{Y} = \mathbf{C} + \into \mathbf{P}(t,\mathbf{Y}(t),\d \mathbf{X}(t)) + \into \mathbf{Q}(t,\mathbf{Y}(t))\,\d t
\]
has a unique global solution.
Writing $\mathbf{C} = (C_1,\ldots,C_m)$ and $\mathbf{X} = (X_1,\ldots,X_n)$, this means that the coupled system of NCSIEs
\[
Y_i = C_i + \sum_{j=1}^m\left(\sum_{k=1}^n \into P_{ijk}(t,Y_j(t),\d X_k(t)) + \into Q_{ij}(t,Y_j(t))\,\d t\right) \qquad (i=1,\ldots,m)
\]
has a unique global solution.
In particular, the coupled system of NCSDEs
\[
\begin{cases}
    \d Y_i(t) = \displaystyle\sum_{j=1}^m\left(\sum_{k=1}^n P_{ijk}(t,Y_j(t),\d X_k(t)) + Q_{ij}(t,Y_j(t))\,\d t\right) & \\
    \;\, Y_i(0) = y_{0,i}
\end{cases} \qquad (i=1,\ldots,m)
\]
has a unique global solution for each initial condition $\mathbf{y}_0 = (y_{0,1},\ldots,y_{0,m}) \in L^p(\cA_0,\E)^m$.
\end{example}

The second class of examples involves NCSIEs with coefficients arising from scalar functions via the continuous functional calculus;
please see \cite[Ch.\ VIII]{Conway1990} for the relevant background.
To treat such NCSIEs, we need the following (special case of an) important result of Potapov and Sukochev.

\begin{theorem}[Potapov--Sukochev \cite{PS2011}]\label{thm.funkycalcLpLip}
If $(\cA,\E)$ is a $\mathrm{W}^*$-probability space and $1 < p <  \infty$, then there exists a $c_p < \infty$ such that for all Lipschitz functions $f \colon \R \to \C$,
\[
\norm{f(a) - f(b)}_p \leq c_p[f]_{\mathrm{Lip}}\norm{a-b}_p \qquad (a,b \in \cA_{\sa}),
\]
where
\[
[f]_{\mathrm{Lip}} \coloneqq \sup\left\{ \frac{|f(\lambda) - f(\mu)|}{|\lambda-\mu|} : \lambda,\mu \in \R \text{ with } \lambda \neq \mu \right\}
\]
is the Lipschitz seminorm of $f$.
In particular, the function $\cA_{\sa} \ni a \mapsto f(a) \in \cA \hookrightarrow L^p(\E)$ extends uniquely to a Lipschitz function from $L^p(\E)_{\sa}$ to $L^p(\E)$, which we notate similarly.
\end{theorem}

We now use Theorem \ref{thm.funkycalcLpLip} to establish some preparatory results for the aforementioned second class of examples of NCSIEs.

\begin{definition}\label{def.Lip,bddspacelocuniftime}
Let $f \colon \R_+ \times \R \to \C$ be a function, and write $f_t \coloneqq f(t,\cdot) \colon \R \to \C$ for all $t \geq 0$.
If $\sup \left\{ [f_t]_{\mathrm{Lip}} : 0 \leq t \leq T\right\} < \infty$ for all $T \geq 0$, then $f$ is said to be \textbf{Lipschitz in space, locally uniformly in time}.
Now, write $\norm{g}_{\ell^{\infty}} \coloneqq \sup\left\{|g(x)| : x \in \R\right\}$ for a function $g \colon \R \to \C$.
If $\sup\left\{\norm{f_t}_{\ell^{\infty}} : 0 \leq t \leq T\right\} < \infty$ for all $T \geq 0$, then $f$ is said to be \textbf{bounded in space, locally uniformly in time}.
\end{definition}

\begin{proposition}\label{prop.funkycalcLpLip}
Let $(\cA,\E)$ be a $\mathrm{W}^*$-probability space and $f \colon \R_+ \times \R \to \C$ be a continuous function that is Lipschitz in space, locally uniformly in time.
Also, suppose $1 < p < \infty$.
If $f_{L^p(\E)} \colon \R_+ \times L^p(\E)_{\sa} \to L^p(\E)$ is the function $(t,a) \mapsto f_t(a)$, where $f_t(a)$ is defined as in {\rm Theorem \ref{thm.funkycalcLpLip}}, then $f_{L^p(\E)}$ is continuous, and $f(t,a) \coloneqq f_{L^p(\E)}(t,a) \in L^p(\cA_s,\E)$ whenever $s,t \geq 0$ and $a \in L^p(\cA_s,\E)_{\sa}$.
If, in addition, $f_t \colon \R \to \C$ is bounded for some $t \geq 0$, then $f(t,a) \in \cA$, and $\norm{f(t,a)}_{\infty} \leq \norm{f_t}_{\ell^{\infty}}$ for all $a \in L^p(\E)_{\sa}$.
\end{proposition}

\begin{proof}
Let $(t,a) \in \R_+ \times L^p(\E)_{\sa}$.
By definition, if $(a_n)_{n \in \N}$ is any sequence in $\cA_{\sa}$ converging in $L^p(\E)$ to $a$, then $f(t,a) = L^p\text{-}\lim_{n \to \infty} f(t,a_n)$, where $f(t,a_n) = f_t(a_n)$ ($n \in \N$) is defined via the continuous functional calculus in $\cA$.
If $f_t$ is bounded, then by basic properties of the functional calculus $\norm{f(t,a_n)}_{\infty} = \norm{f_t(a_n)}_{\infty} \leq \norm{f_t}_{\ell^{\infty}}$ for all $n \in \N$.
Since $f(t,a_n) \to f(t,a)$ in $L^p(\E)$ as $n \to \infty$, we conclude from basic properties of the noncommutative $L^p$ spaces (as in the fourth paragraph of the proof of Theorem \ref{thm.exist}) that $f(t,a) \in L^{\infty}(\E) = \cA$ and
\[
\norm{f(t,a)}_{\infty} \leq \liminf_{n \to \infty}\norm{f(t,a_n)}_{\infty} \leq \norm{f_t}_{\ell^{\infty}},
\]
as asserted in the last sentence of the statement.

Next, if $s \geq 0$ and $a \in L^p(\cA_s,\E)_{\sa}$, then there is a sequence $(a_n)_{n \in \N}$ in $(\cA_s)_{\sa}$ converging to $a$ in $L^p(\cA_s,\E) \subseteq L^p(\cA,\E)$.
In this case, $f_t(a_n) \in \cA_s \subseteq L^p(\cA_s,\E)$ for all $n \in \N$, so $f(t,a) = L^p\text{-}\lim_{n \to \infty}f_t(a_n) \in L^p(\cA_s,\E)$, as desired.

Finally, we prove that $f_{L^p(\E)} \colon \R_+ \times L^p(\E)_{\sa} \to L^p(\E)$ is continuous at $(t,a)$.
To begin, note that the function $\R_+ \times \cA_{\sa} \ni (s,c) \mapsto f(s,c) = f_s(c) \in \cA$ is continuous, as can be seen by locally uniformly approximating $f$ by polynomials in two variables (via the Stone--Weierstrass theorem).
Now, let $\e > 0$.
Choose a $c \in \cA_{\sa}$ such that
\[
c_p\norm{a-c}_p \sup_{0 \leq r \leq t+1}[f_r]_{\mathrm{Lip}} \leq \frac\e6,
\]
where $c_p$ is as in Theorem \ref{thm.funkycalcLpLip}.
Also, choose a $\delta > 0$ small enough that
\[
c_p\delta\sup_{0 \leq r \leq t+1}[f_r]_{\mathrm{Lip}} \leq \frac\e3
\]
and $s \geq 0$ and $|t-s| \leq \delta$ imply $s \leq t+1$ and $\norm{f(t,c) - f(s,c)}_{\infty} \leq \e/3$ (achievable by the second sentence of this paragraph).
By Theorem \ref{thm.funkycalcLpLip}, if $(s,b) \in \R_+ \times L^p(\E)_{\sa}$ and $\max\big\{|t-s|,\norm{a-b}_p\big\} \leq \delta$, then
\begin{align*}
    \norm{f(t,a) - f(s,b)}_p & \leq \norm{f(t,a) - f(t,c)}_p + \norm{f(t,c) - f(s,c)}_p + \norm{f(s,c) - f(s,b)}_p \\
    & \leq c_p[f_t]_{\mathrm{Lip}}\norm{a-c}_p + \norm{f(t,c) - f(s,c)}_{\infty} + c_p[f_s]_{\mathrm{Lip}}\norm{c-b}_p \\
    & \leq c_p\left(2\norm{a-c}_p + \norm{a-b}_p\right)\sup_{0 \leq r \leq t+1}[f_r]_{\mathrm{Lip}} + \frac\e3 \\
    & \leq c_p\left(2\norm{a-c}_p + \delta\right)\sup_{0 \leq r \leq t+1}[f_r]_{\mathrm{Lip}} + \frac\e3 \leq \frac\e3 + \frac\e3 + \frac\e3 = \e.
\end{align*}
Thus, $f_{L^p(\E)}$ is continuous at $(t,a)$.
This completes the proof.
\end{proof}

\begin{proposition}\label{prop.funkycoeffLp}
Let $(\cA,\E)$ be a $\mathrm{W}^*$-probability space and $f,g \colon \R_+ \times \R \to \C$ be continuous functions that are bounded and Lipschitz in space, locally uniformly in time.
Also, suppose $1 < p < \infty$.
If
\[
F(t,y)[x] \coloneqq f(t,y)\,x\,g(t,y) \in L^p(\E) \qquad \big(t \geq 0, \, y \in L^p(\E)_{\sa}, \, x \in L^{\infty}(\E) = \cA\big),
\]
then $F \colon \R_+ \times L^p(\E)_{\sa} \to B^{\infty;p}$ is continuous,
\[
\norm{F(t,y) - F(t,z)}_{\infty;p} \leq c_p\left( [f_t]_{\mathrm{Lip}}\norm{g_t}_{\ell^{\infty}} + [g_t]_{\mathrm{Lip}}\norm{f_t}_{\ell^{\infty}}\right)\norm{y-z}_p \quad (t \geq 0, \; y,z \in L^p(\E)_{\sa}),
\]
and $F(t,y) \in  \cT_t^{\infty;p} \subseteq \cF_t^{\infty;p}$ for all $t \geq 0$ and $y \in L^p(\cA_t,\E)_{\sa}$.
\end{proposition}

\begin{proof}
Suppose $s,t \geq 0$ and $y,z \in L^p(\E)_{\sa}$.
If $x \in \cA$, then
\[
F(t,y)[x] - F(s,z)[x] = (f(t,y) - f(s,z))\,x\,g(t,y) + f(s,z)\,x\,(g(t,y) - g(s,z)).
\]
Consequently, by Proposition \ref{prop.funkycalcLpLip},
\begin{align*}
    \norm{F(t,y)  - F(s,z)}_{\infty;p} & \leq \norm{f(t,y) - f(s,z)}_p \norm{g(t,y)}_{\infty} + \norm{g(t,y) - g(s,z)}_p \norm{f(s,z)}_{\infty} \\
    & \leq \norm{f(t,y) - f(s,z)}_p \norm{g_t}_{\ell^{\infty}} + \norm{g(t,y) - g(s,z)}_p \norm{f_s}_{\ell^{\infty}}.
\end{align*}
Since $f$ is bounded in space, locally uniformly in time, $\norm{f_s}_{\ell^{\infty}}$ is bounded in $s$ as $s \to t$.
Thus, since Proposition \ref{prop.funkycalcLpLip} says that $f_{L^p(\E)}$ and $g_{L^p(\E)}$ are continuous, we conclude that $F(s,z) \to F(t,y)$ in $B^{p;\infty}$ as $(s,z) \to (t,y)$ in $\R_+ \times L^p(\E)_{\sa}$.
Also, if we take $s=t$ above and apply Theorem \ref{thm.funkycalcLpLip}, then we obtain
\begin{align*}
    \norm{F(t,y)  - F(s,z)}_{\infty;p} & \leq \norm{f(t,y) - f(t,z)}_p \norm{g_t}_{\ell^{\infty}} + \norm{g(t,y) - g(t,z)}_p \norm{f_t}_{\ell^{\infty}} \\
    & \leq c_p\left( [f_t]_{\mathrm{Lip}}\norm{g_t}_{\ell^{\infty}} + [g_t]_{\mathrm{Lip}}\norm{f_t}_{\ell^{\infty}}\right)\norm{y-z}_p,
\end{align*}
as asserted in the statement.
Finally, that $F(t,y) \in \cT_t^{\infty;p}$ whenever $y \in L^p(\cA_t,\E)_{\sa}$ is immediate from the fact that $f(t,y), g(t,y) \in \cA_t$ whenever $y \in L^p(\cA_t,\E)_{\sa}$.
\end{proof}

Finally, here is the aforementioned second class of NCSIEs with $p_0 = \infty$.

\begin{example}\label{ex.funkycalc}
Let $(\cA,(\cA_t)_{t \geq 0},\E)$ be a filtered $\mathrm{W}^*$-probability space.
Take $(\cB,(\cB_t)_{t \geq 0},\E_{\scB})$ to be the direct sum $(\cA \oplus \C, (\cA_t \oplus \C)_{t \geq 0}, \E_{\mathsmaller{\cA\oplus \C}})$, $(\cC,(\cC_t)_{t \geq 0},\E_{\scC}) = (\cA,(\cA_t)_{t \geq 0},\E)$, and $\beta = \sa$.
Also, we return now to the standing assumption that $2 \leq p < \infty$, unlike in the statements of Theorem \ref{thm.funkycalcLpLip} and Propositions \ref{prop.funkycalcLpLip} and \ref{prop.funkycoeffLp}.

Suppose $f_i,g_i \colon \R_+ \times \R \to \C$ ($i=1,\ldots,k$) and $h \colon \R_+ \times \R \to \R$ are continuous functions that are bounded and Lipschitz in space, locally uniformly in time.
Define
\begin{align*}
    F_0(t,y)[x] & \coloneqq \sum_{i=1}^k \left(f_i(t,y)\,x\,g_i(t,y) + \overline{g_i}(t,y)\,x^*\,\overline{f_i}(t,y)\right) \\
    & = \sum_{i=1}^k \left(f_i(t,y)\,x\,g_i(t,y) + g_i(t,y)^*\,x^*\, f_i(t,y)^*\right) \\
    F(t,y)[z] & \coloneqq F_0(t,y)[x] + h(t,y) \cRe \lambda
\end{align*}
for all $t \geq 0$, $y \in L^p(\E)_{\sa}$, $x \in L^p(\E)$, and $z = (x,\lambda) \in L^p(\E) \times \C$.
By Theorem \ref{thm.funkycalcLpLip} and Propositions \ref{prop.funkycalcLpLip} and \ref{prop.funkycoeffLp}, $F \colon \R_+ \times L^p(\E)_{\sa} \to B^{\infty;p} = B^{\infty;p}(\E_{\scB};\E)$ is continuous, $F(t,y) \in \cF_t^{\infty;p}(\E_{\scB};\E_{\scA})$ for all $t \geq 0$ and $y \in L^p(\cA_t,\E_{\scA})_{\sa}$, and
\[
\sup_{0 \leq t \leq T} \norm{F(t,y)-F(t,z)}_{\infty;p} \leq c_T\norm{y-z}_p \qquad (T \geq 0, \; y,z \in L^p(\E)_{\sa}),
\]
where
\[
c_T = c_p\sup_{0 \leq t \leq T}\left\{2\sum_{i=1}^k \left([(f_i)_t]_{\mathrm{Lip}}\norm{(g_i)_t}_{\ell^{\infty}}+ [(g_i)_t]_{\mathrm{Lip}}\norm{(f_i)_t}_{\ell^{\infty}}\right)  + [h_t]_{\mathrm{Lip}}\right\} \qquad (T \geq 0).
\]
Also,
\[
F(t,y)[z]^* = F(t,y)[z] \qquad \big(t \geq 0, \; y \in L^p(\E)_{\sa}, \; z \in \cB\big).
\]
Consequently, by Theorem \ref{thm.SIEMDdriver}, if $X \colon \R_+ \to \cA$ is an $L^{\infty}$-measured decomposable process, $Z(t) \coloneqq (X(t),t)$ for all $t \geq 0$, and $C \in C_a(\R_+;L^p(\E_{\scA})_{\sa})$, then the NCSIE
\begin{align*}
    & Y = C + \into F(t,Y(t))[\d Z(t)] \\
    & = C + \sum_{i=1}^k \into \left(f_i(t,Y(t))\d X(t)g_i(t,Y(t)) + \overline{g_i}(t,Y(t))\d X^*(t)\overline{f_i}(t,Y(t))\right) + \into h(t,Y(t))\,\d t
\end{align*}
has a unique global solution.
In particular, the NCSDE
\[
\begin{cases}
    \displaystyle\d Y(t) = \sum_{i=1}^k \left( f_i(t,Y(t))\,\d X(t)\,g_i(t,Y(t)) + \overline{g_i}(t,Y(t))\,\d X^*(t)\,\overline{f_i}(t,Y(t)) \right) + h(t,Y(t))\,\d t & \\
    \;\, Y(0) = y_0
\end{cases}
\]
has a unique global solution for each initial condition $y_0 \in L^p(\cA_0,\E)_{\sa}$.
Taking $X$ to be a $q$-Brownian motion with $-1 \leq q < 1$, this is Corollary \ref{cor.funkycalcqSDE}.

As in Example \ref{ex.linpoly}, the NCSIEs in the previous paragraph have ``multidimensional'' generalizations.
We leave the details to the reader.
\end{example}

\subsection{Free Brownian driver}\label{subsec.SCBMSIE}

For the duration of this subsection, let $(\cA,(\cA_t)_{t \geq 0},\E)$ be a $\mathrm{W}^*$-probability space, $m,n \in \N$, and $\mathbf{X} \colon \R_+ \to \cA_{\sa}^n$ be an $n$-dimensional free Brownian motion.

\begin{theorem}\label{thm.SIESCBMdriver}
Suppose $\mathbf{F} \colon\R_+ \times \cA_{\beta}^m \to B^{2;2} = B^{2;2}(\E^{\oplus n};\E^{\oplus m})$ and $\mathbf{G} \colon \R_+ \times \cA_{\beta}^m \to \cA_{\beta}^m$ are continuous and satisfy:
\begin{enumerate}[font=\normalfont,label=(\roman*)]
    \item $\mathbf{F}(t,\mathbf{y}) \in (\cT_t^{2;2})^{m \times n}$ and $\mathbf{G}(t,\mathbf{y}) \in (\cA_t)_{\beta}^m$ for all $t\geq 0$ and $\mathbf{y} \in (\cA_t)_{\beta}^m$;\label{item.boldFGadapted}
    \item $\displaystyle \int_r^t \mathbf{F}(s,\mathbf{Y}(s))[\d \mathbf{X}(s)] \in (\cA_t)_{\beta}^m$ for all $t \geq r \geq 0$ and $\mathbf{Y} \in C_a([r,t];\cA_{\beta}^m)$;\label{item.intFbeta}
    \item for all $R > 0$, there exists an increasing family $(c_{T,R})_{T \geq 0}$ of positive real numbers such that for all $\mathbf{y},\mathbf{z} \in \cA_{\beta}^m$ with $\norm{\mathbf{y}}_{\infty} \leq R$ and $\norm{\mathbf{z}}_{\infty} \leq R$,\label{item.boldFGLip}
    \[
    \norm{\mathbf{F}(t,\mathbf{y}) - \mathbf{F}(t,\mathbf{z})}_{\infty;2}+\norm{\mathbf{G}(t,\mathbf{y})-\mathbf{G}(t,\mathbf{z})}_2 \leq c_{T,R}\norm{\mathbf{y} - \mathbf{z}}_2 \qquad (0 \leq t \leq T);
    \]
    \item for all $T,R \geq 0$,\label{item.boldFGbdd}
    \[
    \sup\left\{\norm{\mathbf{F}(t,\mathbf{y})}_{2;2} + \norm{\mathbf{G}(t,\mathbf{y})}_{\infty} : 0 \leq t \leq T, \, \mathbf{y} \in \cA_{\beta}^m, \, \norm{\mathbf{y}}_{\infty} \leq R  \right\} < \infty.
    \]
\end{enumerate}
For each $\mathbf{C} \in C_a(\R_+;\cA_{\beta}^m)$, there exists a unique maximal solution $\mathbf{Y} \colon [0,T_{\ell}) \to \cA_{\beta}^m$ to the~NCSIE
\[
\mathbf{Y} = \mathbf{C} + \into \mathbf{F}(t,\mathbf{Y}(t))[\d \mathbf{X}(t)] + \into \mathbf{G}(t,\mathbf{Y}(t))\,\d t.
\]
In addition, if $T_{\ell} < \infty$, then $\norm{\mathbf{Y}(t)}_{\infty} \to \infty$ as $t \nearrow T_{\ell}$.
Finally, if for each $T \geq 0$, there exists an $M_T <\infty$ such that 
\[
\sup_{0 \leq t \leq T}\left\{\norm{\mathbf{F}(t,\mathbf{y})}_{2;2}+\norm{\mathbf{G}(t,\mathbf{y})}_{\infty}\right\} \leq M_T\left(\norm{\mathbf{y}}_{\infty}+1\right) \qquad \big(\mathbf{y} \in \cA_{\beta}^m\big),
\]
then $T_{\ell}=\infty$.
\end{theorem}

\begin{proof}[Sketch of proof]
We apply Theorem \ref{thm.SIE} and Corollary \ref{cor.sublingrowth}.
Therein, take $(\cB,(\cB_t)_{t \geq 0},\E_{\scB})$ equal to $(\cA^n\oplus \C, (\cA_t^n\oplus\C)_{t \geq 0},\E_{\mathsmaller{\cA^n\oplus \C}})$, $(\cC,(\cC_t)_{t \geq 0},\E_{\scC})$ equal to $(\cA^m, (\cA_t^m)_{t \geq 0},\E_{\mathsmaller{\cA^m}})$, $p_0 = \infty$, $q=2$, $p=\infty$, $X(t) \coloneqq (\mathbf{X}(t),t) \in (\cB_t)_{\sa}$ for all $t \geq 0$, and 
\[
(\H,\norm{\cdot}_{\H}) \coloneqq \big(\big(\cF_t^{\infty;2}(\E_{\scB};\E_{\scC})\big)_{t \geq 0},\norm{\cdot}_{\infty;2}\big).
\]
Also, for each $t \geq 0$, take $\cG_t$ to be the set of $T \colon L^2(\E_{\scB}) \to L^2(\E_{\scC})$ of the form
\[
T[\mathbf{z}] = \mathbf{H}[\mathbf{x}] + \mathbf{a}\lambda \in L^2(\E_{\scC}) = L^2(\E)^n \qquad \big(\mathbf{z} = (\mathbf{x},\lambda) \in L^2(\E_{\scB}) = L^2(\E)^n\oplus \C\big),
\]
where $\mathbf{H} \in (\cT_t^{2;2})^{m \times n}$ and $\mathbf{a} \in \cA_t^m$.
For such $T$, define $\norm{T}_{\G} \coloneqq \norm{\mathbf{H}}_{2;2}+\norm{\mathbf{a}}_{\infty}$.

After observing from Example \ref{ex.qBm} that $X$ is an $L^{\infty}$-measured decomposable process, it follows from Theorems \ref{thm.MDass} and \ref{thm.Linfbound}, the usual triangle inequality for vector-valued integrals, and the Cauchy--Schwarz inequality that Assumption \ref{ass.integrator} is satisfied with $\rho$ equal to the Lebesgue measure and $a_T = 4\sqrt{2n}+\sqrt{T}$ for all $T \geq 0$.

Finally, if $F \colon \R_+ \times L_{\alg}^{\infty}(\E_{\scB})_{\beta} \to \cG_{\infty}$ is defined by
\[
F(t,\mathbf{y})[\mathbf{z}] \coloneqq \mathbf{F}(t,\mathbf{y})[\mathbf{x}] + \mathbf{G}(t,\mathbf{y})\lambda \qquad  \big(t \geq0, \, \mathbf{y} \in L_{\alg}^{\infty}(\E_{\scB})_{\beta} , \,\mathbf{z} = (\mathbf{x},\lambda) \in L^2(\E_{\scB})\big),
\]
then it is simply a matter of unraveling notation to see that $F$ satisfies Assumption \ref{ass.coefficient}.
\end{proof}

\begin{example}\label{ex.trpoly}
Let $k \in \N_0$, let
\[
\mathbf{P}(\mathbf{z},\mathbf{y},\mathbf{x}) = \begin{bmatrix}
    P_1(z_1,\ldots,z_k,y_1,\ldots,y_m,x_1,\ldots,x_n) \\
    \vdots \\
    P_m(z_1,\ldots,z_k,y_1,\ldots,y_m,x_1,\ldots,x_n)
\end{bmatrix}
\]
be an $m$-tuple of trace $\ast$-polynomials in $k+m+n$ indeterminates that are real linear in $\mathbf{x}$ (vid.\ \cite[\S2.3]{JKN2026}), and let
\[
\mathbf{Q}(\mathbf{z},\mathbf{y}) = \begin{bmatrix}
    Q_1(z_1,\ldots,z_k,y_1,\ldots,y_m) \\
    \vdots \\
    Q_m(z_1,\ldots,z_k,y_1,\ldots,y_m)
\end{bmatrix}
\]
be an $m$-tuple of trace $\ast$-polynomials in $k+m$ indeterminates.
For $\mathbf{Z} \in C_a(\R_+;\cA^k)$, define
\[
\mathbf{F}(t,\mathbf{y}) \coloneqq \mathbf{P}(\mathbf{Z}(t),\mathbf{y},\cdot) \text{ and } \mathbf{G}(t,\mathbf{y}) \coloneqq \mathbf{Q}(\mathbf{Z}(t),\mathbf{y}) \qquad \big(t \geq 0, \, \mathbf{y} \in \cA^m\big).
\]
We leave it to the reader to check that $\mathbf{F}$ and $\mathbf{G}$ satisfy the hypotheses of Theorem \ref{thm.SIESCBMdriver}.
Consequently, if $\mathbf{C} \in C_a(\R_+;\cA^m)$, then there exists a unique maximal solution to the~NCSIE
\[
\mathbf{Y} = \mathbf{C} + \into \mathbf{F}(t,\mathbf{Y}(t))[\d \mathbf{X}(t)] + \into \mathbf{G}(t,\mathbf{Y}(t))\,\d t.
\]
If the $\mathbf{y}$-degree of $\mathbf{P}$ and $\mathbf{Q}$ is at most one, then the maximal solution is global;
otherwise, the $L^{\infty}$ norm of the maximal solution may blow up in finite time.
(Consider the Riccati ODE $\dot{y} = y^2$.)
This vastly generalizes and improves Example \ref{ex.linpoly} when the driver is free Brownian motion.
\end{example}

Next, we consider NCSIEs with functional-calculus coefficients as in Example \ref{ex.funkycalc}.
This time, however, we are able to establish a substantially more general result.

\begin{definition}\label{def.locLiploctim}
For $R \geq 0$ and a function $g \colon \R \to \C$, write
\[
[g]_{\mathrm{Lip},R} \coloneqq \sup\left\{\frac{|g(\lambda) - g(\mu)|}{|\lambda - \mu|} :  \lambda,\mu \in [-R,R], \, \lambda \neq \mu \right\} \in [0,\infty].
\]
A function $f \colon \R_+ \times \R \to \C$ is \textbf{locally Lipschitz in space, locally uniformly in time} if $\sup \left\{ [f_t]_{\mathrm{Lip},R} : 0 \leq t \leq T\right\} < \infty$ for all $R,T \geq 0$.
\end{definition}

\begin{lemma}\label{lem.funkycalccont}
If $f \colon \R_+ \times \R \to \C$ is continuous, then the function $f_{\scA} \colon \R_+ \times \cA_{\sa} \to \cA$ defined by $(t,a) \mapsto f(t,a) \coloneqq f_t(a)$ is continuous.
Furthermore, $f(t,a) \in \cA_s$ for all $s,t \geq 0$ and $a \in (\cA_{\sa})_s$.
\end{lemma}

\begin{proof}
By the Stone--Weierstrass theorem, there exists a sequence $(p_n)_{n \in \N}$ of complex polynomials in two variables converging locally uniformly on $\R_+ \times \R$ to $f$.
Since both conclusions are obvious for $(p_n)_{\scA}$ and $(p_n)_{\scA} \to f_{\scA}$ uniformly on bounded sets as $n \to \infty$ by basic properties of the continuous functional calculus, we obtain both conclusions for~$f_{\scA}$.
\end{proof}

\begin{example}\label{ex.funkycalcSCBMSIE}
Assume $m=n=1$, and write $X \coloneqq \mathbf{X}$.
Suppose $f_i,g_i \colon \R_+ \times \R \to \C$ ($i=1,\ldots,k$) and $h \colon \R_+ \times \R \to \R$ are continuous functions that are locally Lipschitz in space, locally uniformly in time.
Define
\[
F(t,y)[x] \coloneqq \sum_{i=1}^k \left(f_i(t,y) \, x\, g_i(t,y) +\overline{g_i}(t,y) \, x^* \, \overline{f_i}(t,y)\right) \in L^2(\E)\pagebreak
\]
for all $t \geq 0$, $y \in \cA_{\sa}$, and $x \in L^2(\E)$.
By Lemma \ref{lem.funkycalccont} and the definitions, the functions $\mathbf{F} \coloneqq F \colon \R_+ \times \cA_{\sa} \to B^{2;2} = B^{2;2}(\E;\E)$ and $\mathbf{G} \coloneqq G \coloneqq h_{\scA} \colon \R_+ \times \cA_{\sa} \to \cA_{\sa}$ are continuous and satisfy hypotheses \ref{item.boldFGadapted} and \ref{item.intFbeta} of Theorem~\ref{thm.SIESCBMdriver}.
Since $\norm{f(a)}_{\infty} = \norm{f}_{\ell^{\infty}(\sigma(a))}$ for all continuous $f \colon \R \to \C$ and all $a \in \cA_{\sa}$, we have that
\begin{align*}
    \sup &\left\{\|F(t,y)\|_{2;2} + \norm{G(t,y)}_{\infty} : 0 \leq t \leq T, \; y \in \cA_{\sa}, \; \norm{y}_{\infty} \leq R\right\} \\
    & \leq 2\sum_{i=1}^k\norm{f_i}_{\ell^{\infty}([0,T] \times [-R,R])}\norm{g_i}_{\ell^{\infty}([0,T] \times [-R,R])} + \norm{h}_{\ell^{\infty}([0,T] \times [-R,R])} < \infty
\end{align*}
for all $R,T \geq 0$, which verifies hypothesis \ref{item.boldFGbdd}.
To verify hypothesis \ref{item.boldFGLip}, let $R \geq 0$, and suppose $\varphi_R \colon \R \to \R$ is a compactly supported smooth function such that $\varphi_R \equiv 1$ on $[-R,R]$.
If $f^R(t,\lambda) \coloneqq \varphi_R(\lambda) \, f(t,\lambda)$ for all $t \geq 0$, $\lambda \in \R$, and $f \in \{f_1,\ldots,f_k,g_1,\ldots,g_k,h\}$, then $f_i^R$, $g_i^R$, and $h^R$ are Lipschitz and bounded in space, locally uniformly in time.
Furthermore, $f_i^R = f_i$, $g_i^R =g_i$, and $h^R = h$ on $\R_+ \times [-R,R]$.
Now, define $F_R$ and $G_R$ in the same way as $F$ and $G$ were defined, except with $f_i^R$ in place of $f_i$, $g_i^R$ in place of $g_i$, and $h^R$ in place of $h$.
By Proposition \ref{prop.funkycoeffLp} and Theorem \ref{thm.funkycalcLpLip}, if $0 \leq t \leq T$ and $y,z \in \cA_{\sa}$ are such that $\norm{y}_{\infty} \leq R$ and $\norm{z}_{\infty} \leq R$, then
\begin{align*}
    \|F(t&,y)- F(t,z)\|_{\infty;2} + \norm{G(t,y) - G(t,z)}_2\\
    & = \norm{F_R(t,y) - F_R(t,z)}_{\infty;2} + \norm{G_R(t,y) - G_R(t,z)}_2 \leq c_{T,R}\norm{y-z}_2,
\end{align*}
where
\[
c_{T,R} = c_2\sup_{0 \leq t \leq T}\left\{2\sum_{i=1}^k\big(\big[\big(f_i^R\big)_t\big]_{\mathrm{Lip}}\norm{\big(g_i^R\big)_t}_{\ell^{\infty}} + \big[\big(g_i^R\big)_t\big]_{\mathrm{Lip}}\norm{\big(f_i^R\big)_t}_{\ell^{\infty}} \big) + \big[h_t^R\big]_{\mathrm{Lip}}  \right\}.
\]
This verifies hypothesis \ref{item.boldFGLip}.
Consequently, if $C \in C_a(\R_+;\cA_{\sa})$, then there exists a unique maximal solution to the NCSIE
\begin{equation}\label{eq.funkycalcSCBM}
\begin{split}
    Y = C + \sum_{i=1}^k \into &\big(f_i(t,Y(t))\,\d X(t)\,g_i(t,Y(t)) \\
    & + \overline{g_i}(t,Y(t))\,\d X(t)\,\overline{f_i}(t,Y(t))\big) + \into h(t,Y(t))\,\d t.
\end{split}
\end{equation}
In general, this maximal solution may blow up in finite time.
However, assume, in addition, that for each $T \geq 0$, there exists an $M_T < \infty$ such that
\[
2\sum_{i=1}^k\norm{f_i}_{\ell^{\infty}([0,T] \times [-R,R])}\norm{g_i}_{\ell^{\infty}([0,T] \times [-R,R])} + \norm{h}_{\ell^{\infty}([0,T] \times [-R,R])} \leq M_T (R+1) \quad (R \geq 0).
\]
If $y \in \cA_{\sa}$, then
\begin{align*}
    \sup_{0 \leq t \leq T}\left\{\norm{F(t,y)}_{2;2}+\norm{G(t,y)}_{\infty}\right\} & \leq \sup_{0 \leq t \leq T}\left\{2\sum_{i=1}^k \norm{f_i(t,y)}_{\infty}\norm{g_i(t,y)}_{\infty}+\norm{h(t,y)}_{\infty}\right\} \\
    & \leq M_T\left(\norm{y}_{\infty}+1\right).
\end{align*}
Thus, the maximal solution to \eqref{eq.funkycalcSCBM} is global in this case.
This proves Corollary \ref{cor.funkycalcFSDE}.
\end{example}

\subsection{Classical Brownian driver}\label{subsec.clBMSIE}

We wrap up this section by explaining briefly how one of the standard results on classical SDEs driven by Brownian motion follows from Theorem \ref{thm.SIE}.
Let $(\Om,\sF,(\sF_t)_{t \geq 0},P)$ be a complete filtered probability space and $B \colon \R_+ \times \Om \to \R$ be a standard real Brownian motion.
Take
\begin{itemize}
    \item $(\cC,(\cC_t)_{t \geq 0},\E_{\scC}) = (L^{\infty}(\Om,\sF,P),(L^{\infty}(\Om,\sF_t,P))_{t \geq 0},\E_P)$ and $(\cB,(\cB_t)_{t \geq 0},\E_{\scB})$ equal to the direct sum of $(\cC,(\cC_t)_{t \geq 0},\E_{\scC})$ and $(\C,(\C)_{t \geq 0},\id_{\C})$;
    \item $2 \leq q=p=p_0 < \infty$;
    \item $X(t) = (B_t,t) \in L^{p_0}(\cB_t,\E_{\scB})$ for all $t \geq 0$;
    \item $(\H,\norm{\cdot}_{\H})$ and $(\G,\norm{\cdot}_{\G})$ equal to $(\M,\norm{\cdot}_{\infty;p}) = ((\cM_t)_{t \geq 0},\norm{\cdot}_{\infty;p})$, where $\cM_t$ is the set of all $M \in B^{p;1} = B^{p;1}(\E_{\scB};\E_{\scC})$ of the form
    \[
    Mz = M_yx+ \lambda c = yx+\lambda c \qquad \big(z=(x,\lambda) \in L^p(\E_P) \times \C = L^p(\E_{\scB})\big),
    \]
    for $y \in L^p(\Om,\sF_t,P)$ and $c \in \C$;
    \item $\rho$ equal to the Lebesgue measure on $\R_+$ and $a_T \coloneqq \delta_p+T^{1/2}$ for all $T \geq 0$, where $\delta_p$ is as in the classical BDG inequalities (Theorem \ref{thm.BDG}).
\end{itemize}
It follows from the discussion after Example \ref{ex.Doleanmeasac} that Assumption \ref{ass.integrator} is satisfied in this~case.

Now, suppose $f,g \colon \R_+ \times \R \to \R$ are continuous functions that are Lipschitz in space, locally uniformly in time.
Define $F \colon \R_+ \times L_{\alg}^p(\cC_{\infty},\E_{\scC})_{\sa} \to \cM_{\infty}$ by
\[
F(t,y)[z] \coloneqq (f_t \circ y) \,x +(g_t \circ y)\,\lambda = f(t,y)\,x + g(t,y)\,\lambda \qquad \big( z = (x,\lambda) \in L^p(\E_{\scB})\big).
\]
Note that $f_t \circ y = f(t,y), g_t \circ y = g(t,y) \in L^p(\E_P)$ whenever $y \in L^p(\E_P)$ because $f_t$ and $g_t$ are Lipschitz and thus grow sublinearly.
Assumption \ref{ass.coefficient} is satisfied for this $F$ with $c_{T,R}$ independent of $R$.
Thus, by Theorem \ref{thm.SIE} and Corollary \ref{cor.sublingrowth}, if $C \in C_a(\R_+;L^p(\E_{\scC})_{\sa})$, then
\[
Y = C + \into F(t,Y(t))[\d X(t)]
\]
has a unique global solution.
Taking $C \equiv y_0 \in L^p(\cC_0,\E_{\scC})_{\sa}$, this means that there exists a unique $Y \in C_a(\R_+;L^p(\E_{\scC})_{\sa})$ satisfying $\d Y(t) = F(t,Y(t))[\d X(t)]$ and $Y(0) = y_0$, i.e.,
\begin{equation}
    Y(t) = y_0 + \int_0^t F(s,Y(s))[\d X(s)] \qquad (t \geq 0).\label{eq.ncItoSDE}
\end{equation}
By classical It\^o NCSDE theory, there exists a unique-up-to-indistinguishability continuous adapted process $Z \colon \R_+ \times \Om \to \R$ such that $\||Z|_t^*\|_p < \infty$ for all $t \geq 0$ and
\[
Z_t = y_0 + \int_0^t f(s,Z_s)\,\d B_s + \int_0^t g(s,Z_s)\,\d s\qquad (t \geq 0),
\]
where $\into \boldsymbol{\cdot}\,\d B_s$ denotes the classical stochastic integral against Brownian motion.
It follows from the discussion after Example \ref{ex.Doleanmeasac} and the definition of $X$ that if $U,V \colon \R_+ \times \Om \to \C$ are continuous adapted processes satisfying
\[
\norm{|U|_t^*}_p+\norm{|V|_t^*}_p < \infty \qquad (t \geq 0),\pagebreak
\]then
\[
\int_0^t U_s\,\d B_s + \int_0^t V_s\,\d s = \int_0^t H(s)[\d X(s)] \qquad (t \geq 0),
\]
where $H(s)[z] = U_sx + V_s\lambda$ for $z = (x,\lambda) \in L^p(\E_{\scC})$ and $s \geq 0$.
Consequently, the ($L^p$-continuous adapted) noncommutative process
\[
\R_+ \ni t  \mapsto Y(t) \coloneqq Z_t \in L^p(\E_{\scC})_{\sa}
\]
satisfies \eqref{eq.ncItoSDE}.
Thus, our NCSDE result produces the modification classes of the solutions to classical It\^o SDEs.

The same reasoning but more complicated notation yields the conclusion that our NCSDE result produces the modification classes of the solutions to multidimensional classical It\^o SDEs of the form
\begin{equation}\label{eq.ItoSDE}
    \begin{cases}
    \d \mathbf{Y}_t = \mathbf{F}(t,\mathbf{Y}_t)\, \d \mathbf{B}_t +  \mathbf{G}(t,\mathbf{Y}_t)\,\d t & \\
    \;\,\mathbf{Y}_0 = \mathbf{y}_0, &
\end{cases}
\end{equation}
where $\mathbf{B} \colon \R_+ \times \Om \to \R^n$ is a Brownian motion in $\R^n$, and $\mathbf{F} \colon \R_+ \times \R^m \to \mathrm{M}_{m \times n}(\R)$ and $\mathbf{G} \colon \R_+ \times \R^m \to \R^m$ are continuous functions that are Lipschitz in space, locally uniformly in time.
We leave the details to the interested reader.

It is worth noting that the classical result on It\^o SDEs deduced above from our NCSIE result is only the global version of the classical result even though our NCSIE result covers local existence as well.
The reason is that the local-Lipschitz assumption in our NCSIE result is fundamentally different from that in the classical setting, and for that matter, so is the notion of local existence.
Indeed, in the notation of the previous paragraph, the standard assumption on $\mathbf{F}$ and $\mathbf{G}$ guaranteeing the ``local'' existence and uniqueness of solutions to \eqref{eq.ItoSDE} is:
For all $T,R > 0$, there exists a constant $c_{T,R} < \infty$ such that for all $\mathbf{y},\mathbf{z} \in \R^m$ with $\norm{\mathbf{y}}_{\R^m} \leq R$ and $\norm{\mathbf{z}}_{\R^m} \leq R$,
\[
\norm{\mathbf{F}(t,\mathbf{y}) - \mathbf{F}(t,\mathbf{z})}_{\mathrm{M}_{m \times n}(\R)} + \norm{\mathbf{G}(t,\mathbf{y}) - \mathbf{G}(t,\mathbf{z})}_{\R^m} \leq c_{T,R}\norm{\mathbf{y}-\mathbf{z}}_{\R^m} \qquad (0 \leq t \leq T).
\]
Under this assumption, a solution to \eqref{eq.ItoSDE} exists \emph{up to a random time}---a stopping time, to be precise.
Furthermore, the essential infimum of this stopping time can be zero, in which case \eqref{eq.ItoSDE} does not admit a solution up to any strictly positive deterministic time.
The reader might be concerned that this contradicts our NCSIE result.
It does not, precisely because the classical local-Lipschitz assumption above does \emph{not} generally imply the local-Lipschitz assumption in our NCSIE result.
The ``local'' assumption in the NCSIE result enforces estimates on balls in $L^p$ while the (classical) local-Lipschitz assumption above only guarantees estimates on balls in $L^{\infty}$ (and classical Brownian motion does not satisfy Assumption \ref{ass.integrator} with $p=\infty$).
We encourage the reader to think through the details.

There \emph{is} a notion of a noncommutative stopping time, introduced in \cite{BGW1996}.
Consequently, it may be possible to formulate a fruitful notion of existence up to a noncommutative stopping time for a solution to an NCSIE.
It would be interesting to know if there is a version of our result (Theorem \ref{thm.SIE}), at least for a subclass of the NCSIEs we consider, that establishes the existence-up-to-a-noncommutative-stopping-time of solutions to NCSIEs that captures the classical result described in the previous paragraph as well as Theorems \ref{thm.SDESCBMdriver} and \ref{thm.SDEqBM}.
Such a result would require the formulation of a different kind of local-Lipschitz assumption.

\appendix
\section{Proof of Theorem \ref{thm.multidimBSineq}}\label{app.multidimBSineq}

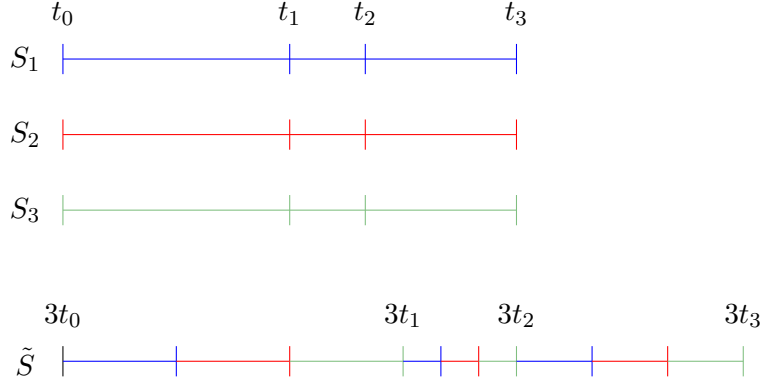
\begin{figure}
    \centering
    \begin{tikzpicture}
        \node at (-0.5,-1) {$S_1$};
        \node at (-0.5,-2) {$S_2$};
        \node at (-0.5,-3) {$S_3$};
        
        \node at (0,-0.4) {$t_0$};
        \node at (3,-0.4) {$t_1$};
        \node at (4,-0.4) {$t_2$};
        \node at (6,-0.4) {$t_3$};
        
        \draw[blue] (0,-1) -- (6,-1);
        \draw[blue] (0,-0.8) -- (0,-1.2);
        \draw[blue] (3,-0.8) -- (3,-1.2);
        \draw[blue] (4,-0.8) -- (4,-1.2);
        \draw[blue] (6,-0.8) -- (6,-1.2);
        \draw[red] (0,-2) -- (6,-2);
        \draw[red] (0,-1.8) -- (0,-2.2);
        \draw[red] (3,-1.8) -- (3,-2.2);
        \draw[red] (4,-1.8) -- (4,-2.2);
        \draw[red] (6,-1.8) -- (6,-2.2);
        \draw[green!50!black!50] (0,-3) -- (6,-3);
        \draw[green!50!black!50] (0,-2.8) -- (0,-3.2);
        \draw[green!50!black!50] (3,-2.8) -- (3,-3.2);
        \draw[green!50!black!50] (4,-2.8) -- (4,-3.2);
        \draw[green!50!black!50] (6,-2.8) -- (6,-3.2);

        \node at (-0.5,-5) {$\tilde{S}$};

        \node at (0,-4.4) {$3t_0$};
        \node at (4.5,-4.4) {$3t_1$};
        \node at (6,-4.4) {$3t_2$};
        \node at (9,-4.4) {$3t_3$};

        \draw[black] (0,-4.8) -- (0,-5.2);
        \draw[blue] (0,-5) -- (1.5,-5);
        \draw[blue] (1.5,-4.8) -- (1.5,-5.2);
        \draw[red] (1.5,-5) -- (3,-5);
        \draw[red] (3,-4.8) -- (3,-5.2);        
        \draw[green!50!black!50] (3,-5) -- (4.5,-5);
        \draw[green!50!black!50] (4.5,-4.8) -- (4.5,-5.2);
        \draw[blue] (4.5,-5) -- (5,-5);
        \draw[blue] (5,-4.8) -- (5,-5.2);
        \draw[red] (5,-5) -- (5.5,-5);
        \draw[red] (5.5,-4.8) -- (5.5,-5.2);
        \draw[green!50!black!50] (5.5,-5) -- (6,-5);
        \draw[green!50!black!50] (6,-4.8) -- (6,-5.2);
        \draw[blue] (6,-5) -- (7,-5);
        \draw[blue] (7,-4.8) -- (7,-5.2);
        \draw[red] (7,-5) -- (8,-5);
        \draw[red] (8,-4.8) -- (8,-5.2);
        \draw[green!50!black!50] (8,-5) -- (9,-5);
        \draw[green!50!black!50] (9,-4.8) -- (9,-5.2);
        
    \end{tikzpicture}
    \caption{Schematic representation of transformation from $\mathbf{S}$ to $\tilde{S}$ for $n = 3$ and $\ell = 3$}
    \label{fig:interval.rearrangement}
\end{figure}

Write $\mathbf{S} \coloneqq (S_1,\ldots,S_n)$.
Assume $U_1,\ldots,U_n \colon \R_+ \to \cA \otimes \cA^{\op}$ are all supported in a common interval $[0,T]$, and fix a partition $0 = t_0 < t_1 < \dots < t_\ell = T$ such that $U_j$ is constant on $(t_{k-1},t_k]$ for each $j=1,\ldots,n$ and $k=1,\ldots,\ell$, which is clearly possible by taking a common refinement.
In particular,
\begin{align*}
    U_j & = 1_{\{0\}} u_{j,0} +  \sum_{k=1}^{\ell} 1_{(t_{k-1},t_k]} u_{j,k}, \; \text{ where} \\
    u_{j,k} & \in \cA_{t_{k-1}} \otimes \cA_{t_{k-1}}^{\op} \quad (j=1,\dots,n; \;k=0,\ldots,\ell).
\end{align*}
In order to apply the one-variable inequality, we will create $n$ copies of each sub-interval $(t_{k-1},t_k]$, each associated to one of the coordinates of $\d S_j(t)$ on the interval $(t_{k-1},t_k]$ and arranging these intervals as depicted in Figure \ref{fig:interval.rearrangement}, which results in a one-dimensional free Brownian motion $\tilde{S}$ on $[0,nT]$.
For convenience of notation, let
\[
\delta_k \coloneqq t_k - t_{k-1} \qquad (k=1,\ldots,\ell).
\]
Define for $k = 1, \ldots, \ell$, $j = 1, \ldots, n$, and $s \in (0,\delta_k]$,
\[
\tilde{S}(nt_{k-1} + (j-1)\delta_k + s) \coloneqq \sum_{i=1}^{j-1} S_i(t_k) + S_j(t_{k-1} + s) + \sum_{i=j+1}^n S_i(t_{k-1}).
\]
We invite the reader to check that the left and right limiting values of $\tilde{S}$ agree at the gluing points $nt_{k-1} + j \delta_k$ for $k = 1$, \dots, $\ell$ and $j = 0$, \dots, $n$, so that $\tilde{S}$ is continuous.  Now $\tilde{S}$ is constructed so that for $r,s \in (0,\delta_k]$,
\[
\tilde{S}(nt_{k-1} + (j-1)\delta_k + r) - \tilde{S}(nt_{k-1} + (j-1)\delta_k + s) = S_j(t_{k-1} + r) - S_j(t_{k-1} + s),
\]
or in other words, the increments of $\tilde{S}$ on $(nt_{k-1} + (j-1) \delta_k,nt_{k-1} + j \delta_k]$ correspond to the increments of $S_j$ on $(t_{k-1},t_k]$.
Of course, define $\tilde{S}(0) \coloneqq 0$ as well.
Now, define a filtration $(\tilde{\cA}_t)_{t \in [0,nT]}$ as follows:
\[
\text{For } t \in (nt_{k-1},nt_k], \quad \tilde{\cA}_t \coloneqq \mathrm{C}^*(\cA_{t_{k-1}}, (\tilde{S}(s))_{s \leq t}).
\]
Then one can check that $\tilde{S}$ is a free Brownian motion on $[0,nT]$ with respect to $(\tilde{\cA}_t)_{t \in [0,nT]}$.
For instance, to demonstrate that $\tilde{S}(t) - \tilde{S}(s)$ is freely independent of $\tilde{\cA}_s$ for $s < t$, we can argue as follows.
Fix $k$ and $j$ so that $s \in (nt_{k-1} + (j-1)\delta_k, nt_{k-1} + j \delta_k]$, and write $s = nt_{k-1} + (j-1)\delta_k + r$ where $r \in (0,\delta_k]$.
The properties of $\mathbf{S}$ imply that the following list of $\mathrm{C}^*$-subalgebras is freely independent:
\begin{align*}
&\mathcal{A}_{t_{k-1}}, \\
&\mathrm{C}^*(S_i(u) - S_i(t_{k-1}): u \in (t_{k-1},t_k], i = 1, \dots, j - 1), \\
&\mathrm{C}^*(S_j(u) - S_j(t_{k-1}): u \in (t_{k-1},t_{k-1}+r]), \\
&\mathrm{C}^*(S_j(u) - S_j(t_{k-1} +r): u \in (t_{k-1}+r,t_k]), \\
&\mathrm{C}^*(S_i(u) - S_i(t_{k-1}): u \in (t_{k-1},t_k], i = j+1, \dots, n) \\
&\mathrm{C}^*(S_i(u) - S_i(t_k): u \in (t_k,T], i = 1, \dots, n).
\end{align*}
Translating this in terms of $\tilde{\cA}_s$ and $\tilde{S}$, we obtain that the following list of $\mathrm{C}^*$-subalgebras is freely independent:
\begin{align*}
&\tilde{\cA}_{nt_{k-1}}, \\
&\mathrm{C}^*(\tilde{S}(u) - \tilde{S}(nt_{k-1})): u \in (nt_{k-1},nt_{k-1}+(j-1) \delta_k]), \\
&\mathrm{C}^*(\tilde{S}(u) - \tilde{S}(nt_{k-1}+(j-1)\delta_k): u \in (nt_{k-1} + (j-1) \delta_k,s]), \\
&\mathrm{C}^*(\tilde{S}(u) - \tilde{S}(s): u \in (s,nt_{k-1}+j\delta_k]), \\
&\mathrm{C}^*(\tilde{S}(u) - \tilde{S}(nt_{k-1} + j \delta_k): u \in (nt_{k-1}+j\delta_k,nt_k]) \\
&\mathrm{C}^*(\tilde{S}(u) - \tilde{S}(nt_k): u \in (nt_k,nT]).
\end{align*}
Now, $\tilde{\cA}_s$ is generated by the first three algebras in the list, and $\tilde{S}(t) - \tilde{S}(s)$ is in the $\mathrm{C}^*$-subalgebra generated by the last three; therefore, they are freely independent.
The other properties of the Brownian motion are similarly checked by direct casework using the free independence of the algebras listed above.

We then define a biprocess $\tilde{U}$ on $[0,nT]$ by
\[
\tilde{U} \coloneqq \sum_{j=1}^n \mathbf{1}_{\{0\}} u_{j,0} + \sum_{k=1}^\ell \sum_{j=1}^n \mathbf{1}_{(nt_{k-1}+(j-1)\delta_k,nt_{k-1}+j\delta_k]} u_{j,k},
\]
or in other words, the values of $\tilde{U}$ on $(nt_{k-1} + (j-1) \delta_k,nt_{k-1} + j \delta_k]$ correspond to the values of $U_j$ on $(t_{k-1},t_k]$.
Thus, by construction,
\begin{align*}
\sum_{j=1}^n \int_0^T U_j(t) \# \d S_j(t) &= \sum_{j=1}^n \sum_{k=1}^\ell \int_{t_{k-1}}^{t_k} U_j(t) \# \d S_j(t) \\
&= \sum_{k=1}^\ell \sum_{j=1}^n \int_{nt_{k-1} + (j-1) \delta_k}^{nt_{k-1} + j \delta_k} \tilde{U}(t) \# \d \tilde{S}(t) \\
&= \int_0^{nT} \tilde{U}(t) \# \d \tilde{S}(t).
\end{align*}
Now we can apply the single-variable result \cite[Thm.\ 3.2.1]{BS1998}, i.e., \eqref{eq.BSelembip} above, to $\tilde{U}$ and $\tilde{S}$ to conclude that
\begin{align*}
\norm{ \sum_{j=1}^n \int_0^T U_j(t) \# \d S_j(t) }_\infty & = \norm{\int_0^{nT} \tilde{U}(t) \# \d \tilde{S}(t)}_{\infty} \\
&\leq 2 \sqrt{2} \left( \int_0^{nT} \big\|\tilde{U}(t)\big\|_{\cA \otimes_{\min} \cA^{\op}}^2\, \d t \right)^\frac12 \\
&= 2 \sqrt{2} \left( \sum_{k=1}^\ell \sum_{j=1}^n \int_{nt_{k-1} + (j-1) \delta_k}^{nt_{k-1} + j \delta_k} \big\|\tilde{U}(t)\big\|_{\cA \otimes_{\min} \cA^{\op}}^2\,\d t \right)^\frac12 \\
&= 2 \sqrt{2} \left( \sum_{j=1}^n \sum_{k=1}^\ell \int_{t_{k-1}}^{t_k} \norm{U_j(t)}_{\cA \otimes_{\min} \cA^{\op}}^2\, \d t \right)^\frac12 \\
&= 2 \sqrt{2} \left( \int_0^T \sum_{j=1}^n \norm{U_j(t)}_{\cA \otimes_{\min} \cA^{\op}}^2 \,\d t \right)^\frac12,
\end{align*}
which proves the assertion of the theorem. \qed

\begin{acknowledgments}
\phantomsection
\addcontentsline{toc}{section}{Acknowledgments}
D.\ A.\ Jekel acknowledges support from NSF grant DMS-2002826 (during the initial conception of the project) and the EU Horizon Marie Sk{\l}odowska Curie Action, FREEINFOGEOM, grant 101209517 (when the paper was written).
T.\ A.\ Kemp acknowledges support from NSF grants DMS-2400246, DMS-2055340, and DMS-1800733.
E.\ A.\ Nikitopoulos acknowledges support from NSF grant DGE-2038238 and partial support from NSF grant DMS-2055340.

E.\ A.\ Nikitopoulos is grateful to Guillaume C\'ebron and Roland Speicher for stimulating conversations and their hospitality when they hosted him at their respective institutions, the Institut de Math\'ematiques de Toulouse and the Universit\"at des Saarlandes, in 2023.

All the creative research, formulation, proof development, and draft writing of this work were accomplished by the three authors without assistance from any AI tools.  We acknowledge the use of ChatGPT and Gemini only for proofreading at various stages of the drafting of the manuscript.  
\end{acknowledgments}

\phantomsection
\small
\addcontentsline{toc}{section}{References}
\bibliographystyle{amsplain}
\bibliography{NCSDE.bib}
\end{document}